\documentclass[11pt]{article}
\usepackage{amsmath,amsthm,amssymb}
\usepackage[usenames,dvipsnames]{xcolor}
\usepackage{enumerate}
\usepackage{graphicx}
\usepackage{cite}
\usepackage{comment}
\usepackage{oands}
\usepackage{tikz}
\usepackage{changepage}
\usepackage{bbm}
\usepackage{mathtools}
\usepackage[margin=1in]{geometry}
\usepackage[pagewise,mathlines]{lineno}
\usepackage{appendix}
\usepackage{stmaryrd} 
\usepackage{multicol}
\usepackage{microtype}
\usepackage[colorinlistoftodos]{todonotes}

\usepackage[normalem]{ulem}

\usepackage[pdftitle={Weak LQG metrics and Liouville first passage percolation},
  pdfauthor={Julien Dubedat, Hugo Falconet, Ewain Gwynne, Joshua Pfeffer, Xin Sun},
colorlinks=true,linkcolor=NavyBlue,urlcolor=RoyalBlue,citecolor=PineGreen,bookmarks=true,bookmarksopen=true,bookmarksopenlevel=2,unicode=true,linktocpage]{hyperref}

\theoremstyle{plain}
\newtheorem{thm}{Theorem}[section]

\newtheorem{lem}[thm]{Lemma}
\newtheorem{prop}[thm]{Proposition}

\def\@rst #1 #2other{#1}
\newcommand\MR[1]{\relax\ifhmode\unskip\spacefactor3000 \space\fi
\MRhref{\expandafter\@rst #1 other}{#1}}
\newcommand{\MRhref}[2]{\href{http://www.ams.org/mathscinet-getitem?mr=#1}{MR#2}}

\theoremstyle{definition}
\newtheorem{defn}[thm]{Definition}
\newtheorem{remark}[thm]{Remark}

\numberwithin{equation}{section}

\newcommand{\dsb}{\begin{adjustwidth}{2.5em}{0pt}
\begin{footnotesize}}
\newcommand{\dse}{\end{footnotesize}
\end{adjustwidth}}

\newcommand{\ssb}{\begin{adjustwidth}{2.5em}{0pt}}
\newcommand{\sse}{\end{adjustwidth}}

\newcommand{\aryb}{\begin{eqnarray*}}
\newcommand{\arye}{\end{eqnarray*}}
\def\alb#1\ale{\begin{align*}#1\end{align*}}
\def\allb#1\alle{\begin{align}#1\end{align}}
\newcommand{\eqb}{\begin{equation}}
\newcommand{\eqe}{\end{equation}}
\newcommand{\eqbn}{\begin{equation*}}
\newcommand{\eqen}{\end{equation*}}

\newcommand{\BB}{\mathbbm}
\newcommand{\ol}{\overline}

\newcommand{\op}{\operatorname}

\newcommand{\frk}{\mathfrak}
\newcommand{\eqD}{\overset{d}{=}}
\newcommand{\ep}{\varepsilon}
\newcommand{\rta}{\rightarrow}
\newcommand{\xrta}{\xrightarrow}

\newcommand{\wt}{\widetilde}
\newcommand{\wh}{\widehat} 
\newcommand{\mcl}{\mathcal}

\newcommand{\bdy}{\partial}
\newcommand{\rng}{\mathring}

\let\originalleft\left
\let\originalright\right
\renewcommand{\left}{\mathopen{}\mathclose\bgroup\originalleft}
\renewcommand{\right}{\aftergroup\egroup\originalright}

\title{Weak LQG metrics and Liouville first passage percolation
}
\author{Julien Dub\'edat, Hugo Falconet, Ewain Gwynne, Joshua Pfeffer, and Xin Sun
}

\date{   }

\begin{document}

\maketitle 

\begin{abstract}
For $\gamma \in (0,2)$, we define a \emph{weak $\gamma$-Liouville quantum gravity (LQG) metric} to be a function $h\mapsto D_h$ which takes in an instance of the planar Gaussian free field (GFF) and outputs a metric on the plane satisfying a certain list of natural axioms. 
We show that these axioms are satisfied for any subsequential limits of Liouville first passage percolation. Such subsequential limits were proven to exist by Ding-Dub\'edat-Dunlap-Falconet (2019).  
It is also known that these axioms are satisfied for the $\sqrt{8/3}$-LQG metric constructed by Miller and Sheffield (2013-2016). 

For any weak $\gamma$-LQG metric, we obtain moment bounds for diameters of sets as well as point-to-point, set-to-set, and point-to-set distances.
We also show that any such metric is locally bi-H\"older continuous with respect to the Euclidean metric and compute the optimal H\"older exponents in both directions.
Finally, we show that LQG geodesics cannot spend a long time near a straight line or the boundary of a metric ball.  These results are used in subsequent work by Gwynne and Miller which proves that the weak $\gamma$-LQG metric is unique for each $\gamma \in (0,2)$, which in turn gives the uniqueness of the subsequential limit of Liouville first passage percolation. However, most of our results are new even in the special case when $\gamma=\sqrt{8/3}$.  
\end{abstract}

\tableofcontents

\section{Introduction}
\label{sec-intro}

 \subsection{Overview}
\label{sec-overview}

Let $\gamma \in (0,2)$, let $U\subset \BB C$ be open, and let $h$ be some variant of the Gaussian free field (GFF) on $U$. 
The \emph{$\gamma$-Liouville quantum gravity (LQG)} surface corresponding to $(U,h)$ is, heuristically speaking, the random two-dimensional Riemannian manifold with metric tensor $e^{\gamma h} (dx^2+dy^2)$, where $dx^2+dy^2$ denotes the Euclidean metric tensor. LQG surfaces are the scaling limits of various types of random planar maps: the case when $\gamma=\sqrt{8/3}$ corresponds to uniform random planar maps. Other values of $\gamma$ correspond to random planar maps weighted by the partition function of a statistical mechanics model on the map, e.g., the uniform spanning tree for $\gamma=\sqrt 2$ or the critical Ising model for $\gamma=\sqrt 3$. More generally, convergence to $\gamma$-LQG is expected if the planar map is weighted by the partition function of a critical statistical mechanics model with central charge $\mathbf{c} = 25 - 6(2/\gamma + \gamma/2)^2$; see, e.g.,~\cite[Section 3.1]{ghs-mating-survey} and the references therein for further discussion. 

The above definition of a $\gamma$-LQG surface does not make rigorous sense since the GFF is a random distribution, not a function. In particular, it does not have well-defined pointwise values and so cannot be exponentiated. 
Therefore, one needs to use various regularization procedures to make rigorous sense of LQG surfaces.
For example, one can construct a random measure $\mu_h$ on $U$, called the \emph{$\gamma$-LQG area measure}, as a limit of regularized versions of ``$e^{\gamma h} dz$", where $dz$ denotes Lebesgue measure~\cite{kahane,shef-kpz,rhodes-vargas-review}. 
This measure can be thought of as the volume form associated with the $\gamma$-LQG surface.
One way to construct $\mu_h$ is as follows. 
Let $p_s(z,w) =   \frac{1}{2\pi s} \exp\left( - \frac{|z-w|^2}{2s} \right) $ be the heat kernel on $\BB C$. For $\ep >0$, we define a mollified version of the GFF by
\eqb \label{eqn-gff-convolve}
h_\ep^*(z) := (h*p_{\ep^2/2})(z) = \int_{U} h(w) p_{\ep^2/2}^U(z,w) \, dw ,\quad \forall z\in U  ,
\eqe
where the integral is interpreted in the sense of distributional pairing (see Remark~\ref{rmk:tight} for some discussion on the particular choice of mollifier). 
One can then define the $\gamma$-LQG measure $\mu_h$ as the a.s.\ weak limit~\cite{rhodes-vargas-review,berestycki-gmt-elementary}
\eqb \label{eqn-measure-construct}
\lim_{\ep\rta 0} \ep^{\gamma^2/2} e^{\gamma h_\ep^*(z)} \,dz .
\eqe   

The LQG measure $\mu_h$ satisfies a conformal coordinate change formula: if $\phi : \wt U \rta U$ is a conformal map and
\eqb \label{eqn-lqg-coord}
\wt h := h\circ\phi + Q \log |\phi'| ,\quad \text{where} \quad Q = \frac{2}{\gamma} + \frac{\gamma}{2} 
\eqe
then $\mu_h(A) = \mu_{\wt h}(\phi^{-1}(A))$ for each Borel set $A\subset U$. 
We think of two pairs $(U,h)$ and $(\wt U , \wt h)$ which are related by a conformal map as in~\eqref{eqn-lqg-coord} as being two different parametrizations of the same LQG surface. Thus the coordinate change formula for $\mu_h$ says that this measure depends only  on the quantum surface, not on the particular choice of parametrization. 

Since $\gamma$-LQG surfaces are thought of as random Riemannian manifolds, one expects that such a surface also gives rise to a random metric $D_h$ on $U$.
Constructing such a metric is a much harder problem than constructing the measure $\mu_h$. 
Miller and Sheffield~\cite{lqg-tbm1,lqg-tbm2,lqg-tbm3} constructed such a metric in the special case when $\gamma=\sqrt{8/3}$ by using a process called \emph{quantum Loewner evolution}~\cite{qle} to define $\sqrt{8/3}$-LQG metric balls. 
They also showed that in this case, the metric space $(U,D_h)$ for certain special choices of $U$ and $D_h$ is isometric to a known \emph{Brownian surface} --- like the Brownian map~\cite{legall-uniqueness,miermont-brownian-map} or the Brownian disk~\cite{bet-mier-disk}. Brownian surfaces are random metric spaces which arise as the scaling limits of uniform random planar maps with respect to the Gromov-Hausdorff topology.

This paper is part of a program whose eventual goal is to construct a metric on $\gamma$-LQG for all $\gamma \in (0,2)$ as a limit of regularized metrics analogous to~\eqref{eqn-measure-construct}. These regularized metrics are called \emph{Liouville first passage percolation (LFPP)}. We recall the precise definition of LFPP just below. It was previously shown by Ding, Dub\'edat, Dunlap, and Falconet~\cite{dddf-lfpp} that LFPP admits non-degenerate subsequential limits in law w.r.t.\ the local uniform topology (i.e., the topology of uniform convergence on compact sets). The main contributions of this paper are as follows.
\begin{itemize}
\item \textbf{Properties of subsequential limits of LFPP.} We prove, using a general theorem from~\cite{local-metrics}, that every subsequential limit of LFPP can be realized as a measurable function of the field, so the convergence occurs in probability, not just in distribution. We also check that every subsequential limit of LFPP satisfies a certain natural list of axioms which one would expect any reasonable notion of a metric on $\gamma$-LQG to satisfy (see Section~\ref{sec-axioms}).
We call a metric satisfying these axioms a \emph{weak LQG metric}. A closely related list of axioms appeared previously in~\cite{mq-geodesics}.
\item \textbf{Properties of weak LQG metrics.} We prove several quantitative properties for a general weak LQG metric. We compute the optimal H\"older exponents between the LQG metric and the Euclidean metric in both directions. We also give moment bounds for LQG diameters and for point-to-point, set-to-set, and point-to-set distances; these bounds are analogous to known moment bounds for the $\gamma$-LQG measure (see, e.g.,~\cite{rhodes-vargas-review}). 
See Section~\ref{sec-main-results} for precise statements. Since our list of axioms is satisfied for the Miller-Sheffield $\sqrt{8/3}$-LQG metric, our results apply to this metric as well. Even in this special case, most of our results are new. 
\end{itemize}

The results in this paper are used to prove further properties of weak LQG metrics (including subsequential limits of LFPP) in~\cite{gm-uniqueness,gm-confluence,gm-coord-change}, eventually culminating in the proof in~\cite{gm-uniqueness} that there is only one weak $\gamma$-LQG metric for each $\gamma \in (0,2)$, which establishes the existence and uniqueness of the $\gamma$-LQG metric for all $\gamma \in (0,2)$.
However, even after this program is completed, we expect that our results will continue to be a useful tool in the study of the $\gamma$-LQG metric. 
For example, our estimates for the LQG metric are used in~\cite{gp-kpz} to prove a version of the KPZ formula~\cite{shef-kpz,kpz-scaling} for this metric.
Moreover, as explained in Remark~\ref{rmk:tight}, our results for subsequential limits of LFPP apply to variants of LFPP defined using different continuous approximations for the GFF (other than convolution with the heat kernel) once tightness is established for these variants.

We remark that versions of some of the estimates for weak LQG metrics which are proven in this paper (including tail estimates for the distance across a rectangle, the first moment bound for diameters, and H\"older continuity) were previously proven for subsequential limits of LFPP in~\cite{dddf-lfpp}. However, it is important to have these estimates for general weak $\gamma$-LQG metrics: indeed, such estimates will be used in~\cite{gm-uniqueness} to show the uniqueness of the weak $\gamma$-LQG metric (which is a stronger statement than just the uniqueness of the subsequential limit for the variant of LFPP considered in~\cite{dddf-lfpp}). Many of our estimates are also new for subsequential limits of LFPP, e.g., the optimality of the H\"older exponents in Theorem~\ref{thm-optimal-holder}, the moment bounds in Theorems~\ref{thm-moment}, \ref{thm-pt-to-circle-moment}, and~\ref{thm-pt-to-pt-moment}, and the estimates for geodesics in Section~\ref{sec-geodesic-bounds}.  

Due to our axiomatic approach, our proofs do not require any outside input besides the existence of LFPP subsequential limits from~\cite{dddf-lfpp} and a general theorem about local metrics from~\cite{local-metrics} (both of which can be taken as black boxes). To understand the paper, the reader only needs to be familiar with basic properties of the GFF, as reviewed, e.g., in~\cite{shef-gff} and the introductory sections of~\cite{ss-contour,ig1,ig4}. 
\bigskip

\noindent
\textbf{Acknowledgments.}
We thank an anonymous referee for helpful comments on an earlier version of the manuscript. We thank Jian Ding, Alex Dunlap, Jason Miller, Scott Sheffield, as well as Vincent Tassion, Wendelin Werner, and their research group at ETH Z\"urich for helpful discussions. Part of the project was carried out during E. Gwynne and X. Sun's visit to MIT  and E.\ Gwynne and J.\ Pfeffer's visit to Columbia University in Fall 2018. We thank the two institutions for their hospitality. 
J.\ Dub\'edat was partially supported by NSF grant DMS-1512853.
E.\ Gwynne was supported by a Herchel Smith fellowship and a Trinity College junior research fellowship.
J.\ Pfeffer was partially supported by the National Science Foundation Graduate Research Fellowship under Grant No. 1122374.
 X.\ Sun was supported by the Simons Foundation as a Junior Fellow at Simons Society of Fellows, by NSF grant DMS-1811092, and by the Minerva fund at the Department of Mathematics at Columbia University.

\subsection{Weak LQG metrics and subsequential limits of LFPP}
\label{sec-axioms}

Let us now discuss the approximations of LQG metrics which we will be interested in.   
We first need to introduce an exponent which plays a fundamental role in the study of $\gamma$-LQG distances.
It is shown in~\cite{dg-lqg-dim} that for each $\gamma \in (0,2)$, there is an exponent $d_\gamma > 2$ which arises in various approximations of LQG distances.
For example, for certain random planar maps in the $\gamma$-LQG universality class, a graph-distance ball of radius $r \in \BB N$ in the map typically has of order $r^{d_\gamma + o_r(1)}$ vertices. 
It is shown in~\cite{gp-kpz} that $d_\gamma$ is the Hausdorff dimension of the $\gamma$-LQG metric. 
The value of $d_\gamma$ is not known explicitly except for $d_{\sqrt{8/3}}= 4$, but reasonably tight upper and lower bounds are available; see~\cite{dg-lqg-dim}.
We define
\eqb \label{eqn-xi-def}
\xi  = \xi_\gamma := \frac{\gamma}{d_\gamma} .
\eqe 

For concreteness, we will primarily focus on the whole-plane case.
We say that a random distribution $h$ on $\BB C$ is a \emph{whole plane GFF plus a continuous function}
 if there exists a coupling of $h$ with a random continuous function $f : \BB C\rta\BB R$ such that the law of $h-f$ is that of a whole-plane GFF. 
If such a coupling exists for which $f$ is bounded, then we say that $h$ is a \emph{whole-plane GFF plus a bounded continuous function}.\footnote{The reason why we sometimes restrict to bounded continuous functions is that it ensures that the convolution with the whole-plane heat kernel is finite (so $D_h^\ep$ is defined) and it makes parts of the proof of Theorem~\ref{thm-lfpp-axioms} simpler.}
 Note that the whole-plane GFF is defined only modulo a global additive constant, but these definitions do not depend on the choice of additive constant.

If $h$ is a whole-plane GFF, or more generally a whole-plane GFF plus a bounded continuous function, we define the mollified GFF $h_\ep^*(z)$ for $\ep > 0$ and $z\in\BB C$ as in~\eqref{eqn-gff-convolve}.
For $z,w\in\BB C$ and $\ep> 0$, we define the \emph{$\ep$-LFPP metric} by\footnote{The intuitive reason why we look at $e^{\xi h_\ep^*(z)}$ instead of $e^{\gamma h_\ep^*(z)}$ to define the metric is as follows. By~\eqref{eqn-measure-construct}, we can scale LQG areas by a factor of $C>0$ by adding $\gamma^{-1}\log C$ to the field. By~\eqref{eqn-lfpp}, this results in scaling distances by $C^{\xi/\gamma} = C^{1/d_\gamma}$, which is consistent with the fact that the ``dimension" should be the exponent relating the scaling of areas and distances.}
\eqb \label{eqn-lfpp}
D_h^\ep(z,w) := \inf_{P : z\rta w} \int_0^1 e^{\xi h_\ep^*(P(t))} |P'(t)| \,dt 
\eqe
where the infimum is over all piecewise continuously differentiable paths from $z$ to $w$. One should think of LFPP as the metric analog of the approximations of the LQG measure in~\eqref{eqn-measure-construct}. 

\begin{remark}\label{rmk:tight}
The reason why we define LFPP using $h_\ep^*$ instead of some other continuous approximation of the GFF is that this is the approximation for which tightness is proven in~\cite{dddf-lfpp}. If we had a tightness result similar to those in \cite{dddf-lfpp} for LFPP defined using a different approximation (such as the circle average process of~\cite[Section 3.1]{shef-kpz} or the convolution of $h$ with $\ep^{-1} \phi(|z-w|/\sqrt\ep)$, where $\phi$ is a continuous non-negative radially symmetric function with total integral one), then similar arguments to those in Section~\ref{sec-checking-axioms} would show that the subsequential limits are also weak LQG metrics. Together with the uniqueness of weak LQG metrics proven in~\cite{gm-uniqueness}, this means that in order to show that such approximations converge to the $\gamma$-LQG metric one only needs to prove tightness.
\end{remark}
 
For $\ep > 0$, let $\frk a_\ep$ be the median of the $D_h^\ep$-distance between the left and right boundaries of the unit square along paths which stay in the unit square.
It follows from results in~\cite{dddf-lfpp} (see Lemma~\ref{lem-lfpp-tight} below) that the laws of the metrics $\{\frk a_\ep^{-1} D_h^\ep\}_{\ep > 0}$ are tight with respect to the local uniform topology on $\BB C\times \BB C$ and every subsequential limit induces the Euclidean topology on $\BB C$. 

Building on this, we will prove that in fact the metrics $\frk a_\ep^{-1} D_h^\ep$ admit subsequential limits in probability and that every subsequential limit satisfies a certain natural list of axioms. 
To state these axioms, we need some preliminary definitions.
Let $(X,D)$ be a metric space.  
\medskip
 
\noindent
For a curve $P : [a,b] \rta X$, the \emph{$D$-length} of $P$ is defined by 
\eqbn
\op{len}\left( P ; D  \right) := \sup_{T} \sum_{i=1}^{\# T} D(P(t_i) , P(t_{i-1})) 
\eqen
where the supremum is over all partitions $T : a= t_0 < \dots < t_{\# T} = b$ of $[a,b]$. Note that the $D$-length of a curve may be infinite.
\medskip

\noindent
For $Y\subset X$, the \emph{internal metric of $D$ on $Y$} is defined by
\eqb \label{eqn-internal-def}
D(x,y ; Y)  := \inf_{P \subset Y} \op{len}\left(P ; D \right) ,\quad \forall x,y\in Y 
\eqe 
where the infimum is over all paths $P$ in $Y$ from $x$ to $y$. 
Then $D(\cdot,\cdot ; Y)$ is a metric on $Y$, except that it is allowed to take infinite values.  
\medskip
 
\noindent
We say that $(X,D)$ is a \emph{length space} if for each $x,y\in X$ and each $\ep > 0$, there exists a curve of $D$-length at most $D(x,y) + \ep$ from $x$ to $y$. 
\medskip

\noindent
A \emph{continuous metric} on a domain $U\subset\BB C$ is a metric $D$ on $U$ which induces the Euclidean topology on $U$, i.e., the identity map $(U,|\cdot|) \rta (U,D)$ is a homeomorphism. 
We equip the space of continuous metrics on $U$ with the local uniform topology for functions from $U\times U$ to $[0,\infty)$ and the associated Borel $\sigma$-algebra.
We allow a continuous metric to have $D(u,v) = \infty$ if $u$ and $v$ are in different connected components of $U$.
In this case, in order to have $D^n\rta D$ w.r.t.\ the local uniform topology we require that for large enough $n$, $D^n(u,v) = \infty$ if and only if $D(u,v)=\infty$.   
\medskip

\noindent
Let $\mcl D'(\BB C)$ be the space of distributions (generalized functions) on $\BB C$, equipped with the usual weak topology.
For $\gamma \in (0,2)$, a \emph{weak $\gamma$-LQG metric} is a measurable function $h\mapsto D_h$ from $\mcl D'(\BB C)$ to the space of continuous metrics on $\BB C$ such that the following is true whenever $h$ is a whole-plane GFF plus a continuous function.
\begin{enumerate}[I.] 
\item \textbf{Length space.} Almost surely, $(\BB C , D_h)$ is a length space, i.e., the $D_h$-distance between any two points of $\BB C$ is the infimum of the $D_h$-lengths of $D_h$-continuous paths (equivalently, Euclidean continuous paths) between the two points. \label{item-metric-length}
\item \textbf{Locality.} Let $U\subset\BB C$ be a deterministic open set. 
The $D_h$-internal metric $D_h(\cdot,\cdot ; U)$ is determined a.s.\ by $h|_U$.  \label{item-metric-local}  
\item \textbf{Weyl scaling.} Let $\xi$ be as in~\eqref{eqn-xi-def} and for each continuous function $f : \BB C\rta \BB R$, define \label{item-metric-f}  
\eqb \label{eqn-metric-f}
(e^{\xi f} \cdot D_h) (z,w) := \inf_{P : z\rta w} \int_0^{\op{len}(P ; D_h)} e^{\xi f(P(t))} \,dt , \quad \forall z,w\in \BB C ,
\eqe
where the infimum is over all continuous paths from $z$ to $w$ parametrized by $D_h$-length. Then a.s.\ $ e^{\xi f} \cdot D_h = D_{h+f}$ for every continuous function $f : \BB C\rta\BB R$.  
\item \textbf{Translation invariance.} For each deterministic point $z \in \BB C$, a.s.\ $D_{h(\cdot + z)} = D_h(\cdot+ z , \cdot+z)$.  \label{item-metric-translate}
\item \textbf{Tightness across scales.} Suppose that $h$ is a whole-plane GFF and let $\{h_r(z)\}_{r > 0, z\in\BB C}$ be its circle average process. For each $r > 0$, there is a deterministic constant $\frk c_r > 0$ such that the set of laws of the metrics $\frk c_r^{-1} e^{-\xi h_r(0)} D_h (r \cdot , r\cdot)$ for $r > 0$ is tight (w.r.t.\ the local uniform topology). Furthermore, the closure of this set of laws w.r.t.\ the Prokhorov topology on continuous functions $\BB C\times \BB C \rta [0,\infty)$ is contained in the set of laws on continuous metrics on $\BB C$ (i.e., every subsequential limit of the laws of the metrics $\frk c_r^{-1} e^{-\xi h_r(0)} D_h (r \cdot  , r \cdot )$ is supported on metrics which induce the Euclidean topology on $\BB C$). Finally, there exists  \label{item-metric-coord} 
$\Lambda > 1$ such that for each $\delta \in (0,1)$, 
\eqb \label{eqn-scaling-constant}
\Lambda^{-1} \delta^\Lambda \leq \frac{\frk c_{\delta r}}{\frk c_r} \leq \Lambda \delta^{-\Lambda} ,\quad\forall r  > 0.
\eqe
\end{enumerate}
We emphasize that the definition of a weak $\gamma$-LQG metric depends on $\gamma$ only via the parameter $\xi$ in Axiom~\ref{item-metric-f}. 
We will therefore sometimes say that a metric satisfying the above axioms is a \emph{weak LQG metric with parameter $\xi$}. 

It is easy to see, at least heuristically, why Axioms~\ref{item-metric-length} through~\ref{item-metric-coord} should be satisfied for subsequential limits of LFPP, although there is some subtlety involved in checking these axioms rigorously. The first main result of this paper is the following statement, whose proof builds on results from~\cite{dddf-lfpp,local-metrics}. 

\begin{thm} \label{thm-lfpp-axioms}
Let $\gamma \in (0,2)$. For every sequence of $\ep$'s tending to zero, there is a weak $\gamma$-LQG metric $D$ and a subsequence $\{\ep_n\}_{n\in\BB N}$ for which the following is true. Let $h$ be a whole-plane GFF, or more generally a whole-plane GFF plus a bounded continuous function. Then the re-scaled LFPP metrics $\frk a_{\ep_n}^{-1} D_h^{\ep_n}$ from~\eqref{eqn-lfpp} converge in probability to $D_h$. 
\end{thm}

We will explain why we get convergence in probability, instead of just in law, in Theorem~\ref{thm-lfpp-axioms} just below.
Let us first discuss the axioms for a weak LQG metric.
Axioms~\ref{item-metric-length} through~\ref{item-metric-translate} are natural from the perspective that $\gamma$-LQG is a ``random two-dimensional Riemannian manifold" obtained by exponentiating $h$. 
Axiom~\ref{item-metric-coord} is a substitute for exact scale invariance of the metric. 
To explain this, it is expected (and will be proven in~\cite{gm-uniqueness,gm-coord-change}) that the $\gamma$-LQG metric, like the $\gamma$-LQG measure, is invariant under coordinate changes of the form~\eqref{eqn-lqg-coord}. 
In particular, it should be the case that for any $a\in\BB C \setminus \{0\}$, a.s.\
\eqb \label{eqn-lqg-metric-coord}
 D_h \left( a\cdot  , a \cdot   \right) = D_{h(a\cdot )  +Q\log |a| }(\cdot,\cdot)  ,\quad \text{for} \quad Q =\frac{2}{\gamma} + \frac{\gamma}{2} .
\eqe
Under Axiom~\ref{item-metric-f}, the formula~\eqref{eqn-lqg-metric-coord} together with the scale invariance of the law of $h$, modulo an additive constant, implies Axiom~\ref{item-metric-coord} with $\frk c_r = r^{\xi Q}$. 
We define a \emph{strong LQG metric} to be a mapping $h\mapsto D_h$ which satisfies Axioms~\ref{item-metric-length} through~\ref{item-metric-translate} as well as~\eqref{eqn-lqg-metric-coord}. 

A similar definition of a strong LQG metric has appeared in earlier literature.
Indeed, the paper~\cite{mq-geodesics} proved several properties of geodesics for any metric associated with $\gamma$-LQG which satisfies a similar list of axioms to the ones in our definition of a strong LQG metric; however, at that point such a metric had only been constructed for $\gamma=\sqrt{8/3}$.\footnote{Although the axioms in~\cite{mq-geodesics} are formulated in a slightly different way from our axioms for a strong LQG metric, it can be proven, with some work, that the two notions are equivalent. The analog of Axiom~\ref{item-metric-local} in~\cite{mq-geodesics}, which asserts that metric balls are local sets, is proven to be equivalent to our Axiom~\ref{item-metric-local} in~\cite[Lemma 2.2]{local-metrics}. The analog of Axiom~\ref{item-metric-f} in~\cite{mq-geodesics} is stated only for constant functions, but it is easy to check that this axiom implies Axiom~\ref{item-metric-f}. For example, this is explained in~\cite[Section 2.4]{gms-poisson-voronoi} in the special case when $\gamma=\sqrt{8/3}$, and the same argument works for general $\gamma \in (0,2)$. In~\cite[Assumption 1.1]{mq-geodesics}, the authors allow for fields on any open domain in $\BB C$ and assume that the metric satisfies a LQG coordinate change formula for general conformal maps, not just complex affine maps. It is shown in~\cite{gm-coord-change} that a strong LQG metric in the sense of this paper gives rise to a metric associated with a GFF on any proper sub-domain of $\BB C$ which satisfies the LQG coordinate change formula for general conformal maps.}

It far from obvious that subsequential limits of LFPP satisfy~\eqref{eqn-lqg-metric-coord}. The reason for this is that scaling space results in scaling the value of $\ep$ in~\eqref{eqn-lfpp}, which in turn changes the subsequence which we are working with. It will eventually be proven in~\cite{gm-uniqueness} that every weak LQG metric satisfies~\eqref{eqn-lqg-metric-coord}, i.e., every weak LQG metric is a strong LQG metric, but the proof requires all of the results of the present paper as well as those of~\cite{local-metrics,gm-confluence}.

Nevertheless, Axiom~\ref{item-metric-coord} can be used in place of~\eqref{eqn-lqg-metric-coord} in many situations. 
Basically, this axiom allows us to compare distance quantities at the same Euclidean scale. 
For example, Axiom~\ref{item-metric-coord} implies that if $U\subset\BB C$ is open and $K\subset U$ is compact, then the laws of 
\eqb\label{eq:tight}
\left( \frk c_r^{-1} e^{-\xi h_r(0)} D_h (r K , r\bdy U) \right)^{-1} \quad \text{and} \quad \frk c_r^{-1} e^{-\xi h_r(0)} \sup_{u,v\in r K} D_h (  u ,  v ; r U )
\eqe
as $r$ varies are tight.  

Part of the proof of Theorem~\ref{thm-lfpp-axioms} is to show that for any joint subsequential limit $(h,D_h)$ of the laws of the pairs $(h,\frk a_\ep^{-1} D_h^\ep)$, the limiting metric $D_h$ is a measurable function of $h$. 
This is not obvious since convergence in law does not in general preserve measurability.
In our setting, we will prove that $D_h$ is determined by $h$ by checking the conditions of~\cite[Corollary 1.8]{local-metrics}, which gives a list of conditions under which a random metric coupled with the GFF is determined by the GFF. 
The reason why we have convergence in probability, instead of convergence in law, in Theorem~\ref{thm-lfpp-axioms} is the following elementary probabilistic lemma (see e.g. \cite[Lemma~4.5]{ss-contour}).\footnote{Since the space of continuous metrics is not complete w.r.t.\ any natural choice of metric which induces the local uniform topology, we apply the lemma with $(\Omega_2,d_2)$ equal to the larger space of continuous functions $\BB C\times\BB C \rta [0,\infty)$ equipped with the local uniform topology, which is completely metrizable.}
 
\begin{lem} \label{lem-in-prob}
Let $(\Omega_1 , d_1) $ and $(\Omega_2 ,d_2)$ be complete separable metric spaces.
Let $X$ be a random variable taking values in $\Omega_1$ and let $\{Y^n\}_{n\in\BB N}$ and $Y$ be random variables taking values in $\Omega_2$, all defined on the same probability space, such that $(X,Y^n) \rta (X,Y)$ in law.
If $Y$ is a.s.\ determined by $X$, then $Y^n\rta Y$ in probability.
\end{lem}

Theorem~\ref{thm-lfpp-axioms} will be proven in Section~\ref{sec-checking-axioms}. Once this is done, throughout the rest of the paper we will only ever work with a weak $\gamma$-LQG metric --- we will not need to make explicit reference to LFPP.
An important advantage of this approach is that the Miller-Sheffield $\sqrt{8/3}$-LQG metric from~\cite{lqg-tbm1,lqg-tbm2,lqg-tbm3} is known to satisfy the axioms for a weak $\sqrt{8/3}$-LQG metric. See~\cite[Section 2.4]{gms-poisson-voronoi} for a careful explanation of why this is the case.
Note that~\cite[Section 2.4]{gms-poisson-voronoi} checks the coordinate change relation~\eqref{eqn-lqg-metric-coord} for the Miller-Sheffield metric which (as discussed above) implies Axiom~\ref{item-metric-coord}.    
Hence all of our results for weak $\gamma$-LQG metrics apply to both this $\sqrt{8/3}$-LQG metric and to subsequential limits of LFPP.\footnote{
The uniqueness of the weak LQG metric proven in~\cite{gm-uniqueness} implies that the Miller-Sheffield $\sqrt{8/3}$-LQG metric is the limit of LFPP for $\gamma=\sqrt{8/3}$.}
 
\begin{remark}[Liouville graph distance]
Besides LFPP, there is another natural scheme for approximating LQG metrics called \emph{Liouville graph distance} (LGD). 
The $\ep$-LGD distance between two points in $\BB C$ is defined to be the minimum number of Euclidean balls with LQG mass $\ep$ whose union contains a path between the two points.
It has been proven in~\cite{ding-dunlap-lgd} that for each $\gamma\in (0,2)$, the $\ep$-LGD metric, appropriately renormalized, admits subsequential limiting metrics as $\ep\rta 0$ which induce the Euclidean topology. 
In the contrast to LFPP, for subsequential limits of LGD the coordinate change relation~\eqref{eqn-lqg-metric-coord} is easy to verify but Weyl scaling (Axiom~\ref{item-metric-f}) appears to be very difficult to verify, so these subsequential limits are \emph{not} known to be weak LQG metrics in the sense of this paper.  
It is still an open problem to establish uniqueness of the scaling limit for LGD. 
Similar considerations apply to variants of LGD defined using embedded planar maps (such as maps constructed from LQG square subdivision~\cite{shef-kpz,ghpr-central-charge} or mated-CRT maps~\cite{ghs-dist-exponent,gms-tutte}) instead of Euclidean balls, although for these variants tightness has not been checked. 
\end{remark}

\subsection{Quantitative properties of weak LQG metrics}
\label{sec-main-results}

In what follows, we assume that $D$ is a weak $\gamma$-LQG metric and $h$ is a whole-plane GFF. 
Perhaps surprisingly, the axioms for a weak LQG metric imply much sharper bounds on the scaling constants $\frk c_r$ than~\eqref{eqn-scaling-constant}. 

\begin{thm} \label{thm-metric-scaling}
Let $\xi$ be as in~\eqref{eqn-xi-def} and let $Q =2/\gamma+\gamma/2$.  
Then for $r>0$, the scaling constants satisfy
\eqb \label{eqn-metric-scaling}
\frac{\frk c_{\delta  r } }{ \frk c_{r} } = \delta^{\xi Q + o_\delta(1)} \quad \text{as $\delta \rta 0$},
\eqe
at a rate which is uniform over all $r>0$. 
\end{thm}

The definition of a weak LQG metric uses only the parameter $\xi$. Theorem~\ref{thm-metric-scaling} connects this definition to the coordinate change parameter $Q$.
This will be important for the proof in~\cite{gm-uniqueness} that any weak LQG metric satisfies the coordinate change formula~\eqref{eqn-lqg-metric-coord}.
Theorem~\ref{thm-metric-scaling} will be proven in Section~\ref{sec-metric-scaling} by comparing $D_h$-distances to LFPP distances and using the fact that the $\delta$-LFPP distance between two fixed points is typically of order $\delta^{1-\xi Q + o_\delta(1)}$~\cite[Theorem 1.5]{dg-lqg-dim} (for convenience, for this argument we will work with a variant of LFPP which is defined in a slightly different manner than the version in~\eqref{eqn-lfpp}). 

\begin{remark} \label{remark-sqrt8/3}
Theorem~\ref{thm-metric-scaling} gives a proof purely in the continuum that the exponent $d_{\sqrt{8/3}}$ of~\cite{dzz-heat-kernel,dg-lqg-dim} is equal to $4$. Previously, this was proven in~\cite{dg-lqg-dim} (building on~\cite{ghs-map-dist}) using the known ball volume growth exponent for random triangulations~\cite{angel-peeling}.
To see why Theorem~\ref{thm-metric-scaling} implies that $d_{\sqrt{8/3}}=4$, we observe that the $\sqrt{8/3}$-LQG metric of~\cite{lqg-tbm1,lqg-tbm2,lqg-tbm3} satisfies the axioms for a weak LQG metric with parameter $\xi = 1/\sqrt 6$. 
Moreover, by the LQG coordinate change formula for the $\sqrt{8/3}$-LQG metric, Axiom~\ref{item-metric-coord} holds for this metric with with $\frk c_r = r^{5/6}$.
Theorem~\ref{thm-metric-scaling} therefore implies that if $\gamma \in (0,2)$ is chosen so that $\gamma / d_\gamma = 1/\sqrt 6$, then the associated parameter $Q =2/\gamma + \gamma/2$ satisfies $Q/\sqrt 6 = 5/6$, i.e., $Q = 5/\sqrt 6$ which is equivalent to $\gamma = \sqrt{8/3}$. 
Hence $\gamma/d_\gamma = 1/\sqrt 6$ when $\gamma =\sqrt{8/3}$, so $d_{\sqrt{8/3}}=4$. 
\end{remark}

Our next main result gives the optimal H\"older exponents for $D_h$ with respect to the Euclidean metric. 
 
\begin{thm}[Optimal H\"older exponents] \label{thm-optimal-holder}
Let $U\subset\BB C$ be open and bounded.
Almost surely, the identity map from $U$, equipped with the Euclidean metric, to $(U,D_h)$ is locally H\"older continuous with any exponent smaller than $\xi(Q-2)$ and is not locally H\"older continuous with any exponent larger than $\xi(Q-2)$. 
Furthermore, the inverse of this map is a.s.\ locally H\"older continuous with any exponent smaller than $\xi^{-1}(Q+2)^{-1}$ and is not locally H\"older continuous with any exponent larger than $\xi^{-1}(Q+2)^{-1}$. 
\end{thm}

For $\gamma=\sqrt{8/3}$, one has $\xi = 1/\sqrt 6$ and $Q = 5/\sqrt 6$, so the optimal H\"older exponents are given by 
\eqb
\xi(Q-2) = \frac{1}{6} (5 - 2 \sqrt 6) \approx 0.0168
\quad \text{and} \quad
\xi^{-1}(Q+2)^{-1} = 30 - 12 \sqrt 6 \approx 0.6061 .
\eqe

The intuitive reason why Theorem~\ref{thm-optimal-holder} is true is as follows. If $z$ is an $\alpha$-thick point for $h$, i.e., the circle average satisfies $h_\ep(z) = (\alpha+o_\ep(1)) \log\ep^{-1}$ as $\ep\rta 0$, then we can show that the $D_h$-distance from $z$ to $\bdy B_\ep(z)$ behaves like $\ep^{\xi(Q-\alpha)+ o_\ep(1)}$ as $\ep\rta 0$. Indeed, this is an easy consequence of the estimates in Section~\ref{sec-ptwise-dist}. 
Almost surely, $\alpha$-thick points exist for $\alpha \in (-2,2)$ but not for $|\alpha| >2$~\cite{hmp-thick-pts}.

We next state some basic moment estimates for distances which are metric analogues of the well-known fact that the $\gamma$-LQG measure has finite moments of all orders in $(-\infty , 4/\gamma^2)$~\cite[Theorems 2.11 and 2.12]{rhodes-vargas-review}.

\begin{thm}[Moment bounds for diameters] \label{thm-moment}
Let $U\subset \BB C$ be open and let $K\subset U$ be a compact connected set with more than one point. 
Then the $U$-internal diameter of $K$ satisfies 
\eqb \label{eqn-moment}
\BB E\left[ \left( \sup_{z, w \in K} D_h(z,w ; U)  \right)^p \right] < \infty ,\quad \forall p \in \left( -\infty , \frac{4 d_\gamma }{\gamma^2} \right) .
\eqe 
\end{thm}
 
For $\gamma=\sqrt{8/3}$, we get finite moments up to order 6. 
We also have the following bound for distances between sets. 
In this case, we get finite moments of all orders.

\begin{thm}[Distance between sets] \label{thm-two-set-dist0}
Let $U \subset \BB C$ be an open set (possibly all of $\BB C$) and let $K_1,K_2\subset U$ be connected, disjoint compact sets which are not singletons. 
Then 
\eqb \label{eqn-two-set-dist0}
 \BB E\left[ \left( D_h( K_1,  K_2 ;   U) \right)^p \right] < \infty,\quad\forall p \in \BB R .
\eqe
\end{thm}

The results of~\cite{dddf-lfpp} show that if $D_h$ is a subsequential scaling limit of the LFPP metrics~\eqref{eqn-lfpp}, then one has the following slightly stronger version of Theorem~\ref{thm-two-set-dist0}: 
\eqb \label{eqn-lognormal}
\BB P\left[A^{-1} \leq \frk a_\ep^{-1} D_h^\ep( K_1,  K_2 ;   U) \leq A \right] \geq 1 - c_0 e^{-c_1 (\log A)^2/\log\log A} ,\quad\forall A > 2 e^e
\eqe
for constants $c_0 , c_1 > 0$ allowed to depend on $K_1,K_2,U$. \emph{A posteriori}, one gets~\eqref{eqn-lognormal} for every weak LQG metric since~\cite{gm-uniqueness} proves that the weak LQG metric is unique for each $\gamma \in (0,2)$, so in particular it is the limit of LFPP.
 
We now turn our attention to point-to-point distances. These estimates also work if we allow the field to have a log singularity. To make sense of the metric in this case, we note that since $\log|\cdot|$ is continuous away from 0, we can define $D_{h-\alpha\log|\cdot|} $ as a continuous length metric on $\BB C\setminus \{0\}$ by $D_{h-\alpha\log|\cdot|} = |\cdot|^{-\alpha\xi} \cdot D_h$, in the notation~\eqref{eqn-metric-f}. We can then extend $D_{h-\alpha\log|\cdot|}$ to a metric defined on all of $\BB C$ which is allowed to take the value $\infty$ by taking the infima of the $D_{h-\alpha\log|\cdot|}$-lengths of paths.  We can similarly define the metric associated with fields with two or more log singularities.

\begin{thm}[Distance from a point to a circle] \label{thm-pt-to-circle-moment}
Let $\alpha \in \BB R$ and let $h^\alpha := h - \alpha\log|\cdot|$. 
If $\alpha \in (-\infty ,Q)$, then 
\eqb\label{eqn-pt-to-circle-moment0} 
\BB E\left[ \left(  D_{h^\alpha}\left(  0 , \bdy \BB D  \right)  \right)^p \right]  < \infty , \quad \forall p\in \left(-\infty , \frac{2d_\gamma }{\gamma}(Q-\alpha)\right) .
\eqe     
If $\alpha >Q$, then a.s.\ $D_{h^\alpha}(0,z) = \infty$ for every $z\in\BB C\setminus \{0\}$. 
\end{thm}

For example, if $\gamma =\sqrt{8/3}$ and $\alpha = 0$, we get finite moments up to order 10. 
If instead $\gamma =\sqrt{8/3}$ and $\alpha = \gamma$ (which corresponds to the case when 0 is a ``quantum typical" point, see, e.g.,~\cite[Proposition 3.4]{shef-kpz}) we only get finite moments up to order 2. 
In the critical case when $\alpha =  Q$, our estimates at this point are not sufficiently sharp to determine whether $D_{h^Q}\left(  0 , \bdy \BB D  \right) $ is finite.
However, once we know that every weak LQG metric is a strong LQG metric (which is proven in~\cite{gm-uniqueness}) it is not hard to check that a.s.\ $D_{h^Q}\left(  0 , z \right) = \infty$ for every $z\in\BB C\setminus \{0\}$. Similar comments apply in the case when $\alpha=Q$ or $\beta=Q$ in Theorem~\ref{thm-pt-to-pt-moment} just below. 
 
\begin{thm}[Distance between two points] \label{thm-pt-to-pt-moment}
Let $\alpha , \beta \in \BB R$, let $z,w\in \BB C$ be distinct, and let $h^{\alpha,\beta} := h - \alpha \log|\cdot - z| - \beta \log|\cdot - w|$.  
If $\alpha,\beta \in (-\infty,Q)$, then 
\eqb\label{eqn-pt-to-pt-moment0}
\BB E\left[ \left(   D_{h^\alpha}\left(  z,w ; B_{4|z-w|}(z)  \right)  \right)^p \right] < \infty, \quad \forall p\in \left(-\infty , \frac{2d_\gamma}{\gamma}(Q- \max\{\alpha,\beta\} ) \right).
\eqe   
If either $\alpha > Q$ or $\beta > Q$, then a.s.\ $D_{h^{\alpha,\beta}}(z,w) = \infty$.
\end{thm}
 
As applications of our main results, in Section~\ref{sec-geodesic-bounds} we will also prove some estimates which constrain the behavior of $D_h$-geodesics and which will be important in~\cite{gm-uniqueness}.
To be more precise, the first main estimate of Section~\ref{sec-geodesic-bounds} is Proposition~\ref{prop-line-path}, which gives an upper bound for the amount of time that a $D_h$-geodesic can spend in a small neighborhood of a line segment or a circular arc. Intuitively, one expects that this amount of time is small since LQG geodesics should be fractal and hence should look very different from smooth curves. The particular bound given in Proposition~\ref{prop-line-path} is used in~\cite[Section 3]{gm-uniqueness} to prevent a geodesic from spending a long time in an annulus with a small aspect ratio; and in~\cite[Section 5]{gm-uniqueness} in order to force a geodesic to enter a ``good" region of the plane in which certain distance bounds hold. 

The other main estimate in Section~\ref{sec-geodesic-bounds} is Proposition~\ref{prop-geo-bdy}, which is an upper bound for how much time an LQG geodesic can spend near the boundary of an LQG metric ball centered at its starting point. Intuitively, this amount of time should be small since if $P$ is a $D_h$-geodesic, then $D_h(P(0) , P(t))  = t$ but $D_h(P(0) , \cdot) $ is constant on the boundary of a $D_h$-ball centered at $P(0)$. The bound given in Proposition~\ref{prop-geo-bdy} is used in~\cite[Lemma 4.7]{gm-uniqueness}.

\begin{remark}[The case when $\xi  >2/d_2$] \label{remark-c>1}
Throughout this paper, we focus on the case of weak $\gamma$-LQG metrics.
Since $\gamma \mapsto \gamma/d_\gamma$ is increasing~\cite[Proposition 1.7]{dg-lqg-dim}, weak $\gamma$-LQG metrics have parameter $\xi \in (0,2/d_2)$ (here, $d_2 := \lim_{\gamma\rta 2^-} d_\gamma$). It is natural to wonder whether one can say anything about weak LQG metrics which satisfy the same axioms but with a parameter $\xi \geq 2/d_2$. 
In the critical case when $\xi = 2/d_2$ (i.e., $\gamma=2$), we expect that a weak LQG metric still exists and is the scaling limit of LFPP with parameter $2/d_2$. 
This metric should be the $\gamma$-LQG metric with $\gamma=2$ (the $\gamma=2$ metric should also be the limit as $\gamma\nearrow 2$ of the $\gamma$-LQG metrics, appropriately renormalized). 
We expect that all of the theorem statements in this section still hold for $\xi=2/d_2$, except that the metric $D_h$ is not H\"older continuous w.r.t.\ the Euclidean metric. 

For $\xi > 2/d_2$, we do not expect that any weak LQG metrics with parameter $\xi$ exist. 
However, there should be metrics which satisfy a similar list of properties except that such metrics no longer induce the Euclidean topology.
Instead, there should be an uncountable, dense set of points $z\in\BB C$ such that $D_h(z,w) = \infty$ for every $w\in\BB C\setminus \{z\}$. 
More precisely, let $\lambda(\xi)$ be the exponent for the typical LFPP distance between the left and right sides of $[0,1]^2$ and let $Q(\xi) = (1-\lambda(\xi))/\xi$. By~\cite[Theorem 1.5]{dg-lqg-dim}, $Q(\gamma/d_\gamma) = 2/\gamma+\gamma/2 > 2$. By~\cite[Lemma 4.1]{gp-lfpp-bounds} and~\cite[Theorem 1.1]{lfpp-pos}, $Q(\xi)  \in (0,2)$ for $\xi  >2/d_2$. 
For $\xi  >2/d_2$, the points $z\in\BB C$ which lie at infinite $D_h$-distance from every other point should correspond to so-called \emph{thick points} of $h$ (as defined in~\cite{hmp-thick-pts}) with thickness $\alpha > Q$. 

It is shown in~\cite{dg-supercritical-lfpp} that LFPP with parameter $\xi > 2/d_2$ admits subsequential scaling limits in law w.r.t.\ the topology on lower semicontinuous functions. We expect that the subsequential limit is unique, satisfies the properties discussed in the preceding paragraph, and is related to LQG with matter central charge $\mathbf c \in (1,25)$ (LQG with $\gamma \in (0,2]$ corresponds to $\mathbf c \in (-\infty,1]$). 
In particular, with $Q(\xi)$ as above, the central charge should be related to $\xi$ by $\mathbf c = 25 - 6Q(\xi)^2$. 
See~\cite{ghpr-central-charge,gp-lfpp-bounds,dg-supercritical-lfpp,lfpp-pos,apps-central-charge} for further discussion of this extended phase of LQG and some justification for the above predictions.
\end{remark}

\subsection{Outline}
\label{sec-outline}

In Section~\ref{sec-checking-axioms}, we prove Theorem~\ref{thm-lfpp-axioms}, which says that subsequential limits of LFPP are weak $\gamma$-LQG metrics, taking~\cite{dddf-lfpp} as a starting point. 
Throughout the rest of the paper, we work with an arbitrary weak $\gamma$-LQG metric (not necessarily assumed to arise as a subsequential limit of LFPP). 
Section~\ref{sec-a-priori-estimates} contains the proofs of the results stated in Section~\ref{sec-main-results}. 
In fact, for most of these results, we will prove more quantitative versions which are required to be uniform over all Euclidean scales.
At this point, these statements are not implied by the statements in Section~\ref{sec-main-results} since we are working with a weak $\gamma$-LQG metric, which is only known to be ``tight across scales" (Axiom~\ref{item-metric-coord}) instead of exactly scale invariant.

The first result that we prove for a weak $\gamma$-LQG metric is the estimate for the distance between two sets from Theorem~\ref{thm-two-set-dist0}; this is the content of Section~\ref{sec-perc-estimate}.
In Section~\ref{sec-metric-scaling}, we use this estimate to relate $D_h$-distances to LFPP distances and thereby prove Theorem~\ref{thm-metric-scaling}.
Once Theorem~\ref{thm-metric-scaling} is established, we have some ability to compare $D_h$-distances at different Euclidean scales.
This allows us to prove the moment estimate~\eqref{eqn-moment} of Theorem~\ref{thm-moment} in Section~\ref{sec-diam-moment} as well as the moment estimates of Theorems~\ref{thm-pt-to-circle-moment} and~\ref{thm-pt-to-pt-moment} in Section~\ref{sec-ptwise-dist}.
Using these moment estimates, we then prove Theorem~\ref{thm-optimal-holder} in Section~\ref{sec-holder}.

In Section~\ref{sec-geodesic-bounds}, we apply the estimates of Section~\ref{sec-main-results} to prove some bounds for $D_h$-geodesics.

\subsection{Basic notation}
\label{sec-notation}

\noindent
We write $\BB N = \{1,2,3,\dots\}$ and $\BB N_0 = \BB N \cup \{0\}$. 
\medskip

\noindent
For $a < b$, we define the discrete interval $[a,b]_{\BB Z}:= [a,b]\cap\BB Z$. 
\medskip

\noindent
If $f  :(0,\infty) \rta \BB R$ and $g : (0,\infty) \rta (0,\infty)$, we say that $f(\ep) = O_\ep(g(\ep))$ (resp.\ $f(\ep) = o_\ep(g(\ep))$) as $\ep\rta 0$ if $f(\ep)/g(\ep)$ remains bounded (resp.\ tends to zero) as $\ep\rta 0$. We similarly define $O(\cdot)$ and $o(\cdot)$ errors as a parameter goes to infinity. 
\medskip

\noindent
If $f,g : (0,\infty) \rta [0,\infty)$, we say that $f(\ep) \preceq g(\ep)$ if there is a constant $C>0$ (independent from $\ep$ and possibly from other parameters of interest) such that $f(\ep) \leq  C g(\ep)$. We write $f(\ep) \asymp g(\ep)$ if $f(\ep) \preceq g(\ep)$ and $g(\ep) \preceq f(\ep)$. 
\medskip

\noindent
Let $\{E^\ep\}_{\ep>0}$ be a one-parameter family of events. We say that $E^\ep$ occurs with
\begin{itemize}
\item \emph{polynomially high probability} as $\ep\rta 0$ if there is a $p > 0$ (independent from $\ep$ and possibly from other parameters of interest) such that  $\BB P[E^\ep] \geq 1 - O_\ep(\ep^p)$. 
\item \emph{superpolynomially high probability} as $\ep\rta 0$ if $\BB P[E^\ep] \geq 1 - O_\ep(\ep^p)$ for every $p>0$.  
\end{itemize}
We similarly define events which occur with polynomially or superpolynomially high probability as a parameter tends to $\infty$. 
\medskip

\noindent
We will often specify any requirements on the dependencies on rates of convergence in $O(\cdot)$ and $o(\cdot)$ errors, implicit constants in $\preceq$, etc., in the statements of lemmas/propositions/theorems, in which case we implicitly require that errors, implicit constants, etc., appearing in the proof satisfy the same dependencies. 
\medskip

\noindent
For $z\in\BB C$ and $r>0$, we write $B_r(z)$ for the Euclidean ball of radius $r$ centered at $z$. We also define the open annulus
\eqb \label{eqn-annulus-def}
\BB A_{r_1,r_2}(z) := B_{r_2}(z) \setminus \ol{B_{r_1}(z)} ,\quad\forall 0 < r_r < r_2 < \infty .
\eqe 
\medskip

\noindent
We write $\BB S = (0,1)^2$ for the open Euclidean unit square. 
\medskip

\section{Subsequential limits of LFPP are weak LQG metrics}
\label{sec-checking-axioms}

The goal of this section is to deduce Theorem~\ref{thm-lfpp-axioms} from the tightness result of~\cite{dddf-lfpp}.
We start in Section~\ref{sec-localized-lfpp} by introducing a ``localized" variant of LFPP, defined using the convolution of $h$ with a truncated version of the heat kernel, which (unlike the $\ep$-LFPP metric $D_h^\ep$ defined in~\eqref{eqn-lfpp}) depends locally on $h$. We then show that this localized variant of LFPP is a good approximation for $D_h^\ep$ (Lemma~\ref{lem-localized-approx}). 
In Section~\ref{sec-lfpp-tight}, we explain why the results of~\cite{dddf-lfpp} imply that the re-scaled LFPP metrics $\frk a_\ep^{-1} D_h^\ep$ as well as the associated internal metrics on certain domains in $\BB C$ are tight w.r.t.\ the local uniform topology and that every subsequential limit is a continuous length metric on $\BB C$. 
In Sections~\ref{sec-weyl-scaling}, \ref{sec-lfpp-coord}, and~\ref{sec-lfpp-local}, respectively, we will prove versions of Weyl scaling, tightness across scales, and locality for the subsequential limits (i.e., Axioms~\ref{item-metric-f}, \ref{item-metric-coord}, and~\ref{item-metric-local}).
In Section~\ref{sec-lfpp-msrble}, we use a theorem from~\cite{local-metrics} to show that subsequential limits of LFPP can be realized as measurable functions of $h$.
We then conclude the proof of Theorem~\ref{thm-lfpp-axioms}.

\newcommand{\hg}{h}
\newcommand{\hf}{{\mathsf h}}

Throughout this section, we will frequently need to switch between working with a whole-plane GFF and working with a whole-plane GFF plus a continuous function. 
As such, we will always write $\hg$ for a whole-plane GFF (with some choice of additive constant, specified as needed) and $\hf$ for a whole-plane GFF plus a continuous function (usually, this will be a whole-plane GFF plus a bounded continuous function). 
Note that this differs from the convention elsewhere in the paper, where $h$ is sometimes used to denote a whole-plane GFF plus a continuous function.

\subsection{A localized version of LFPP}
\label{sec-localized-lfpp}

Let $\hf$ be a whole-plane GFF plus a bounded continuous function. 
The mollified field $\hf_\ep^*(z)$ of~\eqref{eqn-gff-convolve} does not depend on $\hf$ in a local manner, and hence $D_\hf^\ep$-distances do not depend on $\hf$ in a local manner.
However, as $\ep\rta 0$ the heat kernel $p_{\ep^2/2}(z,w)$ concentrates around the diagonal, so we expect that $\hf_\ep^*(z)$ ``almost" depends locally on $\hf$ when $\ep$ is small. 
To quantify this, we will introduce an approximation $\wh \hf_\ep^*$ of $\hf_\ep^*$ which depends locally on $\hf$ and prove a lemma (Lemma~\ref{lem-localized-approx}) to the effect that $\wh \hf_\ep^*$ and $\hf_\ep^*$ are close when $\ep$ are small.
This will be useful at several places in this section, especially for the proof of locality (essentially, Axiom~\ref{item-metric-local}) in Section~\ref{sec-lfpp-local}.

For $\ep > 0$, let $\psi_\ep  : \BB C\rta [0,1]$ be a deterministic, smooth, radially symmetric bump function which is identically equal to 1 on $B_{\ep^{1/2}/2}(0)$ and vanishes outside of $B_{\ep^{1/2}}(0)$ (in fact, the power $1/2$ could be replaced by any $p\in (0,1)$). 
We can choose $\psi_\ep$ in such a way that $\ep\mapsto \psi_\ep$ is a continuous mapping from $(0,\infty)$ to the space of continuous functions on $\BB C$, equipped with the uniform topology. Recalling that $p_s(z,w)$ denotes the heat kernel, we define
\eqb \label{eqn-localized-def}
\wh \hf_\ep^*(z) :=  \int_{\BB C} \psi_\ep(z-w) \hf(w) p_{\ep^2/2} (z,w) \, dw ,
\eqe
with the integral interpreted in the sense of distributional pairing.
Since $\psi_\ep$ vanishes outside of $B_{\ep^{1/2}}(0)$, we have that $\wh \hf_\ep^*(z)$ is a.s.\ determined by $\hf|_{B_{\ep^{1/2}}(z)}$. 
It is easy to see that $\wh \hf_\ep^*$ a.s.\ admits a continuous modification (see Lemma~\ref{lem-localized-approx} below).
We henceforth assume that $\wh \hf_\ep^*$ is replaced by such a modification. 

As in~\eqref{eqn-lfpp}, we define the localized LFPP metric
\eqb \label{eqn-localized-lfpp}
\wh D_\hf^\ep(z,w) := \inf_{P : z\rta w} \int_0^1 e^{\xi \wh \hf_\ep^*(P(t))} |P'(t)| \,dt ,
\eqe
where the infimum is over all piecewise continuously differentiable paths from $z$ to $w$. 
By the definition of $\wh \hf_\ep^*$, 
\eqb \label{eqn-localized-property}
\text{for any open $U\subset \BB C$, the internal metric $\wh D_\hf^\ep(\cdot,\cdot; U)$ is a.s.\ determined by $\hf|_{B_{\ep^{1/2}}(U)}$} .
\eqe

\begin{lem} \label{lem-localized-approx}
Let $\hf$ be a GFF plus a bounded continuous function.
Then a.s.\ $(z,\ep)\mapsto \wh \hf_\ep^*(z)$ is continuous. 
Furthermore, for each bounded open set $U\subset \BB C$, a.s.\ 
\eqb \label{eqn-localized-approx}
\lim_{\ep\rta 0} \sup_{z\in \ol U} |\hf_\ep^*(z) - \wh \hf_\ep^*(z)|  = 0.
\eqe
In particular, a.s.\
\eqb \label{eqn-localized-lfpp-approx}
\lim_{\ep\rta 0}  \frac{\wh D_\hf^\ep(z,w;U)}{D_\hf(z,w;U)} = 1 , \quad \text{uniformly over all $z,w\in U$ with $z\not=w$}.
\eqe
\end{lem}

To prove Lemma~\ref{lem-localized-approx}, we will need the following elementary estimate for the circle average process, whose proof we postpone until after the proof of Lemma~\ref{lem-localized-approx}. 

\begin{lem} \label{lem-gff-sup} 
Let $\hg$ be a whole-plane GFF (with any choice of additive constant) and let $\{\hg_r\}_{r\geq 0}$ be its circle average process. For each $R>0$ and $\zeta > 0$, a.s.\
\eqb \label{eqn-gff-sup}
  \sup_{z\in B_R(0)}  \sup_{r > 0}   \frac{|\hg_r(z)|}{  \max\{  (2+\zeta)  \log(1/r) , (\log r)^{1/2+\zeta} ,  1\} }    < \infty .
\eqe
\end{lem}

\begin{proof}[Proof of Lemma~\ref{lem-localized-approx}]
We first consider the case when $\hf = \hg$ is a whole-plane GFF normalized so that $\hg_1(0) = 0$. 
The functions $w \mapsto \psi_\ep(z-w)$ and $w\mapsto p_{\ep^2/2} (z,w)$ are each radially symmetric about $z$, i.e., they depend only on $|z-w|$. 
Using the circle average process $\{\hg_r\}_{r > 0}$, we may therefore write in polar coordinates
\eqb
\hg_\ep^*(z) = \frac{2}{ \ep^2} \int_0^\infty  r \hg_r(z) e^{-r^2/\ep^2} \,dr \quad \text{and} \quad
\wh \hg_\ep^*(z) = \frac{2 }{  \ep^2} \int_0^{\ep^{1/2}} r \hg_r(z) \psi_\ep(r) e^{-r^2/\ep^2} \,dr .
\eqe
From this representation and the continuity of the circle average process, we infer that $(z,\ep) \mapsto \wh \hg_\ep^*(z)$ a.s.\ admits a continuous modification. 

Since $\psi_\ep \equiv 1$ on $B_{\ep^{1/2}/2}(z)$ and $\psi_\ep$ takes values in $[0,1]$, 
\eqb \label{eqn-localized-compare-pt}
|\hg_\ep^*(z) - \wh \hg_\ep^*(z)| \leq  \frac{2}{\ep^2} \int_{\ep^{1/2}/2}^\infty r |\hg_r(z)| e^{-r^2/\ep^2} \, dr .
\eqe
By Lemma~\ref{lem-gff-sup} (applied with $\zeta  =1/2$, say), there is a random constant $C = C(U)  > 0$ such that $|\hg_r(z) | \leq C \max\{ \log(1/r) , \log r , 1\}$ for each $z\in U$ and $r  > 0$. 
Plugging this into~\eqref{eqn-localized-compare-pt} shows that a.s.\
\eqb
\sup_{z\in U} |\hg_\ep^*(z) - \wh \hg_\ep^*(z)|
\leq \frac{2 C }{\ep^2} \int_{\ep^{1/2}}^\infty r   \max\{ \log(1/r) , \log r , 1\}     e^{-r^2/\ep^2} \, dr   ,
\eqe
which tends to zero exponentially fast as $\ep\rta 0$. This gives~\eqref{eqn-localized-approx} in the case of a whole-plane GFF with $\hg_1(0) = 0$.

If $f : \BB C\rta \BB R$ is a bounded continuous function, we similarly obtain a.s.\ $\lim_{\ep\rta 0} \sup_{z\in U} |f_\ep^*(z) - \wh f_\ep^*(z)|  = 0$, using the notation~\eqref{eqn-gff-convolve} and~\eqref{eqn-localized-def} with $f$ in place of $h$ or $\hf$. 
This gives~\eqref{eqn-localized-approx} in the case of a whole-plane GFF plus a bounded continuous function. 
The relation~\eqref{eqn-localized-lfpp-approx} is immediate from~\eqref{eqn-localized-lfpp} and the definition of LFPP. 
\end{proof}

To conclude the proof of Lemma~\ref{lem-localized-approx} we still need to prove Lemma~\ref{lem-gff-sup}.
To deal with large values of $r$, we will use the following lemma.

\begin{lem} \label{lem-gff-tail} 
Let $\hg$ be a whole-plane GFF. For each $R>0$ and $\zeta > 0$, a.s.\
\eqb \label{eqn-gff-sup-lim}
\lim_{r\rta \infty} \sup_{z\in B_R(0)} \frac{|\hg_r(z) |}{(\log r)^{1/2+\zeta}} = 0.
\eqe
\end{lem}
\begin{proof} 
The process $\{\hg_r(z) - \hg_r(0) : z\in B_R(0) , r \in [1/2 , 1]\}$ is centered Gaussian with variances bounded above by a constant depending only on $ R$. 
Furthermore, this process a.s.\ admits a continuous modification~\cite[Proposition 3.1]{shef-kpz}, so if we replace it by such a modification then a.s.\ $\sup_{z\in B_R(0)} \sup_{r \in [1/2,1]} |\hg_r(z) - \hg_r(0)| < \infty$.
By the Borel-TIS inequality~\cite{borell-tis1,borell-tis2} (see, e.g.,~\cite[Theorem 2.1.1]{adler-taylor-fields}), we have $\BB E\left[ \sup_{z\in B_R(0)} \sup_{r \in [1/2,1]} |\hg_r(z) - \hg_r(0)| \right] < \infty$ and there are constants $c_0 , c_1 > 0$ depending only on $R$ such that for each $A > 0$,
\eqb \label{eqn-gff-sup0}
\BB P\left[  \sup_{z\in B_R(0)} \sup_{r \in [1/2,1]} |\hg_r(z) - \hg_r(0)|  > A \right] \leq c_0 e^{-c_1 A^2} .
\eqe
Note that we absorbed the $R$-dependent constant $\BB E\left[ \sup_{z\in B_R(0)} \sup_{r \in [1/2,1]} |\hg_r(z) - \hg_r(0)| \right] $ into $c_0$.

By the scale invariance of the law of $\hg$, viewed modulo an additive constant, we infer from~\eqref{eqn-gff-sup0} that for each $k \in \BB N_0$ and $A>0$, 
\eqb \label{eqn-gff-sup-k}
\BB P\left[  \sup_{z\in B_{R 2^k }(0)} \sup_{r \in [2^{k-1}  , 2^k ]} |\hg_r(z) - \hg_r(0)|  > A \right] \leq c_0 e^{-c_1 A^2} .
\eqe
By applying this with $A$ equal to a universal constant times $k^{1/2+\zeta/2}$, say, then using the Borel-Cantelli lemma, we get that a.s.\  
\eqb \label{eqn-gff-sup-as}
 \lim_{k\rta\infty} \sup_{z\in B_{R 2^k }(0)} \sup_{r \in [2^{k-1}  , 2^k ]} \frac{ |\hg_r(z) - \hg_r(0)|}{(\log r)^{1/2+\zeta}}  = 0 .
\eqe
Each $z\in K$ is contained in $B_{R 2^k}(0)$ for each $k\in\BB N$ and each $r \geq 1/2$ is contained in $[2^{k-1}  , 2^k]$ for some $k\in\BB N$.
Hence,~\eqref{eqn-gff-sup-as} implies that a.s.\
\eqb \label{eqn-gff-sup-r}
\lim_{r\rta \infty} \sup_{z\in B_R(0)} \frac{|\hg_r(z) - \hg_r(0)|}{(\log r)^{1/2+\zeta}} = 0.
\eqe
Since $t\mapsto \hg_{e^t}(0)$ is a standard two-sided linear Brownian motion~\cite[Section 3]{shef-kpz}, it follows that a.s.\ $|\hg_r(0)| / (\log r)^{1/2+\zeta} \rta 0$ as $r\rta\infty$. 
Combining this with~\eqref{eqn-gff-sup-r} yields~\eqref{eqn-gff-sup-lim}.
\end{proof}

\begin{proof}[Proof of Lemma~\ref{lem-gff-sup}]
Standard estimates for the maximum of the circle average process (see, e.g., the proof of~\cite[Lemma 3.1]{hmp-thick-pts}) show that a.s.\
\eqb
\sup_{z\in B_R(0)} \sup_{r\in (0,1/2]} \frac{|\hg_r(z)|}{(2+\zeta) \log(1/r)} < \infty .
\eqe
By the continuity of the circle average process, a.s.\ for any $r_0 > 1/2$, $\sup_{z\in B_R(0)} \sup_{r\in [1/2 , r_0]} |\hg_r(z)| < \infty $.  
By Lemma~\ref{lem-gff-tail}, it is a.s.\ the case that for each large enough $r_0 > 0$, 
\eqb
\sup_{z\in B_R(0)}  \sup_{r\geq r_0}   \frac{|\hg_r(z)|}{(\log r)^{1/2+\zeta} }    < \infty  . 
\eqe
Combining these estimates gives~\eqref{eqn-gff-sup}.
\end{proof}

\subsection{Subsequential limits}
\label{sec-lfpp-tight}

In this subsection we explain why the results of~\cite{dddf-lfpp} imply that the laws of the re-scaled LFPP metrics $\frk a_\ep^{-1} D_\hf^\ep$ are tight (this is not entirely immediate since~\cite{dddf-lfpp} considers a slightly different class of fields and only looks at metrics on bounded domains). 
We will in fact obtain a stronger convergence statement which also includes the convergence of internal metrics of $\frk a_\ep^{-1} D_\hf^\ep$ on a certain class of sub-domains of $\BB C$.  

\begin{defn}[Dyadic domain] \label{def-dyadic-domain}
A closed square $S\subset\BB C$ is \emph{dyadic} if $S$ has side length $2^k$ and corners in $2^k\BB Z^2$ for some $k\in\BB Z$. 
We say that $W\subset\BB C$ is a \emph{dyadic domain} if there exists a finite collection of dyadic squares $\mcl S$ such that $W$ is the interior of $\bigcup_{S\in\mcl S}   S$. Note that a dyadic domain is a bounded open set.
\end{defn}

\begin{lem} \label{lem-lfpp-tight}
Let $\hf$ be a whole-plane GFF plus a bounded continuous function. 
\begin{enumerate}[A.]
\item The laws of the metrics $\frk a_\ep^{-1} D_\hf^\ep$ are tight w.r.t.\ the local uniform topology on $\BB C \times \BB C$ and any subsequential limit of these laws is supported on continuous length metrics on $\BB C$.  \label{item-lfpp-tight}
\item Let $\mcl W$ be the (countable) set of all dyadic domains. For any sequence of positive $\ep$'s tending to zero, there is a subsequence $\mcl E$ and a coupling of a continuous length metric $D_\hf$ on $\BB C$ and a length metric $D_{\hf,W}$ on $\ol W$ for each $W\in\mcl W$ which induces the Euclidean topology on $\ol W$ such that the following is true. 
Along $\mcl E$, we have the convergence of joint laws 
\eqb \label{eqn-lfpp-dyadic}
\left( \frk a_\ep^{-1} D_\hf^\ep ,\left\{ \frk a_\ep^{-1} D_\hf^\ep(\cdot,\cdot; \ol W )   \right\}_{W\in\mcl W} \right) 
\rta \left( D_\hf ,\left\{ D_{\hf,W}  \right\}_{W\in\mcl W} \right)  
\eqe
where the first coordinate is given the local uniform topology on $\BB C\times\BB C$ and each element of the collection in the second coordinate is given the uniform topology on $\ol W\times \ol W$. 
Furthermore, for each $W\in\mcl W$ we have the a.s.\ equality of internal metrics $D_{\hf,W}(\cdot,\cdot;W) = D_\hf(\cdot,\cdot;W)$. \label{item-lfpp-dyadic}
\end{enumerate}
\end{lem}

In the setting of Assertion~\ref{item-lfpp-tight}, we note that the space of continuous functions $\BB C\times \BB C\rta \BB R$, equipped with the local uniform topology, is separable and completely metrizable, which means that we can apply Prokhorov's theorem in this space. 
Assertion~\ref{item-lfpp-dyadic} of Lemma~\ref{lem-lfpp-tight} does \emph{not} give that $D_\hf^\ep(\cdot,\cdot;\ol W) \rta D_\hf(\cdot,\cdot; \ol W)$ in law along $\mcl E$ for each $W\in\mcl W$. 
The reason why we do not prove this statement is to avoid worrying about possible pathologies near $\bdy W$ (see Lemma~\ref{lem-internal-conv}).
We now proceed with the proof of Lemma~\ref{lem-lfpp-tight}. At several places in this section, we will use the following elementary scaling relation for LFPP. 
 
\begin{lem} \label{lem-lfpp-scale}
Let $\hg$ be a whole-plane GFF normalized so that $\hg_1(0)=0$.
Let $r > 0$ and let $\hg^r := \hg(r\cdot) - \hg_r(0)$, so that $\hg^r\eqD \hg$. 
The LFPP metrics defined as in~\eqref{eqn-lfpp} for $\hg$ and $\hg^r$ are related by
\eqb \label{eqn-lfpp-scale}
 D_{\hg^r}^{ \ep / r} \eqD D_\hg^{\ep/r} \quad \text{and} \quad
 D_{\hg^r}^{ \ep / r}( z   , w ) =  r^{-1} e^{-\xi \hg_r(0)} D_\hg^\ep(r z, r w) ,\quad \forall \ep > 0, \quad \forall z,w\in\BB C.
\eqe
\end{lem}
\begin{proof}
Using the notation~\eqref{eqn-gff-convolve}, we get from a standard change of variables that the convolutions of $\hg^r$ and $\hg$ with the heat kernel satisfy $  \hg_{ \ep / r}^{r,*}(z) = \hg_\ep^*(r z) - \hg_r(0)$ for each $\ep > 0$ and $z\in\BB C$. 
Using the definition~\eqref{eqn-lfpp} of LFPP, we now compute
\alb
e^{-\xi \hg_r(0)} D_\hg^\ep (r z  ,  r w ) 
&= \inf_{P : rz \rta   r w} \int_0^1 e^{\xi (\hg_\ep^*(P(t)) - \hg_r(0))} |P'(t)| \,dt \notag \\
&= \inf_{P : r z\rta r w} \int_0^1 e^{\xi \hg^{r,*}_{ \ep/r}(  P(t) / r) } |P'(t)| \,dt \notag \\
&= r \inf_{\wt P :   z\rta   w} \int_0^1 e^{\xi \hg^{r,*}_{ \ep/r}(  \wt P (t) ) } |\wt P'(t)| \,dt \quad \text{(set $\wt P = P/r$)} \notag \\
&= r D_{\hg^r}^{\ep/r}(z,w)  .
\ale
\end{proof}

To check that our limiting metrics are length metrics, we will need the following standard fact from metric geometry.

\begin{lem} \label{lem-bbi}
Let $X$ be a compact topological space and let $\{D^n\}_{n\in\BB N}$ be a sequence of length metrics on $X$ which converge uniformly to a metric $D$ on $X$.
Then $D$ is a length metric on $X$.
\end{lem} 
\begin{proof}
This is~\cite[Exercise 2.4.19]{bbi-metric-geometry}, which in turn is an easy consequence of~\cite[Corollary 2.4.17]{bbi-metric-geometry}. 
\end{proof}

Let us now record what we get from~\cite{dddf-lfpp}.

\begin{lem} \label{lem-lfpp-tight-square}
Let $S\subset \BB C$ be a closed square and let $\hf$ be a whole-plane GFF plus a bounded continuous function. 
The laws of the internal metrics $\frk a_\ep^{-1} D_\hg^\ep(\cdot,\cdot ;  S)$ for $\ep\in (0,1)$ are tight w.r.t.\ the uniform topology on $S\times S$ and any subsequential limit of these laws is supported on length metrics which induce the Euclidean topology on $S$. 
\end{lem}
\begin{proof} 
We first consider the case when $S =[0,1]^2$ is the Euclidean unit square and $\hf = \hg$ is a whole-plane GFF normalized so that $\hg_1(0) =0$. 
Let $\rng \hg$ be a zero-boundary GFF on $(-1,2)^2$. By the Markov property of the whole-plane GFF, we can couple $\hg$ and $\rng \hg$ in such a way that $\hg-\rng \hg$ is a.s.\ harmonic, hence continuous, on $(-1,2)^2$.

Recall the heat kernel $p_s(z,w) = \frac{1}{2\pi s} e^{-|z-w|/(2s)}$. 
For $z\in [0,1]^2$ and $\ep \in (0,1)$, we define the convolution $\rng \hg_\ep^* = \rng \hg*p_{\ep^2/2}$ as in~\eqref{eqn-gff-convolve}.
For $z,w\in (-1,2)^2$, define $D_{\rng \hg}^\ep(z,w)$ as in~\eqref{eqn-lfpp} with $\rng \hg_\ep^*$ in place of $\hg_\ep^*$. 
It is shown in~\cite[Theorem 1]{dddf-lfpp} (see also~\cite[Section 6.1]{dddf-lfpp}) that there are constants $\{\lambda_\ep\}_{\ep>0}$ such that the internal metrics $\lambda_\ep^{-1} D_{\rng \hg}^\ep\left( \cdot,\cdot;[0,1]^2\right)$ are tight  w.r.t.\ the uniform topology on $[0,1]^2\times [0,1]^2$ and any subsequential limit of these laws is supported on length metrics which induce the Euclidean topology on $[0,1]^2$. 

We now want to compare $D_{\rng \hg}^\ep$ and $D_\hg^\ep$ using the fact that $(\hg-\rng\hg)|_{(-1,2)^2}$ is a continuous function.
However, we cannot do this directly since we only have a uniform bound for $\hg - \rng\hg$ on compact subsets of $(-1,2)^2$ and the convolution~\eqref{eqn-gff-convolve} does not depend locally on the field.
To this end, we define the localized LFPP metrics $\wh D_\hg^\ep$ and $\wh D_{\rng \hg}^\ep$ as in~\eqref{eqn-localized-lfpp} with $\hf = \hg$ and with $\rng \hg$ in place of $\hg$, respectively. 
Then Lemma~\ref{lem-localized-approx} remains true with $ D_{\rng \hg}^\ep$ and $\wh D_{\rng \hg}^\ep$ in place of $D_\hg^\ep$ and $\wh D_\hg^\ep$ and with $U$ any open set satisfying $\ol U\subset (-1,2)^2$, with the same proof (actually, the proof is simpler since one does not need Lemma~\ref{lem-gff-tail}). 
Therefore, a.s.\ $\wh D_{\rng \hg}^\ep(z,w;U)  / D_{\rng \hg}^\ep(z,w;U) \rta 1$ uniformly over all distinct $z,w\in U$ and the conclusion of the preceding paragraph is true with $\wh D_{\rng \hg}^\ep$ in place of $D_{\rng \hg}^\ep$. 

Since $\hg-\rng \hg$ is a.s.\ equal to a continuous function on a neighborhood of $[0,1]^2$, we infer from~\eqref{eqn-localized-property} that a.s.\ the metrics $\wh D_{\rng \hg}^\ep(\cdot,\cdot;[0,1]^2)$ and $\wh D_\hg^\ep(\cdot,\cdot;[0,1]^2)$ are bi-Lipschitz equivalent with (random) $\ep$-independent Lipschitz constants.
By combining this with the conclusion of the preceding paragraph and Lemma~\ref{lem-bbi}, we get that the laws of the internal metrics $\lambda_\ep^{-1} D_\hg^\ep(\cdot,\cdot ;  S)$ for $\ep\in (0,1)$ are tight w.r.t.\ the uniform topology on $[0,1]^2 \times [0,1]^2$ and any subsequential limit of these laws is supported on length metrics which induce the Euclidean topology on $S$. 
In particular, this implies that $\lambda_\ep$ is bounded above and below by $\ep$-independent constants times the median $\wh D_\hg^\ep$-distance between the left and right sides of $[0,1]^2$.
By Lemma~\ref{lem-localized-approx} (for $\hg$), we now get that $\{\frk a_\ep / \lambda_\ep\}_{\ep \in (0,1)}$ is bounded above and below by positive, finite constants and the statement of the lemma holds in the special case when $\hf = \hg$ and $S = [0,1]^2$. 
 
By Lemma~\ref{lem-lfpp-scale} and the scale and translation invariance of the law of $\hg$, modulo additive constant, this implies the statement of the lemma for a general choice of $S$, but still with $\hf = \hg$. 
If $\hg$ is a whole-plane GFF and $f$ is a bounded continuous function, then the metrics $D_{\hg+ f}^\ep$ and $D_{\hg}^\ep$ are bi-Lipschitz equivalent, with Lipschitz constants $e^{\pm \xi \|f\|_\infty}$. Hence the case of a whole-plane GFF implies the case of a whole-plane GFF plus a continuous function. 
\end{proof}

We now upgrade from internal metrics on closed squares to internal metrics on closures of dyadic domains. 

\begin{lem} \label{lem-lfpp-tight-dyadic}
Let $W\subset\BB C$ be a dyadic domain. The laws of the internal metrics $\frk a_\ep^{-1} D_\hf^\ep(\cdot,\cdot ; \ol W)$ for $\ep\in (0,1)$ are tight w.r.t.\ the uniform topology on $\ol W\times \ol W$ and any subsequential limit of these laws is supported on length metrics which induce the Euclidean topology on $\ol W$. 
\end{lem}
\begin{proof}
If $W$ is a dyadic domain, then $\ol W$ has finitely many connected components and these connected components are the closures of dyadic domains which lie at positive Euclidean distance from each other. By considering each connected component separately, we can assume without loss of generality that $\ol W$ is connected.

For a connected set $X\subset \BB C$, a collection $\mcl D$ of random metrics on $X$ is tight w.r.t.\ the local uniform topology if and only if for each $\zeta >0$, there exists $\delta > 0$ such that for each $d \in \mcl D$, it holds with probability at least $1-\zeta$ that
\eqb \label{eqn-metric-tight}
d(z,w) \leq \zeta,\quad\forall z,w\in X \quad \text{such that} \quad  |z-w| \leq \delta .
\eqe 
Indeed, this is an easy consequence of the Arz\'ela-Ascoli theorem, the Prokhorov theorem, and the triangle inequality. 

For any closed square $S\subset \ol W$, the restriction of $D_\hf^\ep(\cdot,\cdot;\ol W)$ to $S $ is bounded above by the internal metric of $D_\hf^\ep(\cdot,\cdot;\ol W)$ on $S$, which equals $D_\hf^\ep(\cdot,\cdot;S)$. By Lemma~\ref{lem-lfpp-tight-square} and the above tightness criterion, the laws of the restrictions of $\{\frk a_\ep^{-1} D_\hf^\ep(\cdot,\cdot;\ol W)\}_{\ep \in (0,1)}$ to $S$ are tight. 
Since $W$ is a dyadic domain, we can choose a finite collection $\mcl S$ of closed squares such that $\bigcup_{S\in\mcl S} S = \ol W$. 

By the above tightness criterion applied to each square in $\mcl S$, for each $\zeta >0$, there exists $\delta > 0$ such that for each $\ep \in (0,1)$, it holds with probability at least $1-\zeta$ that
\eqb  \label{eqn-metric-tight'}
\frk a_\ep^{-1} D_\hf^\ep(z,w;\ol W) \leq \zeta, \quad \forall z,w\in \ol W \quad \text{s.t.} \quad  |z-w| \leq \delta \quad \text{and} \quad  \text{$z,w\in S$ for some $S\in \mcl S$}. 
\eqe 
Now assume that~\eqref{eqn-metric-tight'} holds and consider points $z, w \in \ol W$ such that $|z-w| \leq \delta/2$ but $z$ and $w$ do not lie in the same square of $\mcl S$. 
If $\delta$ is sufficiently small (depending only on the collection of squares $\mcl S$), then we can find squares $S,S' \in \mcl S$ such that $z\in S , w\in S'$, and $S\cap S'\not=\emptyset$. Since $S$ and $S'$ are closed squares, geometric considerations show that there is a $u \in S\cap S'$ such that $|z-u| \leq \delta$ and $|w-u| \leq \delta$. By~\eqref{eqn-metric-tight'} and the triangle inequality this implies that $\frk a_\ep^{-1} D_\hf^\ep(z,w ; \ol W) \leq 2\zeta$. Therefore, $\forall \ep\in (0,1)$ it holds with probability at least $1-\zeta$ that 
\eqbn
\frk a_\ep^{-1} D_\hf^\ep(z,w;\ol W) \leq 2 \zeta, \quad \forall z,w\in \ol W \quad \text{such that} \quad  |z-w| \leq \delta /2 .
\eqen
 Since $\zeta $ is arbitrary, the above tightness criterion applied on all of $\ol W$ now shows that the laws of the metrics $\frk a_\ep^{-1} D_\hf^\ep(\cdot,\cdot ; W)$ for $\ep\in (0,1)$ are tight w.r.t.\ the uniform topology on $\ol W\times \ol W$. 

Let $\wt D$ be a subsequential limit of $\frk a_\ep^{-1} D_\hf^\ep(\cdot,\cdot ; W)$ in law w.r.t.\ the local uniform topology. 
A priori $\wt D$ might be a pseudometric, not a metric.
We need to show that $\wt D$ is in fact a length metric and that it induces the Euclidean topology on $\ol W$.
To this end, consider two squares (not necessarily dyadic) $S_1\subset S_2\subset \ol W$ such that $S_1$ lies at positive Euclidean distance from $\bdy S_2\setminus \bdy W$.
For each $\ep > 0$, we have $D_\hf^\ep(S_1, W\setminus S_2 ; \ol W) = D_\hf^\ep(S_1,\bdy S_2 \setminus \bdy W ; S_2)$ and $D_\hf^\ep(S_1, W\setminus S_2 ;\ol W) \rta \wt D(S_1, W\setminus S_2)$ in law.
From this and Lemma~\ref{lem-lfpp-tight-square}, we infer that a.s.\ $\wt D(S_1, W\setminus S_2 ) > 0$. 
By considering an appropriate countable collection of such square annuli whose inner squares $S_1$ cover $\ol W$, we infer that a.s.\ $\wt D(u,v) > 0$ whenever $u,v\in \ol W$ with $u\not=v$. This implies that $\wt D$ is a metric.
Since $\ol W$ is compact, it follows that $\wt D$ induces the Euclidean topology on $\ol W$.
By Lemma~\ref{lem-bbi}, $\wt D$ is a length metric.
\end{proof}

The following lemma will allow us to extract tightness of $\frk a_\ep^{-1} D_\hf^\ep$ from tightness of $\frk a_\ep^{-1} D_\hf^\ep(\cdot,\cdot;S)$ for squares $S \subset\BB C$. 

\begin{lem} \label{lem-square-bdy-dist}
For $r >0$, let $S_r(0)$ be the closed square of side length $r$ centered at zero. 
Let $\hf$ be a whole-plane GFF plus a bounded continuous function.  
For each $p \in (0,1)$ and each $C > 0$, there exists $R = R(p,C)  > 1$ (depending on $p,C$ and the law of $\hf$) such that for each fixed $r>0$, 
\eqb \label{eqn-square-bdy-dist}
\liminf_{\ep\rta 0} \BB P\left[ \sup_{u,v\in S_r(0)} D_\hf^\ep(u,v) < \frac{1}{C} D_\hf^\ep\left(S_r(0) , \bdy S_{R r}(0)\right) \right] \geq p .
\eqe
\end{lem}
\begin{proof}
We first consider the case when $\hf = \hg$ is a whole-plane GFF normalized so that $\hg_1(0) = 0$.
By Lemma~\ref{lem-lfpp-tight-square} applied with $\ol W = S_1(0)$, there exists $R  = R(p , C) > 1$ such 
\eqb \label{eqn-square-bdy-dist0}
\liminf_{\ep\rta 0} \BB P\left[ \sup_{u,v\in S_{1/R}(0)} D_\hg^\ep(u,v) < \frac{1}{ C} D_\hg^\ep\left(S_{1/R}(0) , \bdy S_{1}(0)\right) \right] \geq p .
\eqe
The occurrence of the event in~\eqref{eqn-square-bdy-dist0} is unaffected by re-scaling $D_\hg^\ep$ by a constant factor.
By Lemma~\ref{lem-lfpp-scale} applied with $R r$ in place of $r$, we see that~\eqref{eqn-square-bdy-dist0} implies that for each fixed $r>0$, 
\eqb \label{eqn-square-bdy-dist1}
\liminf_{\ep\rta 0} \BB P\left[ \sup_{u,v\in S_{r}(0)} D_\hg^\ep(u,v) < \frac{1}{ C} D_\hg^\ep\left(S_r(0) , \bdy S_{R r}(0)\right) \right] \geq p .
\eqe

Now suppose that $\hf = \hg + f$ is a whole-plane GFF plus a bounded continuous function. 
If $f$ is a (possibly random) bounded continuous function, then $D_{\hg+f}^\ep$ and $D_\hg^\ep$ are a.s.\ bi-Lipschitz equivalent with Lipschitz constants $e^{-\xi\|f\|_\infty}$ and $e^{\xi \|f\|_\infty}$. Furthermore, since $f$ is a.s.\ bounded exists a deterministic $A > 1$ such that $\BB P\left[  e^{ \xi\|f\|_\infty} \leq A \right] \geq p$.
By~\eqref{eqn-square-bdy-dist1} with $A^2 C$ in place of $C$, we get~\eqref{eqn-square-bdy-dist} but with $1-2(1-p)$ in place of $p$. Since $p$ can be made arbitrarily close to 1, this yields~\eqref{eqn-square-bdy-dist}. 
\end{proof}

The last lemma we need for the proof of Lemma~\ref{lem-lfpp-tight} is the following deterministic compatibility statement for limits of internal metrics, which is used to get the relationship between internal metrics in assertion~\ref{item-lfpp-dyadic} of Lemma~\ref{lem-lfpp-tight}. 

\begin{lem} \label{lem-internal-conv}
Let $V\subset U\subset\BB C$ be open.
Let $\{D^n\}_{n\in\BB N}$ be a sequence of continuous length metrics on $U$ which converges to a continuous length metric $D$ (w.r.t.\ the local uniform topology on $U\times U$). 
Suppose also that $D^n(\cdot,\cdot ; \ol V ) $ converges to a continuous length metric $\wt D$ w.r.t.\ the uniform topology on $\ol V\times \ol V$.
Then $D(\cdot,\cdot;V) = \wt D(\cdot,\cdot ; V)$.
\end{lem}

In the setting of Lemma~\ref{lem-internal-conv}, we do not necessarily have $D(\cdot,\cdot;\ol V) = \wt D$. The reason is that it could be, e.g., that paths of near-minimal $\wt D$-length spend a positive fraction of their time in $\bdy V$.

\begin{proof}[Proof of Lemma~\ref{lem-internal-conv}]
Let $u , v \in V$ such that $D(u,v) < D(u,\bdy V)$. Since $D$ is a length metric, $D(u,v) = D(u,v;V) = D(u,v ; \ol V)$. 
Furthermore, for large enough $n\in\BB N$ we have $D^n(u,v) < D^n(u,\bdy V)$ which implies that $D^n(u,v) = D^n(u,v;V) = D^n(u,v;\ol V)$.
Therefore, $D^n(u,v)$ converges to both $D(u,v) = D(u,v; V)$ and $\wt D(u,v)$.
Furthermore, we have $\wt D(u,v) < \wt D(u,v; \bdy V)$ which implies that $\wt D(u,v) = \wt D(u,v;V)$. 
Consequently, $D(u,v; V) = \wt D(u,v;V)$ for each $u,v\in V$ with $D(u,v) < D(u,\bdy V)$.
This implies that the $ D$-length of any path in $V$ which lies at positive Euclidean distance from $\bdy V$ is the same as its $\wt D$-length.
Since $D(\cdot,\cdot;V)$ and $\wt D(\cdot,\cdot;V)$ are length metrics, we conclude that $  D(\cdot,\cdot;V) = \wt D(\cdot,\cdot;V)$. 
\end{proof}

\begin{proof}[Proof of Lemma~\ref{lem-lfpp-tight}] 
For $r >0$, let $S_r(0)$ be the closed square of side length $r$ centered at zero, as in Lemma~\ref{lem-square-bdy-dist}. 
Let $p\in (0,1)$ and let $R = R(p)  > 1$ be as in Lemma~\ref{lem-square-bdy-dist} with $C=2$ and with $(1+p)/2$, say, in place of $p$. Then for each fixed $r > 0$ and each small enough $\ep > 0$, it holds with probability at least $p$ that
\allb \label{eqn-square-metric-agree}
&\sup_{u,v\in S_r(0)} D_\hf^\ep(u,v) \leq \frac12 D_\hf^\ep(S_r(0) , \bdy S_{R r}(0)) \notag \\
&\qquad \text{which implies} \quad
D_\hf^\ep(u,v) = D_\hf^\ep(u,v; S_{R r}(0)) ,\quad \forall u,v\in S_{r}(0) .
\alle
We now apply Lemma~\ref{lem-lfpp-tight-square} with $S = S_{ R r}(0)$ and use that $p$ can be made arbitrarily close to 1 to get that the laws of $\frk a_\ep^{-1} D_\hf^\ep|_{S_r(0)}$ are tight w.r.t.\ the local uniform topology on $S_r(0)$.
Furthermore, any subsequential limit in law of these metrics a.s.\ induces the Euclidean topology on $S_r(0)$.
Since $r$ can be made arbitrarily large, we get that the metrics $\frk a_\ep^{-1} D_\hf^\ep$ are tight w.r.t.\ the local uniform topology on $\BB C \times \BB C$ and any subsequential limit in law is a.s.\ a continuous metric on $\BB C$.

To prove assertion~\ref{item-lfpp-tight}, it remains to check that if $D_\hf$ is a subsequential limit in law of the metrics $\frk a_\ep^{-1} D_\hf^\ep$, then a.s.\ $D_\hf$ is a length metric. To this end, let $p \in (0,1)$ and let $R = R(p)> 1$ be as above. 
By Lemma~\ref{lem-lfpp-tight-square}, if we are given $r>0$ then by possibly passing to a further subsequence we can arrange that along our subsequence, the joint law of $(\frk a_\ep^{-1} D_\hf^\ep , \frk a_\ep^{-1} D_\hf^\ep(\cdot,\cdot; S_{Rr}(0)) )$ converges to a coupling $(D_\hf,\wt D)$ where $\wt D$ is a length metric on $S_{Rr}(0)$. 
By passing to the (subsequential) limit in~\eqref{eqn-square-metric-agree}, we get that with probability at least $p$,
\eqb \label{eqn-square-metric-agree'}
\sup_{u,v\in S_r(0)} D_\hf (u,v) \leq \frac12 D_\hf (S_r(0) , \bdy S_{R r}(0)) 
\quad \text{and} \quad
D_\hf (u,v) = \wt D(u,v) ,\quad \forall u,v\in S_r(0) .
\eqe
By Lemma~\ref{lem-internal-conv}, a.s.\ the internal metrics of $D_\hf$ and $\wt D$ on the interior of $S_{Rr}(0)$ coincide. 
Hence~\eqref{eqn-square-metric-agree} implies that with probability at least $p$, $D_\hf(u,v)$ is equal to the infimum of the $D_\hf$-lengths of all continuous paths from $u$ to $v$ which are contained in the interior of $S_{Rr}(0)$, which (by the first condition in~\eqref{eqn-square-metric-agree}) is equal to the infimum of the $D_\hf$-lengths of all continuous paths from $u$ to $v$.
Since $p$ can be made arbitrarily close to 1 and $r$ can be made arbitrarily large, we get that a.s.\ $D_\hf$ is a length metric. 
 
To get the joint convergence~\eqref{eqn-lfpp-dyadic}, we first apply Lemma~\ref{lem-lfpp-tight-dyadic} and the Prokhorov theorem to get that the joint law of the metrics on the left side of~\eqref{eqn-lfpp-dyadic} is tight. Moreover any subsequential limit of these joint laws is a coupling of a continuous length metric $D_\hf$ on $\BB C$ and a length metric $D_{\hf,W}$ on $\ol W$ for each $W\in\mcl W$ which induces the Euclidean topology on $\ol W$.
We then apply Lemma~\ref{lem-internal-conv} to say that $D_{\hf ,W}(\cdot,\cdot;W) = D_\hf(\cdot,\cdot;W)$ for each $W\in\mcl W$.  
\end{proof}

\subsection{Weyl scaling}
\label{sec-weyl-scaling}

The following lemma will be used to check Axiom~\ref{item-metric-f}.

\begin{lem} \label{lem-weyl-scaling}
Let $\hf$ be a whole-plane GFF plus a bounded continuous function and consider a sequence $\ep_n \rta 0$ along which $ \frk a_{\ep_n}^{-1} D_\hf^{\ep_n}$ converges in law to some metric $ D_\hf$ w.r.t.\ the local uniform topology. 
Suppose we have, using the Skorokhod theorem, coupled so this convergence occurs a.s.
Then, a.s., for every sequence of bounded continuous functions $f^n : \BB C\rta \BB R$ such that $f^n$ converges to a bounded continuous function $f$ uniformly on compact subsets of $\BB C$, we have the local uniform convergence $D_{\hf+f^n}^{\ep_n} \rta e^{\xi f}\cdot D_\hf$, where here $D_{\hf+f^n}^\ep$ is defined as in~\eqref{eqn-lfpp} with $\hf+f^n$ in place of $\hf$ and $e^{\xi f}\cdot D_\hf$ is defined as in~\eqref{eqn-metric-f}. 
\end{lem} 

As a consequence of Lemma~\ref{lem-weyl-scaling}, if $\hf$ is a whole-plane GFF plus a bounded continuous function and $\ep_n \rta 0$ is a sequence along which $ \frk a_{\ep_n}^{-1} D_\hf^{\ep_n} \rta D_\hf$ in law, then whenever $ \hf'$ is another whole-plane GFF plus a bounded continuous function, we have $\frk a_{\ep_n}^{-1} D_{\hf'}^{\ep_n} \rta D_{\hf'}$ in law for some limiting metric $D_{\hf'}$. Furthermore, $(\hf , \hf' , D_\hf , D_{\hf'})$ can be coupled together in such a way that $\hf'- \hf$ is a bounded continuous function and $D_{\hf'} = e^{\xi(\hf' - \hf)} \cdot D_\hf$. 
Consequently, any subsequence along which $ \frk a_{\ep_n}^{-1} D_\hf^{\ep_n}$ converges in law gives us a way to define a metric associated with any whole-plane GFF plus a bounded continuous function.

\begin{proof}[Proof of Lemma~\ref{lem-weyl-scaling}]
Let $f_{\ep_n}^{*,n} = f^n*p_{\ep_n^2/2}$ be defined as in~\eqref{eqn-gff-convolve} with with $f^n$ in place of $h$. 
Then $f_{\ep_n}^{*,n} \rta f$ uniformly on compact subsets of $\BB C$.
By the definition~\eqref{eqn-lfpp} of LFPP, we have $D_{\hf+f^n}^{\ep_n} = e^{\xi f_{\ep_n}^{*,n}} \cdot D_\hf^{\ep_n}$. 

We now want to apply an argument as in the proof of~\cite[Lemma 7.1]{df-lqg-metric} to say that $D_{\hf+f^n}^{\ep_n}  \rta e^{\xi f}\cdot D_\hf$ w.r.t.\ the local uniform topology. That lemma only applies for metrics defined on squares, so we need to localize. We do this by means of Lemma~\ref{lem-square-bdy-dist}. 
By taking a limit as $\ep\rta 0$ in the estimate of Lemma~\ref{lem-square-bdy-dist}, then sending $p \rta 1$, we find that a.s.\ for each $r > 0$ and each $C > 1$, there exists $r'  = r'(r,C)> 0$ (random) such that
\eqb \label{eqn-square-bdy-weyl}
 \sup_{u,v\in S_r(0)} D_\hf (u,v) \leq \frac{1}{2C} D_\hf (S_r(0) , \bdy S_{r'}(0)) . 
\eqe
Furthermore, the uniform convergence $\frk a_{\ep_n}^{-1} D_\hf^{\ep_n} \rta D_\hf$, we get that~\eqref{eqn-square-bdy-weyl} is a.s.\ true with $\frk a_{\ep_n}^{-1} D_\hf^{\ep_n}$ in place of $D_\hf$ for large enough $n\in\BB N$, but with $C$ instead of $2C$. 
This implies that each path of near-minimal $D_\hf$-length between two points of $S_r(0)$ is contained in $S_{r'}(0)$, and the same is true with $\frk a_{\ep_n}^{-1} D_\hf^{\ep_n}$ in place of $D_\hf$ for large enough $n\in\BB N$.
If we choose $C > \sup_{n\in\BB N} \|f^n\|_\infty$, then from~\eqref{eqn-square-bdy-weyl} we deduce that each path of near-minimal $e^{\xi f} \cdot D_\hf$-length between two points of $S_r(0)$ is contained in $S_{r'}(0)$, and the same is true with $\frk a_{\ep_n}^{-1} D_{\hf + f^n}^{\ep_n}$ in place of $D_\hf$ for large enough $n\in\BB N$.
With these conditions in hand, the lemma now follows from the same proof as in~\cite[Lemma 7.1]{df-lqg-metric}. 
\end{proof}

\subsection{Tightness across scales}
\label{sec-lfpp-coord}

In this section we check that subsequential limits of LFPP satisfy Axiom~\ref{item-metric-coord}. 
For the statement, we note that we can take a subsequential limit of the joint laws of $(\hf ,\frk a_\ep^{-1} D_\hf^\ep)$ due to Lemma~\ref{lem-lfpp-tight} and the Prokhorov theorem.

\begin{lem} \label{lem-lfpp-coord}
Let $\hg$ be a whole-plane GFF normalized so that $\hg_1(0) = 0$. 
Let $(\hg , D_\hg)$ be any subsequential limit of the laws of the field/metric pairs $(\hg , \frk a_\ep^{-1} D_\hg^\ep)$.
There are deterministic constants $\{\frk c_r\}_{r\geq 0}$, depending on the law of $D_\hg$, such that the laws of the metrics $\{\frk c_r^{-1} e^{-\xi \hg_r(0)} D_\hg(r\cdot,r\cdot)\}_{r>0}$ are tight w.r.t.\ the local uniform topology. Furthermore, the closure of this set of laws w.r.t.\ the Prokhorov topology for probability measures on continuous functions $\BB C\times \BB C \rta [0,\infty)$ is contained in the set of laws on continuous metrics on $\BB C$. Finally, there exists  
$\Lambda > 1$ such that for each $\delta \in (0,1)$, 
\eqb \label{eqn-scaling-constant'}
\Lambda^{-1} \delta^\Lambda \leq \frac{\frk c_{\delta r}}{\frk c_r} \leq \Lambda \delta^{-\Lambda} ,\quad\forall r  > 0.
\eqe
\end{lem}

We first produce the scaling constants $\frk c_r$ appearing in Axiom~\ref{item-metric-coord}. 

\begin{lem} \label{lem-lfpp-constant}
Consider a sequence $\mcl E\subset (0,1)$ converging to zero along which $ \frk a_\ep^{-1} D_\hg^\ep$ converges in law to a limiting metric $ D_\hg$. 
For each $r > 0$, the limit 
\eqb \label{eqn-scaling-const-lim}
\frk c_r := \lim_{\mcl E\ni\ep \rta 0} \frac{r \frk a_{\ep /r}  }{ \frk a_{\ep } }
\eqe
exists and satisfies the relation~\eqref{eqn-scaling-constant'} for some choice of $\Lambda>1$ depending only on $\mcl E$ and $\gamma$. 
\end{lem}
\begin{proof}
Let $\hg^r := \hg(r\cdot) - \hg_r(0)$ be as in Lemma~\ref{lem-lfpp-scale}, so that $\hg^r \eqD \hg$. By our choice of subsequence $\mcl E$ and Lemma~\ref{lem-lfpp-scale},  
\eqb \label{eqn-scaled-conv}
\frk a_\ep^{-1}  D_{\hg^r}^{ \ep / r} = r^{-1} e^{-\xi \hg_r(0)} \frk a_\ep^{-1} D_\hg^\ep(r\cdot,r\cdot)  \xrta{\mcl E\ni\ep\rta0} r^{-1} e^{-\xi \hg_r(0)} D_\hg (r \cdot, r\cdot) 
\eqe
in law w.r.t.\ the local uniform topology on $\BB C\times \BB C$.
Let $m_r$ be the median distance between the left and right boundaries of $[0,1]^2$ w.r.t.\ the metric on the right side of~\eqref{eqn-scaled-conv}.
Since $\hg^r\eqD \hg$, 
\eqb \label{eqn-scaled-law}
\underbrace{\frk a_{\ep/r}^{-1}  D_\hg^{ \ep / r} }_{\text{tight}}
\eqD \frk a_{\ep/r}^{-1} D_{\hg^r}^{\ep/r} 
= \frac{\frk a_\ep}{  \frk a_{\ep/r} } \underbrace{ \frk a_\ep^{-1} D_{\hg^r}^{ \ep / r}  }_{\substack{\text{convergent} \\ \text{by~\eqref{eqn-scaled-conv}}}} .
\eqe
If we consider a subsequence $\mcl E'$ of $\mcl E$ along which the joint law of $\frk a_{\ep/r}^{-1}  D_\hg^{ \ep / r}$ and $\frk a_\ep^{-1} D_{\hg^r}^{ \ep / r}$ converges, then~\eqref{eqn-scaled-law} shows that along this subsequence, $\frk a_{\ep/r} / \frk a_\ep $ converges to some number $s_r(\mcl E')  > 0$ (we know the limit is strictly positive since the limits of $\frk a_{\ep/r}^{-1}  D_\hg^{ \ep / r} $ and $\frk a_\ep^{-1} D_{\hg^r}^{ \ep / r} $ are metrics).
By the definitions of $\frk a_\ep$ and of $m_r$ and Portmanteau's lemma, the median distance between the left and right boundaries of $[0,1]^2$ w.r.t.\ the metric on the left (resp.\ right) side of~\eqref{eqn-scaled-law} is 1 (resp.\ $m_r /  s_r(\mcl E')$). Hence $s_r(\mcl E') =  m_r$, i.e., the limit does not depend on the choice of subsequence $\mcl E'\subset\mcl E$.
This shows the convergence of $\frk a_{\ep/r} / \frk a_\ep$ along the subsequence $\mcl E$, which in turn implies the existence of the limit~\eqref{eqn-scaling-const-lim}. 
The bounds~\eqref{eqn-scaling-constant'} (in fact, substantially stronger bounds) are immediate from~\cite[Theorem 1, Equation (1.3)]{dddf-lfpp} and the fact the ratio of our $\frk a_\ep$ and the scaling factor $\lambda_\ep$ from~\cite{dddf-lfpp} is bounded above and below by deterministic, $\ep$-independent constants (see the proof of Lemma~\ref{lem-lfpp-tight-square}).
\end{proof}

\begin{proof}[Proof of Lemma~\ref{lem-lfpp-coord}]
Define $\frk c_r$ for $r > 0$ as in Lemma~\ref{lem-lfpp-constant}. Let $\hg^r := \hg(r\cdot) - \hg_r(0)$, as in Lemma~\ref{lem-lfpp-scale}, so that $\hg^r\eqD \hg$ and the metrics $D_{\hg^r}^{\ep/r}$ and $D_\hg^\ep$ are related as in~\eqref{eqn-lfpp-scale}. 
We know from Lemma~\ref{lem-lfpp-tight} that the laws of the metrics $\{\frk a_\ep^{-1} D_\hg^\ep\}_{ 0 <\ep < 1 }$ are tight, and every element of the closure of this set of laws is supported on continuous metrics on $\BB C$. It follows that the same is true for the laws of the metrics $\{\frk a_{\ep/r}^{-1} D_{\hg^r}^{\ep/r} \}_{0<\ep < r}$. By combining this with~\eqref{eqn-lfpp-scale}, we get that the laws of the metrics
\eqb \label{eqn-scaled-lfpp-tight}
e^{-\xi \hg_r(0)} \left(\frac{r \frk a_{\ep/r}}{\frk a_\ep} \right)^{-1} \frk a_{\ep }^{-1} D_\hg^{\ep } (r \cdot  ,  r \cdot )  
=    \frk a_{\ep/r}^{-1} D_{\hg^r}^{\ep/r} ,\quad \forall r > 0,\quad\forall \ep \in (0,r) 
\eqe
are tight and every element of the closure of this set of laws w.r.t.\ the Prokhorov topology is supported on continuous metrics on $\BB C$. 

Now consider a subsequence $\mcl E \subset (0,1)$ along which $(\hg, \frk a_{\ep}^{-1} D_\hg^{\ep}) \rta (\hg,D_\hg)$ in law.
By the definition~\eqref{eqn-scaling-const-lim} of $\frk c_r$,  
\eqbn
e^{-\xi \hg_r(0)} \left(\frac{r \frk a_{\ep/r}}{\frk a_\ep} \right)^{-1} \frk a_{\ep}^{-1} D_\hg^{\ep} (r \cdot  ,  r \cdot )         \rta e^{-\xi \hg_r(0)} \frk c_r^{-1} D_\hg(r\cdot,r\cdot) ,\quad\text{in law along $\mcl E$}. 
\eqen 
Therefore, the metrics $e^{-\xi \hg_r(0)} \frk c_r^{-1} D_\hg(r\cdot,r\cdot)$ for $r>0$ are all subsequential limits as $\ep\rta 0$ of the family of random metrics~\eqref{eqn-scaled-lfpp-tight}. It follows that the laws of the metrics $e^{-\xi \hg_r(0)} \frk c_r^{-1} D_\hg(r\cdot,r\cdot)$ are tight and every element of the closure of this set of laws is supported on continuous metrics on $\BB C$.  
\end{proof}

\subsection{Locality}
\label{sec-lfpp-local}

In this section, we will prove a variant of Axiom~\ref{item-metric-local} for subsequential limits of LFPP, restricted to the case of a whole-plane GFF (locality for a whole-plane GFF plus a continuous function will be checked in Section~\ref{sec-lfpp-msrble}). 
At this point, we have not yet established that such subsequential limits can be realized as measurable functions of the field, so we will actually check a somewhat different condition. In what follows, if $K\subset\BB C$ is closed we define the $\sigma$-algebra generated by $\hg|_K$ to be $\bigcap_{\delta > 0} \hg|_{B_\delta(K)}$. 
With this definition it makes sense to condition on $\hg|_K$. 
The following definitions first appeared in~\cite{local-metrics}. 
 
\begin{defn}[Local metric] \label{def-local-metric}
Let $U\subset \BB C$ be a connected open set and let $(\hg,D)$ be a coupling of a GFF on $U$ and a random continuous length metric on $U$.  We say that $D$ is a \emph{local metric} for $\hg$ if for any open set $V\subset U$, the internal metric $D(\cdot,\cdot;  V)$ is conditionally independent from the pair $(\hg  , D(\cdot,\cdot; U\setminus \ol V))$ given $\hg|_{\ol V}$.
\end{defn}

Definition~\ref{def-local-metric} is formulated in a slightly different way than~\cite[Definition 1.2]{local-metrics}; the equivalence of the definitions is proven in~\cite[Lemma 2.3]{local-metrics}. The following is~\cite[Definition 1.5]{local-metrics}.

\begin{defn}[Additive local metric] \label{def-additive-local}
Let $U\subset \BB C$ be a connected open set and let $(\hg,D )$ be a coupling of a GFF on $U$ and a random continuous length metric on $U$ which is local for $\hg$.
For $\xi \in\BB R$, we say that $D$ is \emph{$\xi$-additive} for $\hg$ if for each $z\in U$ and each $r> 0$ such that $B_r(z) \subset U$, the metric $ e^{-\xi \hg_r(z)} D $ is local for $\hg - \hg_r(z)$. 
\end{defn}

\begin{lem} \label{lem-lfpp-local}
Let $\hg$ be a whole-plane GFF.
Let $(\hg,D_\hg)$ be any subsequential limit of the laws of the pairs $(\hg,\frk a_\ep^{-1} D_\hg^\ep)$.
Then $D_\hg$ is a $\xi$-additive local metric for $\hg$. 
That is, suppose $z\in\BB C$ and $r>0$ and that $\hg$ is normalized so that the circle average $\hg_r(z)$ is zero. Also let $V\subset \BB C$ be an open set. Then the internal metric $D_\hg(\cdot,\cdot ; V)$ is conditionally independent from the pair $\left(\hg  , D_\hg(\cdot,\cdot;\BB C\setminus \ol V)\right)$ given $\hg|_{\ol V}$.
\end{lem}

There are two main difficulties in the proof of Lemma~\ref{lem-lfpp-local}.
\begin{enumerate}
\item The mollified GFF $\hg_\ep^*(z)$ of~\eqref{eqn-gff-convolve} does not exactly depend locally on $\hg$ (since the heat kernel $p_{\ep^2/2}(z,\cdot)$ does not have compact support), so the $D_\hg^\ep$-lengths of paths are not locally determined by $\hg$. \label{item-issue-discrete}
\item Conditional independence does not in general behave nicely under taking limits in law. \label{item-issue-cond}
\end{enumerate}
Difficulty~\ref{item-issue-discrete} will be resolved by means of the localization results for LFPP in Section~\ref{sec-localized-lfpp}. 
To resolve Difficulty~\ref{item-issue-cond}, we will use the Markov property of the GFF (see Lemma~\ref{lem-whole-plane-markov}) and Weyl scaling (Lemma~\ref{lem-weyl-scaling}) in order to reduce to working with metrics which are actually independent, not just conditionally independent. 
The use of the Markov property is the reason why we restrict to a whole-plane GFF, not a whole-plane GFF plus a bounded continuous function, in Lemma~\ref{lem-lfpp-local}.
 
For the proof of Lemma~\ref{lem-lfpp-local} we will need the following version of the Markov property of the whole-plane GFF, which is proven in~\cite[Lemma 2.2]{gms-harmonic}. We note that the statement of this Markov property is slightly more complicated than in the case of the zero-boundary GFF due to the need to fix the additive constant for $\hg$.

\begin{lem}[\!\!\cite{gms-harmonic}] \label{lem-whole-plane-markov}
Let $z\in\BB C$ and $r>0$ and let $\hg$ be a whole-plane GFF with the additive constant chosen so that $\hg_r(z) = 0$. 
For each open set $V\subset\BB C$ which is non-polar (i.e., Brownian motion started in $V$ a.s.\ hits $\bdy V$ in finite time), we have the decomposition
\eqb
\hg  =   \frk \hg + \rng \hg
\eqe
where $\frk \hg$ is a random distribution which is harmonic on $V$ and is determined by $\hg|_{\BB C\setminus V}$ and $\rng \hg$ is independent from $\frk \hg$ and has the law of a zero-boundary GFF on $V$ minus its average over $\bdy B_r(z) \cap V$. If $V$ is disjoint from $\bdy B_r(z)$, then $\rng \hg$ is a zero-boundary GFF and is independent from $\hg|_{\BB C\setminus V}$. 
\end{lem} 

 The following lemma will allow us to apply Lemma~\ref{lem-whole-plane-markov} to study $\hg|_{\BB C\setminus \ol V}$. 

\begin{lem} \label{lem-inside-circle}
It suffices to prove Lemma~\ref{lem-lfpp-local} in the case when $B_r(z) \subset V$.
\end{lem}
\begin{proof}
Assume that we have proven Lemma~\ref{lem-lfpp-local} in the case when $B_r(z) \subset V$. 
Fix $z_0 \in\BB C$ and $r_0 > 0$ such that $B_{r_0}(z_0) \subset V$ and assume that $\hg$ is normalized so that $\hg_{r_0}(z_0) = 0$.
By assumption, $D_\hg(\cdot,\cdot ; V)$ is conditionally independent from the pair $\left(\hg  , D_\hg(\cdot,\cdot;\BB C\setminus \ol V)\right)$ given $\hg|_{\ol V}$.

Now let $z \in \BB C$ and $r > 0$ and define $\wt \hg := \hg - \hg_r(z)$, so that $\wt \hg$ is a whole-plane GFF normalized so that $\wt \hg_r(z) = 0$. 
Lemma~\ref{lem-weyl-scaling} implies that $D_{\wt \hg}^\ep \rta e^{-\xi \hg_r(z)} D_\hg =: D_{\wt \hg}$ in law along the same subsequence for which $D_\hg^\ep \rta D_\hg$ in law, so $D_{\wt \hg}$ is unambiguously defined. We need to show that the conclusion of the first paragraph remains true with $(\wt \hg ,D_{\wt \hg})$ in place of $(\hg,D_\hg)$. 

The key fact which allows us to show this is that $\wt \hg_{r_0}(z_0) = - \hg_r(z)$. 
Since $B_{r_0}(z_0) \subset V$, this means that $\hg_r(z) \in \sigma\left( \wt \hg|_{\ol V} \right)$. 
In particular, $\hg|_{\ol V} = \wt \hg|_{\ol V} + \hg_r(z)$ is determined by $\wt \hg|_{\ol V}$. 
Therefore, our assumption implies that $D_\hg(\cdot,\cdot ; V)$ is conditionally independent from the pair $\left(\hg  , D_\hg(\cdot,\cdot;\BB C\setminus \ol V)\right)$ given $\wt \hg|_{\ol V}$ (instead of just $\hg|_{\ol V}$).

We have $D_{\wt \hg}(\cdot,\cdot ; V) = e^{-\xi \hg_r(z)} D_\hg(\cdot,\cdot ; V)$, so $D_{\wt \hg}(\cdot,\cdot ; V)$ is determined by $\wt \hg|_{\ol V}$ and $D_\hg(\cdot,\cdot ; V)$. 
Similarly, $D_{\wt \hg}(\cdot,\cdot;\BB C\setminus \ol V)$ is determined by $\wt \hg|_{\ol V}$ and $D_\hg(\cdot,\cdot ; \BB C\setminus \ol V)$. 
Obviously, $\hg$ and $\wt \hg$ determine the same information. 
Therefore, $D_{\wt \hg}(\cdot,\cdot ; V)$ is conditionally independent from the pair $\left(\wt \hg  , D_{\wt \hg}(\cdot,\cdot;\BB C\setminus \ol V)\right)$ given $\wt \hg|_{\ol V}$, as required.
\end{proof}

\begin{figure}[ht!]
\begin{center}
\includegraphics[scale=1]{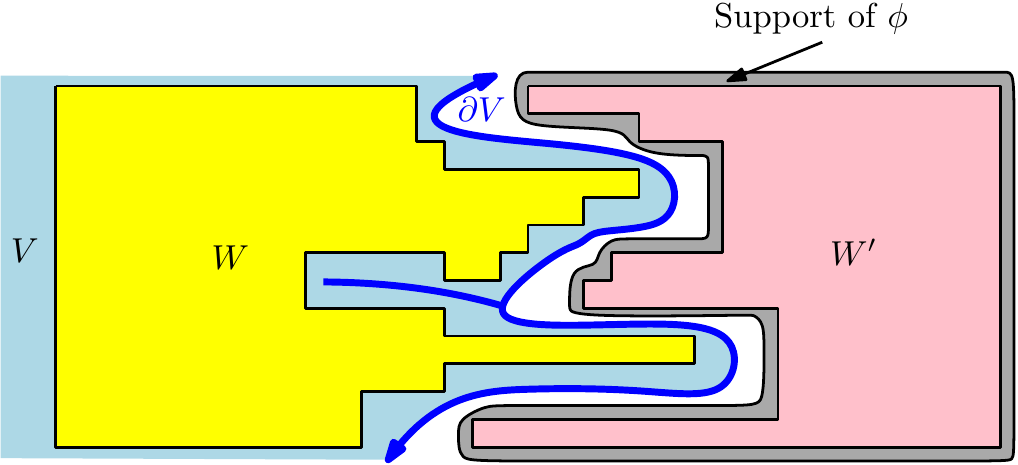} 
\caption{\label{fig-locality} Illustration of the sets used in the proof of Lemma~\ref{lem-lfpp-local}. The set $\phi^{-1}(1)$ is not shown; it contains the closure of the pink set $W'$ and is contained in the grey set $\op{supp}\phi$. 
}
\end{center}
\end{figure}

\begin{proof}[Proof of Lemma~\ref{lem-lfpp-local}] 
\noindent\textit{Step 1: reductions.}
By Lemma~\ref{lem-localized-approx}, for any sequence of $\ep$'s tending to zero along which $(\hg, \frk a_\ep^{-1} D_\hg^\ep) \rta (\hg, D_\hg)$ in law, we also have $(\hg, \frk a_\ep^{-1} \wh D_\hg^\ep) \rta (\hg , D_\hg)$ in law. This allows us to work with $\wh D_\hg^\ep$ instead of $D_\hg^\ep$ throughout the proof. The reason why we want to do this is the locality property~\eqref{eqn-localized-property} of $\wh D_\hg^\ep$. 

The statement of the lemma is vacuous if $\ol V = \BB C$, so we can assume without loss of generality that $\ol V\not=\BB C$, which implies that $\BB C\setminus \ol V$ is non-polar. 
By Lemma~\ref{lem-inside-circle}, we can also assume without loss of generality that $B_r(z) \subset V$. 
These assumptions together with Lemma~\ref{lem-whole-plane-markov} applied with $\BB C\setminus \ol V$ in place of $V$ allows us to write 
\eqb
\hg|_{\BB C\setminus \ol V} = \frk \hg + \rng \hg 
\eqe
where $\frk \hg$ is a random harmonic function on $\BB C\setminus \ol V$ which is determined by $\hg|_{\BB C\setminus \ol V}$ and $\rng \hg$ is a zero-boundary GFF in $\BB C\setminus \ol V$ which is independent from $\hg|_{\BB C\setminus \ol V}$. 
\medskip

\noindent\textit{Step 2: independence for LFPP.}
We want to apply the convergence of internal metrics given in Lemma~\ref{lem-lfpp-tight}, so we fix dyadic domains (Definition~\ref{def-dyadic-domain}) $W,W'$ with $\ol W\subset V$ and $\ol W' \subset \BB C\setminus \ol V$ (we will eventually let $W$ and $W'$ increase to all of $V$ and $\BB C\setminus \ol V$, respectively). 
Let $\phi$ be a deterministic, smooth, compactly supported bump function which is identically equal to 1 on a neighborhood of $\ol W'$ and which vanishes outside of a compact subset of $\BB C\setminus \ol V$. 
See Figure~\ref{fig-locality} for an illustration of these objects.

The restrictions of the fields $\hg - \phi \frk \hg $ and $\rng \hg $ to the set $\phi^{-1}(1)\supset \ol W'$ are identical. 
By the locality property~\eqref{eqn-localized-property} of $\wh D_\hg^\ep$, if $\ep > 0$ is small enough that $B_\ep(W') \subset \phi^{-1}(1)$, then the $\ep$-LFPP metric for $\hg-\phi\frk \hg $ satisfies 
\eqb \label{eqn-internal-metric-truncate}
  \wh D_{\hg-\phi\frk \hg }^\ep (\cdot,\cdot ; \ol W')  \in \sigma\left(\rng \hg \right) .
\eqe
Similarly, for small enough $\ep > 0$ the metric $\wh D_\hg^\ep(\cdot,\cdot ; \ol W)$ is a.s.\ determined by $\hg|_V$. 
Since $\hg |_V$ and $\rng \hg $ are independent, we obtain
\eqb \label{eqn-internal-metric-ind}
\left( \hg|_V , \frk a_\ep^{-1} D_\hg^\ep(\cdot,\cdot; \ol W) \right) \quad \text{and} \quad
\left( \rng\hg , \frk a_\ep^{-1} D_{\hg - \phi\frk \hg}^\ep(\cdot,\cdot ; \ol W')\right) \quad \text{are independent} .
\eqe
\medskip

\noindent\textit{Step 3: passing to the limit.}
We now want to pass the independence~\eqref{eqn-internal-metric-ind} through to the (subsequential) scaling limit. 
To this end, consider a sequence $\mcl E$ of positive $\ep$'s tending to zero along which $(\hg , \frk a_{\ep}^{-1} \wh D_\hg^{\ep}) \rta (\hg,D_\hg)$ in law. 
By possibly passing to a further deterministic subsequence, we can arrange that in fact $(\hg ,   \frk \hg , \frk a_{\ep}^{-1} \wh D_\hg^{\ep}) \rta (\hg  , \frk \hg , D_\hg)  $ in law along $\mcl E$, where here the second coordinate is given the local uniform topology on $\BB C\setminus \ol V$. 
By the analog of Lemma~\ref{lem-weyl-scaling} with $\wh D_\cdot^\ep$ in place of $D_\cdot^\ep$ (which is proven in an identical manner), if we set $D_{\hg-\phi\frk \hg} = e^{-\xi \phi \frk \hg} \cdot D_\hg$, then along this same subsequence we have the convergence of joint laws
\eqb \label{eqn-truncated-metric-conv}
\left(\hg ,   \frk \hg , \frk a_{\ep}^{-1} \wh D_\hg^{\ep} ,  \frk a_{\ep}^{-1} \wh D_{\hg - \phi\frk \hg}^{\ep} \right)
 \rta \left( \hg  , \frk \hg , D_\hg , D_{\hg - \phi\frk \hg} \right)  .
\eqe

By assertion~\ref{item-lfpp-dyadic} of Lemma~\ref{lem-lfpp-tight}, applied once to each of $\hg$ and  $\hg-\phi \frk \hg$, by possibly replacing $\mcl E$ with a further deterministic subsequence we can find a coupling $(\hg , D_\hg ,   D_{\hg, W} , D_{\hg - \phi\frk \hg , W'})$ of $(\hg , D_\hg)$ with length metrics on $\ol W$ and $\ol W'$, respectively, which induce the Euclidean topology and which satisfy
\eqb \label{eqn-internal-identify}
D_{\hg,W}(\cdot,\cdot;W) = D_\hg(\cdot,\cdot;W) \quad \text{and} \quad
D_{\hg-\phi\frk \hg ,W'}(\cdot,\cdot;W') = D_{\hg-\phi\frk \hg}(\cdot,\cdot;W')  
\eqe
such that the following is true. Along $\mcl E$, we have the convergence of joint laws
\allb \label{eqn-internal-joint-law-conv}
&\left(\hg ,   \frk \hg , \frk a_{\ep}^{-1} \wh D_\hg^{\ep} ,  \frk a_{\ep}^{-1} \wh D_{\hg - \phi\frk \hg}^{\ep} , \frk a_{\ep}^{-1} \wh D_\hg^{\ep}(\cdot,\cdot; \ol W) , \frk a_{\ep}^{-1} \wh D_{\hg-\phi\frk \hg}^{\ep}(\cdot,\cdot; \ol W') \right) \notag\\
& \qquad \quad \rta 
\left(  \hg, , \frk \hg , D_\hg , D_{\hg - \phi\frk \hg} , D_{\hg, W} , D_{  \hg-\phi\frk \hg   , W'} \right)
\alle
where the last two coordinates are given the uniform topology on $\ol W\times \ol W$ and on $\ol W'\times \ol W'$, respectively. 
Since independence is preserved under convergence in law, we obtain from~\eqref{eqn-internal-metric-ind} and~\eqref{eqn-internal-joint-law-conv} that $(\hg|_V , D_{\hg,W} )$ and $ (\rng \hg , D_{\hg - \phi \frk \hg , W'})$ are independent. By~\eqref{eqn-internal-identify}, this means that
\eqb \label{eqn-limit-metric-ind}
(\hg|_V , D_\hg(\cdot,\cdot;W) ) \quad \text{and} \quad (\rng \hg , D_{\hg - \phi \frk \hg}(\cdot,\cdot;W') ) \quad \text{are independent}. 
\eqe 
\medskip

\noindent\textit{Step 4: adding back in the harmonic part.} 
By~\eqref{eqn-limit-metric-ind}, $D_\hg(\cdot,\cdot;W)$ is conditionally independent from $ (\rng \hg , D_{\hg - \phi \frk \hg}(\cdot,\cdot;W') )$ given $\hg|_V$. 
We now argue that $(\hg , D_\hg(\cdot,\cdot;W'))$ is a  measurable function of $ (\rng \hg , D_{\hg - \phi \frk \hg}(\cdot,\cdot;W') )$ and $\hg|_V$, so that  $D_\hg(\cdot,\cdot;W)$ is conditionally independent from $(\hg , D_\hg(\cdot,\cdot;W'))$ given $\hg|_V$.
Indeed, by Lemma~\ref{lem-weyl-scaling}, a.s.\ $D_\hg(\cdot,\cdot ; W') = (e^{\xi \phi \frk \hg } \cdot D_{\hg - \phi\frk \hg})(\cdot,\cdot ; W')$. 
Hence $D_\hg(\cdot,\cdot;W')$ is a measurable function of $\frk \hg \in \sigma(\hg|_{\ol V})$ and $D_{\hg - \phi\frk \hg}(\cdot,\cdot;W')$.
Since $\hg|_{\BB C\setminus \ol V} =  \rng \hg + \frk \hg$, we get that $\hg$ is a measurable function of $\rng \hg $ and $ \hg|_{\ol V} $. 
It therefore follows that $D_\hg(\cdot,\cdot;W )$ is conditionally independent from $(\hg ,  D_\hg(\cdot,\cdot;W'))$ given $\hg|_{\ol V}$. 
Letting $W$ increase to $V$ and $W'$ increase to $\BB C\setminus \ol V$ now concludes the proof.
\end{proof}

\subsection{Measurability}
\label{sec-lfpp-msrble}

We have not yet established that subsequential limits of LFPP can be realized as measurable functions of the corresponding field.
We will accomplish this in this subsection using a result from~\cite{local-metrics}. 

\begin{lem} \label{lem-lfpp-msrble}
Let $\hg$ be a whole-plane GFF normalized so that $\hg_1(0) =0$ and let $(\hg,D_\hg)$ be any subsequential limit of the laws of the pairs $(\hg,\frk a_\ep^{-1} D_\hg^\ep)$.
Then $D_\hg$ is a.s.\ determined by $\hg$. 
In particular, $\frk a_\ep^{-1} D_\hg^\ep \rta D_\hg$ in probability along the given subsequence.
\end{lem}

The following theorem is a special case of~\cite[Corollary 1.8]{local-metrics}. 

\begin{thm}[\!\!\cite{local-metrics}] \label{thm-bilip-msrble}
There is a universal constant $p\in (0,1)$ such that the following is true. Let $\xi\in\BB R$, let $\hg$ be a whole-plane GFF normalized so that $\hg_1(0) = 0$, and let $(\hg,D )$ be a coupling of $\hg$ with a random continuous length metric satisfying the following properties.
\begin{enumerate}
\item $D$ is a $\xi$-additive local metric for $\hg$ (Definition~\ref{def-additive-local}).
\item Condition on $\hg$ and let $D$ and $\wt D$ be conditionally i.i.d.\ samples from the conditional law of $D$ given $\hg$.
There is a deterministic constant $C> 0$ such that 
 \eqb \label{eqn-bilip}
\BB P\left[  \sup_{u,v \in \bdy B_r(z)} \wt D\left(u,v; B_{2r}(z) \setminus \ol{B_{r/2}(z)} \right) \leq C D(\bdy B_{r/2}(z) , \bdy B_r(z) ) \right] \geq p    ,\quad \forall z\in \BB C ,\quad \forall r > 0   .
\eqe 
\end{enumerate}
Then $D$ is a.s.\ determined by $\hg$. 
\end{thm}

\begin{proof}[Proof of Lemma~\ref{lem-lfpp-msrble}]
Let $p \in (0,1)$ be as in Theorem~\ref{thm-bilip-msrble}.
Lemma~\ref{lem-lfpp-local} implies that $D_\hg$ is a $\xi$-additive local metric for $\hg$. 
Lemma~\ref{lem-lfpp-coord} along with the translation invariance of the law of $\hg$, modulo additive constant, implies that there exists $C > 0$ (depending only on the choice of subsequence) such that for each $z\in\BB C$ and each $r >0$,
\eqbn
\BB P\left[ D(\bdy B_{r/2}(z) , \bdy B_r(z) ) \geq C^{-1/2} \frk c_r e^{\xi \hg_r(z)} \right] \geq \frac{1-p}{2} \quad\text{and} 
\eqen
\eqbn
  \BB P\left[ \sup_{u,v \in \bdy B_r(z)} D_\hg\left(u,v; B_{2r}(z) \setminus \ol{B_{r/2}(z)} \right) \leq C^{1/2} \frk c_r e^{\xi \hg_r(z)}  \right] \geq \frac{1-p}{2} .
\eqen
This implies that~\eqref{eqn-bilip} holds for two conditionally independent samples from the conditional law of $D_\hg$ given $\hg$. 
Hence the criteria of Theorem~\ref{thm-bilip-msrble} are satisfied, so $D_\hg$ is a.s.\ determined by $\hg$.
The last statement follows from Lemma~\ref{lem-in-prob}.
\end{proof}

\begin{proof}[Proof of Theorem~\ref{thm-lfpp-axioms}]
\noindent\textit{Step 1: Defining a $D_\hf$ for a whole-plane GFF plus a bounded continuous function.}
Let $\hg$ be a whole-plane GFF normalized so that $\hg_1 (0) = 0$. 
Lemma~\ref{lem-lfpp-tight} implies that for any sequence of $\ep$'s tending to zero, there is a subsequence $\ep_n\rta 0$ along which $(\hg ,D_{\hg}^{\ep_n}) \rta (\hg,D_{\hg})$ in law.
By Lemma~\ref{lem-lfpp-msrble}, $D_{\hg}$ is a.s.\ determined by $\hg$ and $D_{\hg}^{\ep_n} \rta D_{\hg}$ in probability.
Hence every deterministic subsequence of the $\ep_n$'s admits a further deterministic subsequence $\ep_{n_k}$ along which $D_{\hg}^{\ep_{n_k}} \rta D_{\hg}$ a.s.
By Lemma~\ref{lem-weyl-scaling}, it is a.s.\ the case that for every bounded continuous function $f : \BB C\rta \BB R$ simultaneously, we have $D_{\hg+f}^{\ep_{n_k}} \rta e^{\xi f}\cdot D_{\hg}$. 
We define $D_{\hg +f} := e^{\xi f}\cdot D_{\hg}$. 
Then $D_{\hg+f}$ is a.s.\ determined by $\hg+f$ and $D_{\hg+f}^{\ep_n} $ converges in probability to $D_{\hg+f}$.

This gives us a measurable function $\hf \mapsto D_\hf$ from distributions to continuous metrics on $\BB C$ which is a.s.\ defined whenever $\hf$ is a whole-plane GFF plus a bounded continuous function: in particular, $D_\hf$ is the a.s.\ limit of $D_\hf^{\ep_{n_k}}$. 
With this definition of $D$, Axiom~\ref{item-metric-length} holds with $\hf$ constrained to be a whole-plane GFF plus a bounded continuous function since we know that the limiting metric in the setting of Lemma~\ref{lem-lfpp-tight} is a length metric.
By the preceding paragraph, Axiom~\ref{item-metric-f} holds for this definition of $D$ and with $f$ constrained to be bounded. 
It is immediate from the definition of LFPP that also Axiom~\ref{item-metric-translate} holds.
By Lemma~\ref{lem-lfpp-coord}, also Axiom~\ref{item-metric-coord} holds.
\medskip

\noindent\textit{Step 2: locality for a whole-plane GFF plus a bounded continuous function.}
Axiom~\ref{item-metric-local} in the case of a whole-plane GFF is immediate from Lemma~\ref{lem-lfpp-local} now that we know that $D_{\hg}$ is a.s.\ determined by $ \hg$.
We now prove Axiom~\ref{item-metric-local} in the case when $\hf$ is a whole-plane GFF plus a bounded continuous function.
Indeed, let $V\subset \BB C$ be open and let $O \subset O' \subset V$ be open and bounded with $\ol O \subset O'$ and $\ol O' \subset V$.
Let $u,v\in O$ be deterministic. We will show that  
\eqb \label{eqn-close-pts-msrble}
D_\hf(u,v) \BB 1_{\left\{ D_\hf(u,v)  <  D_\hf(u,\bdy O') \right\} } \in \sigma\left( \hf|_V \right) .
\eqe
Since $(u,v) \mapsto D_\hf(u,v)$ is a.s.\ continuous,~\eqref{eqn-close-pts-msrble} implies that in fact $\hf|_V$ a.s.\ determines the random function $O\ni (u,v) \mapsto D_\hf(u,v) \BB 1_{\left\{ D_\hf(u,v) <  D_\hf(u,\bdy O') \right\} }$. 
Since $\ol O$ is a compact subset of $O'$, $O$ can be covered by finitely many sets of the form $\{v\in O : D_\hf(u,v)  <  D_\hf(u,\bdy O')\}$ for points $u\in O$.
By the definition of the internal metric $D_\hf(\cdot,\cdot;O)$, this shows that $\hf|_V$ a.s.\ determines $D_\hf(\cdot,\cdot;O)$. Letting $O$ increase to all of $V$ then shows that $\hf|_V$ a.s.\ determines $D_\hf(\cdot,\cdot;V)$. 

To prove~\eqref{eqn-close-pts-msrble}, note that if we define the localized LFPP metric $\wh D_\hf^{\ep_{n}}$ as in~\eqref{eqn-localized-lfpp}, then by Lemma~\ref{lem-localized-approx} we have $\frk a_{\ep_n}^{-1} \wh D_\hf^{\ep_n}(u,v) \rta D_\hf(u,v)$ and $\frk a_{\ep_n}^{-1} \wh D_\hf^{\ep_n}(u,\bdy O') \rta D_\hf(u,\bdy O')$ in probability.
Therefore, 
\eqb \label{eqn-close-pts-conv}
\frk a_{\ep_n}^{-1} \wh D_\hf^{\ep_n} (u,v) \BB 1_{\left\{ \wh D_\hf^{\ep_n}(u,v) <  \wh D_\hf^{\ep_n}(u,\bdy O') \right\} } 
\rta  D_\hf(u,v) \BB 1_{\left\{ D_\hf(u,v) < D_\hf(u,\bdy O') \right\} }  ,
\quad \text{in probability}. 
\eqe
By~\eqref{eqn-localized-property} and since $\ol O' \subset V$, the random variable on the left side of~\eqref{eqn-close-pts-conv} is a.s.\ determined by $\hf|_V$ for large enough $n\in\BB N$.
Thus~\eqref{eqn-close-pts-msrble} holds.  
\medskip

\noindent\textit{Step 3: extending to unbounded continuous function.}
We will now extend the definition of $D$ to the case of a whole-plane GFF plus an unbounded continuous function and check that the axioms remain true. 
To this end, let $\hg$ be a whole-plane GFF and let $f$ be a possibly random unbounded continuous function.
If $V \subset \BB C$ is open and bounded and $\phi$ is a smooth compactly supported bump function which is identically equal to 1 on $V$, then $\phi f$ is bounded so we can define the metric $D_{\hg+f}^V := D_{\hg+\phi f}(\cdot,\cdot;V)$. 
By Axiom~\ref{item-metric-local} in the case of a whole-plane GFF plus a bounded continuous function, this metric is a.s.\ determined by $(\hg+\phi f)|_V = (\hg+f)|_V$, in a manner which does not depend on $\phi$.
We now define the $D_{\hg+f}$-length of any continuous path $P$ in $\BB C$ to be the $D_{\hg+f}^V$-length of $P$, where $V\subset\BB C$ is a bounded open set which contains $P$
The definition does not depend on the choice of $V$. 
We define $D_{\hg+f}(z,w)$ for $z,w\in\BB C$ to be the infimum of the $D_{\hg+f}$-lengths of continuous paths from $z$ to $w$. 
Then $D_{\hg+f}$ is a length metric on $\BB C$ which is a.s.\ determined by $D_{\hg+f}$ and which satisfies $D_{\hg+f}(\cdot,\cdot;V) = D_{\hg+f}^V$ for each bounded open set $V\subset\BB C$. 

With the above definition, it is immediate from the case of a whole-plane GFF plus a bounded continuous function that the axioms in the definition of a weak $\gamma$-LQG metric are satisfied to the mapping $\hf\mapsto D_{\hf}$, which is a.s.\ defined whenever $\hf$ is a whole-plane GFF plus a continuous function.  
\end{proof}

\section{Proofs of quantitative properties of weak LQG metrics}
\label{sec-a-priori-estimates}

In this section we will prove the estimates stated in Section~\ref{sec-main-results}. 
Actually, in many cases we will prove a priori stronger estimates which are required to be uniform across different Euclidean scales.
With what we know now, these estimates are not implied by the estimates stated in Section~\ref{sec-main-results} since we are working with a weak $\gamma$-LQG metric so we have tightness across scales instead of exact scale invariance.
However, \emph{a posteriori}, once it is proven that a weak $\gamma$-LQG metric satisfies the coordinate change formula~\eqref{eqn-lqg-metric-coord} (which will be done in~\cite{gm-uniqueness}, building on the results in the present paper), the estimates in this section are equivalent to the estimates in Section~\ref{sec-main-results}. 
Throughout this section, $D$ denotes a weak LQG metric and $h$ denotes a whole-plane GFF normalized so that $h_1(0) =0$. 

\subsection{Estimate for the distance between sets}
\label{sec-perc-estimate}

The goal of this subsection is to prove the following more precise version of Theorem~\ref{thm-two-set-dist0} which is required to be uniform across scales.  
For the statement, we recall the scaling constants $\frk c_r$ for $r>0$ from Axiom~\ref{item-metric-coord}.

\begin{prop} \label{prop-two-set-dist}
Let $U \subset \BB C$ be an open set (possibly all of $\BB C$) and let $K_1,K_2\subset U$ be connected, disjoint compact sets which are not singletons. 
For each $\BB r  >0$, it holds with superpolynomially high probability as $A\rta \infty$, at a rate which is uniform in the choice of $\BB r$, that 
\eqb \label{eqn-two-set-dist}
 A^{-1}\frk c_{\BB r} e^{\xi h_{\BB r}(0)} \leq D_h(\BB r K_1,\BB r K_2 ; \BB r U) \leq A \frk c_{\BB r} e^{\xi h_{\BB r}(0)} .  
\eqe
\end{prop}

\begin{figure}[ht!]
\begin{center}
\includegraphics[scale=1]{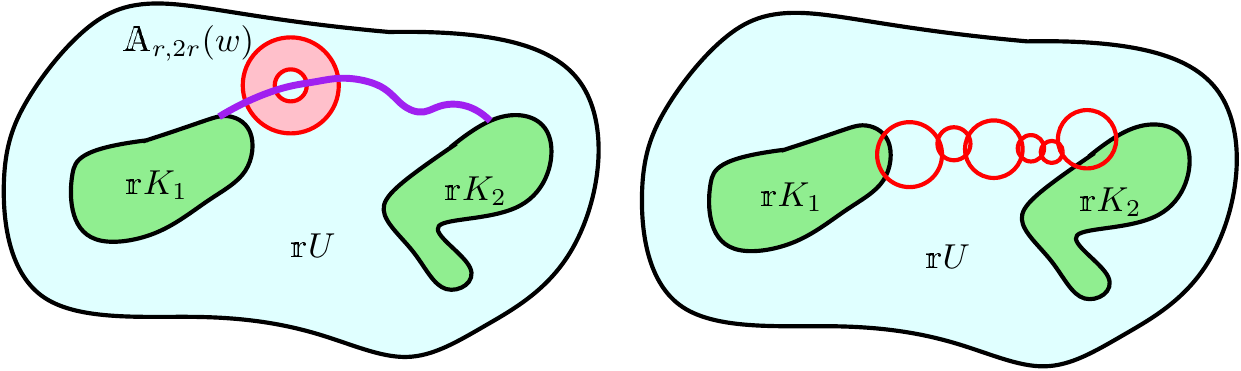} 
\caption{\label{fig-two-set-dist} \textbf{Left:} To prove the lower bound in Proposition~\ref{prop-two-set-dist}, we cover $\BB r U$ by balls $B_{r/2}(w)$ such that the $D_h$-distance across the annulus $\BB A_{r,2r}(w)$ is bounded below. Each path from $\BB r K_1$ to $\BB r (K_2\cup \bdy U)$ must cross at least one of these annuli (one such path is shown in purple). 
\textbf{Right:} To prove the upper bound in Proposition~\ref{prop-two-set-dist}, we cover $\BB r U$ by balls $B_{r/2}(w)$ for which the $D_h$-diameter of the circle $\bdy B_r(w)$ is bounded above, then string together a path of such circles from $K_1$ to $K_2$. 
}
\end{center}
\end{figure}
 
We now explain the idea of the proof of Proposition~\ref{prop-two-set-dist}; see Figure~\ref{fig-two-set-dist} for an illustration. 
Using Axiom~\ref{item-metric-coord} and a general ``local independence" lemma for the GFF (see Lemma~\ref{lem-annulus-iterate} below), we can, with extremely high probability, cover $\BB r U$ by small Euclidean balls $B_{r/2}(w)$ such that $r\in [\ep^2 \BB r , \ep \BB r]$ and the $D_h$-distance across the annulus $\BB A_{r,2r}(w)$ is bounded below by a constant times $\frk c_r e^{\xi h_r(w)}$. 
Any path from $\BB r K_1$ to $\BB r K_2$ must cross at least one of these annuli. 
This leads to a lower bound for $D_h(\BB rK_1 , \BB r K_2 ; \BB r U)$ in terms of 
\eqb \label{eqn-two-set-dist-terms}
\inf_{r \in [\ep^2 \BB r , \ep\BB r]} \frk c_r \quad \text{and} \quad \inf_{r \in [\ep^2\BB r , \ep\BB r]} \inf_{w\in \BB r U} e^{\xi h_r(w)} .
\eqe
The first infimum in~\eqref{eqn-two-set-dist-terms} can be bounded below by a positive power of $\ep$ times $\frk c_{\BB r}$ by~\eqref{eqn-scaling-constant}. 
By being a little more careful about how we choose the balls $B_{r/2}(w)$, the second term in~\eqref{eqn-two-set-dist-terms} can be reduced to an infimum over finitely many values of $r$ and $w$, which can then be bounded below by a positive power of $\ep$ times $e^{\xi h_{\BB r}(0)}$ using the Gaussian tail bound and a union bound (see Lemma~\ref{lem-circle-avg-tail}). Choosing $\ep$ to be an appropriate power of $A$ then concludes the proof.

The upper bound in~\eqref{eqn-two-set-dist} is proven similarly, but in this case we instead cover $U$ by balls $B_{r/2}(w)$ for which the $D_h$-diameter of the circle $\bdy B_r(w)$ is bounded above by a constant times $\frk c_r e^{\xi h_r(w)}$, then ``string together" a collection of such circles to get a path from $\BB r K_1$ to $\BB r K_2$ whose $D_h$-length is bounded above. The hypothesis that $K_1$ and $K_2$ are connected and are not singletons allows us to force some of the circles in this path to intersect $K_1$ and $K_2$.

We now explain how to cover $U$ by Euclidean balls with the desired properties. 
For $C>1$, $z\in\BB C$, and $r>0$, let $E_r(z;C)$ be the event that 
\eqb \label{eqn-annulus-event}
 \sup_{u,v\in \bdy B_r(z)} D_h\left(u,v; \BB A_{r/2,2r}(z) \right) \leq C \frk c_r e^{\xi h_r(0)}   
 \quad \text{and} \quad
  D_h\left( \bdy B_r(z) , \bdy B_{2r}(z)\right) \geq  C^{-1} \frk c_r e^{\xi h_r(0)} .
\eqe

\begin{lem} \label{lem-good-annulus-all}
For each $\nu > 0$ and each $M > 0$, there exists $C  = C(\nu,M) >1$ such that for each $\BB r > 0$, it holds with probability at least $1-O_\ep(\ep^M)$ as $\ep\rta 0$, at a rate which is uniform in $\BB r$, that the following is true.
For each $z\in B_{\BB r \ep^{-M}}(0) $, there exists $w\in B_{\BB r \ep^{-M}}(0) \cap \left(\frac{\ep^{1+\nu}  \BB r}{4} \BB Z^2 \right)  $ and $r\in [\ep^{1+\nu} \BB r , \ep \BB r] \cap \{2^{-k} \BB r\}_{k\in\BB N}$ such that $E_r(w;C)$ occurs and $z\in B_{ \BB r \ep^{1+\nu}/2}(w)$.  
\end{lem}

We will prove Lemma~\ref{lem-good-annulus-all} using the following result from~\cite{local-metrics}, which in turn follows from the near-independence of the GFF across disjoint concentric annuli. See in particular~\cite[Lemma 3.1]{local-metrics}. 

\begin{lem} \label{lem-annulus-iterate}
Fix $0 < s_1<s_2 < 1$. Let $\{r_k\}_{k\in\BB N}$ be a decreasing sequence of positive numbers such that $r_{k+1} / r_k \leq s_1$ for each $k\in\BB N$ and let $\{E_{r_k} \}_{k\in\BB N}$ be events such that $E_{r_k} \in \sigma\left( (h-h_{r_k}(0)) |_{\BB A_{s_1 r_k , s_2 r_k}(0)  } \right)$ for each $k\in\BB N$. 
For $K\in\BB N$, let $N(K)$ be the number of $k\in [1,K]_{\BB Z}$ for which $E_{r_k}$ occurs.  
\item For each $a > 0$ and each $b\in (0,1)$, there exists $p = p(a,b,s_1,s_2) \in (0,1)$ and $c = c(a,b,s_1,s_2) > 0$ such that if \label{item-annulus-iterate-high}
\eqb \label{eqn-annulus-iterate-prob}
\BB P\left[ E_{r_k}  \right] \geq p , \quad \forall k\in\BB N  ,
\eqe 
then 
\eqb \label{eqn-annulus-iterate}
\BB P\left[ N(K)  < b K\right] \leq c e^{-a K} ,\quad\forall K \in \BB N. 
\eqe  
\end{lem}

\begin{proof}[Proof of Lemma~\ref{lem-good-annulus-all}]
By Axioms~\ref{item-metric-translate} and~\ref{item-metric-coord} (also see \eqref{eq:tight}), for each $p \in (0,1)$ there exists $C >1$ such that for every $z\in\BB C$ and $r>0$, $\BB P\left[ E_r(z;C) \right] \geq p$. 
By the locality of $D_h$ and Axiom~\ref{item-metric-f}, the event $E_r(z;C)$ is determined by $(h-h_{3r}(z))|_{\BB A_{r/2,2r}(z)}$. 
We can therefore apply Lemma~\ref{lem-annulus-iterate} to a logarithmic (in $\ep$) number of values of $r\in [ \ep^{1+\nu} \BB r , \ep  \BB r] \cap \{2^{-k} \BB r\}_{k\in\BB N}$ to find that for any choice of $\nu > 1$ and $\wt M> 0$, there is a large enough $C = C(\nu,\wt M)>1$ such that the following is true. For each $z\in \BB C$ it holds with probability at least $1 - O_\ep(\ep^{\wt M})$ that $E_r(z;C)$ occurs for at least one value of $r\in [\ep^{1+\nu} \BB r , \ep \BB r] \cap \{2^{-k} \BB r\}_{k\in\BB N}$.
We now conclude the proof by choosing $\wt M $ to be sufficiently large, in a manner depending only on $\nu , M$, and taking a union bound over all $z\in B_{\BB r \ep^{-M}}(0) \cap \left(\frac{\ep^{1+\nu}  \BB r}{4} \BB Z^2 \right)$. 
\end{proof}

The occurrence of the event $E_r(z;C)$ allows us to bound distances in terms of circle averages and the scaling coefficients $\frk c_r$. The $\frk c_r$'s can be bounded using~\eqref{eqn-scaling-constant}. To bound the circle averages, we will need the following lemma. 

\begin{lem} \label{lem-circle-avg-tail}
For each $\nu > 0$, each $q >  2+2\nu $, each $R >0$, and each $\BB r > 0$, it holds with probability $1 - O_\ep\left( \ep^{\frac{q^2}{ 2(1+\sqrt \nu)^2} - 2 -2\nu} \right) $, at a rate depending only on $q$ and $R$ (not on $\BB r$) that 
\eqb
  \sup\left\{ |h_r(w) - h_{\BB r}(0)| : w\in B_{R\BB r}(0) \cap \left(\frac{\ep^{1+\nu} \BB r}{4} \BB Z^2 \right) ,\, r\in [\ep^{1+\nu} \BB r , \ep \BB r] \right\} \leq q \log\ep^{-1}  .
\eqe 
\end{lem}
\begin{proof}
Fix $s\in (0,q)$ to be chosen momentarily.
For each $w\in B_{R\BB r}(0)$, the random variable $t\mapsto h_{e^{-t} \ep \BB r}(w) - h_{\ep \BB r}(w)$ is a standard linear Brownian motion~\cite[Section 3]{shef-kpz}.
We can therefore apply the Gaussian tail bound to find that 
\eqb \label{eqn-gaussian-tail-inc}
\BB P\left[\sup_{r\in  [\ep^{1+\nu} \BB r , \ep  \BB r]} |h_r(w) - h_{\ep\BB r}(w)| \leq s \log\ep^{-1} \right] \geq 1 - O_\ep\left( \ep^{ s^2/(2\nu)} \right) .
\eqe
The random variables $h_{\ep \BB r}(w) - h_{\BB r}(0)$ for $w\in B_{R\BB r}(0)$ are centered Gaussian with variance $\log\ep^{-1} + O_\ep(1)$. 
Applying the Gaussian tail bound again therefore gives
\eqb \label{eqn-gaussian-tail-start}
\BB P\left[  |  h_{\ep\BB r}(w)  - h_{\BB r}(0)| \leq (q-s) \log\ep^{-1} \right] \geq 1 - O_\ep\left( \ep^{ (q-s)^2 /2 } \right) .
\eqe
Combining~\eqref{eqn-gaussian-tail-inc} and~\eqref{eqn-gaussian-tail-start} applied with $s = q\sqrt{\nu}/(1+\sqrt \nu)$ shows that for $w\in B_{R\BB r}(0)$, 
\eqb \label{eqn-gaussian-tail-combine}
\BB P\left[\sup_{r\in  [\ep^{1+\nu} \BB r , \ep  \BB r]} |h_r(w)   - h_{\BB r}(0)| \leq q \log\ep^{-1} \right] \geq 1 - O_\ep\left( \ep^{\frac{q^2}{2(1+\sqrt \nu)^2}} \right) .
\eqe 
We now conclude by means of a union bound over $O_\ep(\ep^{-2-2\nu})$ values of $w\in B_{R\BB r}(0) \cap \left(\frac{\ep^{1+\nu} \BB r}{4} \BB Z^2 \right)$.
\end{proof}

\begin{proof}[Proof of Proposition~\ref{prop-two-set-dist}]
Throughout the proof, all $O(\cdot)$ and $o(\cdot)$ errors are required to be uniform in the choice of $\BB r$. We also impose the requirement that $U$ is bounded --- we will explain at the very end of the proof how to get rid of this requirement.

Set $\nu = 1$, say, and fix a large $M  > 1$, which we will eventually send to $\infty$. 
Let $C  = C(1,M) > 1$ be chosen as in Lemma~\ref{lem-good-annulus-all} and for $\ep \in (0,1)$ and $\BB r > 0$, let $F_{\BB r}^\ep$ be the event of Lemma~\ref{lem-good-annulus-all} for this choice of $\nu,M,C$, so that $\BB P[F_{\BB r}^\ep] = 1 - O_\ep(\ep^M)$. We will eventually take $\ep = A^{-b/\sqrt M}$ for a small constant $b>0$, so $\ep^M$ will be a large negative power of $A$ (i.e., the power goes to $\infty$ as $M\rta \infty$) but $\ep^{\sqrt M}$ will be a fixed negative power of $A$ (which does not go to $\infty$ when $M\rta\infty$). 

By Lemma~\ref{lem-circle-avg-tail} (applied with $\nu = 1$ and $q = 2\sqrt 2\sqrt{4+M}$), it holds with probability $1-O_\ep(\ep^M)$ that
\eqb \label{eqn-use-circle-avg-tail}
\sup\left\{ |h_r(w) - h_{\BB r}(0)| : w\in B_{ \BB r}(\BB r U) \cap \left(\frac{\ep^2 \BB r}{4} \BB Z^2 \right) ,\, r\in [\ep^2 \BB r , \ep \BB r] \right\} 
\leq 2\sqrt 2\sqrt{4+M} \log\ep^{-1}  .
\eqe
Henceforth assume that $F_{\BB r}^\ep$ occurs and~\eqref{eqn-use-circle-avg-tail} holds, which happens with probability $1-O_\ep(\ep^M)$. 
We will now prove lower and upper bounds for $D_h\left( \BB r K_1 , \BB r K_2  ; \BB r U \right) $ in terms of $\ep$.
\medskip

\noindent\textit{Step 1: lower bound.}
By the definition of $F_{\BB r}^\ep$, if $\ep $ is sufficiently small, depending on $K_1,K_2,U$, then each path from $\BB r K_1$ to $\BB r (K_2 \cup \bdy U)$ must cross from $\bdy B_r(w)$ to $\bdy B_{2r}(w)$ for some $w \in B_{\ep\BB r}(\BB r U) \cap \left(\frac{\ep^2 \BB r}{4} \BB Z^2 \right)$ and $r\in [\ep^2 \BB r , \ep \BB r] \cap \{2^{-k}\BB r\}_{k\in\BB N}$ for which $E_r(w;C)$ occurs.
Therefore, 
\allb \label{eqn-two-set-dist-lower}
D_h\left( \BB r K_1 , \BB r K_2 \right) 
&\geq \inf\left\{ C^{-1} \frk c_r e^{\xi h_r(w)} : w \in B_{\ep\BB r}(\BB r U) \cap \left(\frac{\ep^2  \BB r}{4} \BB Z^2 \right) ,\: r\in [\ep^2 \BB r , \ep \BB r] \right\} \quad \text{(by~\eqref{eqn-annulus-event})} \notag \\
&\geq C^{-1} \ep^{\xi  2\sqrt 2\sqrt{4+M} } e^{\xi h_{\BB r}(0)} \inf\left\{\frk c_r : r\in  [\ep^2 \BB r , \ep \BB r] \right\} \quad \text{(by~\eqref{eqn-use-circle-avg-tail})} \notag \\
&\geq \Lambda^{-1} \ep^{\xi 2\sqrt 2\sqrt{4+M} + 2 \Lambda   + o_\ep(1)  } \frk c_{\BB r} e^{\xi h_{\BB r}(0)} \quad \text{(by~\eqref{eqn-scaling-constant})} . 
\alle
\medskip

\noindent\textit{Step 2: upper bound.}
It is easily seen from the definition of $F_{\BB r}^\ep$ (see Lemma~\ref{lem-connected} below) that if $\ep$ is sufficiently small (depending only on $K_1, K_2,$ and $U$) then the union of the circles $\bdy B_r(w)$ for $w \in B_{\ep\BB r}(\BB r U) \cap \left(\frac{\ep^2  \BB r}{4} \BB Z^2 \right)$ and $r\in [\ep^2 \BB r , \ep \BB r] \cap \{2^{-k}\BB r\}_{k\in\BB N}$ such that $E_r(w;C)$ occurs contains a path from $\BB r K_1$ to $\BB r K_2$ which is contained in $\BB r U$.  
The total number of such circles is at most $\ep^{-4- o_\ep(1)}$, so by the triangle inequality,
\allb \label{eqn-two-set-dist-upper}
D_h\left( \BB r K_1 , \BB r K_2 ; \BB r  U  \right) 
&\leq \ep^{-4 - o_\ep(1)} \sup\left\{ C  \frk c_r e^{\xi h_r(w)} : w \in B_{\ep\BB r}(\BB r U) \cap \left(\frac{\ep^{2}  \BB r}{4} \BB Z^2 \right) ,\: r\in [\ep^{2} \BB r , \ep \BB r] \right\} \quad \text{(by~\eqref{eqn-annulus-event})} \notag \\
&\leq  \ep^{-4 - \xi 2\sqrt 2\sqrt{4+M}- o_\ep(1)} e^{\xi h_{\BB r}(0)} \sup\left\{\frk c_r : r\in  [\ep^2 \BB r , \ep \BB r] \right\} \quad \text{(by~\eqref{eqn-use-circle-avg-tail})} \notag \\
&\leq \Lambda \ep^{-4 - \xi 2\sqrt 2\sqrt{4+M}    - 2\Lambda -    o_\ep(1)} \frk c_{\BB r} e^{\xi h_{\BB r}(0)} \quad \text{(by~\eqref{eqn-scaling-constant})} . 
\alle
\medskip

\noindent\textit{Step 3: choosing $\ep$.}
The bounds~\eqref{eqn-two-set-dist-lower} and~\eqref{eqn-two-set-dist-upper} hold with probability $1-O_\ep(\ep^M)$. 
Given $A > 0$, we now choose $\ep  =  A^{-b/\sqrt M}$, where $b > 0$ is a small constant (depending only on $\xi,\Lambda$) chosen so that the right side of~\eqref{eqn-two-set-dist-lower} is at least $A^{-1} \frk c_{\BB r} e^{\xi h_{\BB r}(0)}$ and the right side of~\eqref{eqn-two-set-dist-upper} is at most $A \frk c_{\BB r} e^{\xi h_{\BB r}(0)}$. 
Then~\eqref{eqn-two-set-dist-lower} and~\eqref{eqn-two-set-dist-upper} imply that 
\eqb \label{eqn-two-set-dist-end}
\BB P\left[ D_h(\BB r K_1,\BB r K_2) \geq A^{-1} \frk c_{\BB r} e^{\xi h_{\BB r}(0)} , \:  D_h(\BB r K_1,\BB r K_2 ; \BB r U)  \leq \frk c_{\BB r} e^{\xi h_{\BB r}(0)} \right] \geq 1 - O_A(A^{-b \sqrt M}) .
\eqe
If $U'$ is a possibly unbounded open subset of $\BB C$ with $U \subset U'$, then $D_h(\BB r K_1, \BB r K_2) \leq D_h(\BB r K_1,\BB r K_2 ; \BB r U') \leq D_h(K_1,K_2 ; \BB r U)$. 
Since $M$ can be made arbitrarily large, we now obtain~\eqref{eqn-two-set-dist} (with $U$ possibly unbounded) from~\eqref{eqn-two-set-dist-end}. 
\end{proof}

The following lemma was used in the proof of the upper bound of Proposition~\ref{prop-two-set-dist}. 

\begin{lem} \label{lem-connected}
Assume that we are in the setting of Proposition~\ref{prop-two-set-dist}, with $U$ bounded. 
Define the event $F_{\BB r}^\ep$ as in the proof of Proposition~\ref{prop-two-set-dist}.
For small enough $\ep >0$ (depending on $ K_1,K_2 , U$), on $F_{\BB r}^\ep$, the union of the circles $\bdy B_r(w)$ for $w \in B_{\ep\BB r}(\BB r U) \cap \left(\frac{\ep^2  \BB r}{4} \BB Z^2 \right)$ and $r\in [\ep^2 \BB r , \ep \BB r] \cap \{2^{-k}\BB r\}_{k\in\BB N}$ such that $E_r(w;C)$ occurs contains a path from $\BB r K_1$ to $\BB r K_2$ which is contained in $\BB r U$.  
\end{lem}
\begin{proof}   
Throughout the proof we assume that $F_{\BB r}^\ep$ occurs. 
By the definition of $F_{\BB r}^\ep$ and since $U$ is connected, if $\ep$ is chosen so be sufficiently small then the union of the balls $B_r(w)$ for $w,r$ as in the lemma statement contains a path from $\BB r K_1$ to $K_2$ which is contained in $U$.  
Let $\mcl B$ be a sub-collection of these balls which is minimal in the sense that $ \bigcup_{B\in\mcl B} B$ contains a path from $\BB r K_1$ to $\BB r K_2$ in $\BB r U$ and no proper sub-collection of the balls in $\mcl B$ has this property. 
Choose a path $P$ from $\BB r K_1$ to $\BB r K_2$ in $(\BB r U) \cap \bigcup_{B\in\mcl B} B$.
 
We first observe that $\bigcup_{B\in\mcl B} B$ is connected. Indeed, if this set had two proper disjoint open subsets, then each would have to intersect $ P $ (by minimality) which would contradict the connectedness of $ P $. 
Furthermore, by minimality, no ball in $\mcl B$ is properly contained in another ball in $\mcl B$. 

We claim that $\bigcup_{B\in\mcl B} \bdy B$ is connected. Indeed, if this were not the case then we could partition $\mcl B = \mcl B_1\sqcup \mcl B_2$ such that $\mcl B_1$ and $\mcl B_2$ are non-empty and $\bigcup_{B\in\mcl B_1} \bdy B$ and $\bigcup_{B\in\mcl B_2} \bdy B$ are disjoint.  
By the minimality of $\mcl B$, it cannot be the case that any ball in $\mcl B_2$ is contained in $ \bigcup_{B\in\mcl B_1}  B$.
Furthermore, since $\bigcup_{B\in\mcl B_1} \bdy B$ and $\bigcup_{B\in\mcl B_2} \bdy B$ are disjoint, it cannot be the case that any ball in $\mcl B_2$ intersects both $\bigcup_{B\in\mcl B_1} B$ and  $\BB C\setminus \bigcup_{B\in\mcl B_1}  B$ (otherwise, such a ball would have to intersect the boundary of some ball in $\mcl B_1$). 
Therefore, $\bigcup_{B\in\mcl B_1}  B$ and $\bigcup_{B\in\mcl B_2} \bdy B$ are disjoint. 
Since no element of $\mcl B_1$ can be contained in $\bigcup_{B\in\mcl B_2}  B$, we get that $\bigcup_{B\in\mcl B_1}  B$ and $\bigcup_{B\in\mcl B_2}  B$ are disjoint.
This contradicts the connectedness of $\bigcup_{B\in\mcl B} B$, and therefore gives our claim. 

Since $P$ is a path from $\BB r K_1$ to $\BB r K_2$ and each of $\BB r K_1$ and $\BB r K_2$ is connected and not a single point, if $\ep < \frac12 (\op{diam}(K_1) \wedge \op{diam}(K_2) )$, then the boundaries of the balls in $\mcl B$ which contain the starting and endpoint points of $P$ must intersect $K_1$ and $K_2$, respectively. 
Hence for such an $\ep$, $\bigcup_{B\in\mcl B} \bdy B$ contains a path from $\BB r K_1$ to $\BB r K_2$, as required. 
\end{proof}

\subsection{Asymptotics of the scaling constants}
\label{sec-metric-scaling}
 
The goal of this section is to prove Theorem~\ref{thm-metric-scaling}. 
We will accomplish this by comparing $D_h$-distances to a variant of the Liouville first passage percolation (LFPP) which we now define.

For $\ep \in (0,1)$ and $U \subset\BB C$, we view $U\cap (\ep \BB Z^2)$ as a graph with adjacency defined by
\eqb\label{eq:graph}
\text{$z,w\in U\cap (\ep\BB Z^2)$ are connected by an edge if and only if $|z-w| \in \{\ep ,\sqrt 2\ep \}$}.  
\eqe
Note that this differs from the standard nearest-neighbor graph structure in that we also include the diagonal edges. 
We define the \emph{discretized $\ep$-LFPP metric with parameter $\xi$} on $U$ by
\eqb \label{eqn-lfpp-approx}
  \wt D_h^\ep(z,w ; U) 
:= \min_{\pi : z\rta w } \sum_{j=0}^{|\pi|}  e^{\xi h_\ep(\pi(j))} ,\quad\forall z,w\in U\cap (\ep \BB Z^2) ,
\eqe
where the minimum is over all paths $\pi : [0,|\pi|]_{\BB Z} \rta U\cap (\ep \BB Z^2)$ from $z$ to $w$ in $U\cap (\ep \BB Z^2)$ (the tilde is to distinguish this from the variant of LFPP defined in~\eqref{eqn-lfpp}).  

Recall that $\BB S = (0,1)^2$ denotes the open Euclidean unit square. 
Below, we will show, using Proposition~\ref{prop-two-set-dist} and a union bound over a polynomial number of $\delta\BB r \times \delta\BB r$ squares contained in $\BB r \BB S$, that with high probability,
\eqb \label{eqn-lfpp-relation}
\frk c_{ \BB r} =  \delta^{o_\delta(1)} \frk c_{\delta \BB r} \times \left( \text{$\wt D_h^{\delta\BB r}$ distance between two sides of $\BB r \BB S$} \right) .
\eqe
The reason why discretized LFPP comes up in this estimate is the circle average term $e^{\xi h_{\BB r}(0)}$ in Proposition~\ref{prop-two-set-dist}.
We know that the $\wt D_h^{\delta\BB r}$ distance across the square $\BB r\BB S$ is of order $\delta^{-\xi Q + o_\delta(1)}$, uniformly in $\BB r$, by the results of~\cite{dg-lqg-dim} (see Lemma~\ref{lem-lfpp-dist} just below). 
Hence~\eqref{eqn-lfpp-relation} leads to $\frk c_{\delta \BB r} =  \delta^{\xi Q + o_\delta(1)} \frk c_{ \BB r}$, as required. 

For a square $S\subset \BB C$, we write $\bdy_{\op{L}}^\ep S$ and $\bdy_{\op{R}}^\ep S$ for the set of leftmost (resp.\ rightmost) vertices of $S\cap (\ep\BB Z^2)$. 

\begin{lem} \label{lem-lfpp-dist}
Fix $\zeta \in (0,1)$. For $\BB r>0$, it holds with probability tending to 1 as $\delta \rta 0$, uniformly in the choice of $\BB r$, that 
\eqb
\wt D_h^{\delta\BB r}\left( \bdy_{\op{L}}^{\delta\BB r} (\BB r \BB S) , \bdy_{\op{R}}^{\delta\BB r} (\BB r \BB S) ; \BB r\BB S  \right) 
\in \left[ \delta^{- \xi Q  +\zeta}  e^{\xi h_{\BB r}(0)}  , \delta^{-\xi Q -\zeta}  e^{\xi h_{\BB r}(0)} \right] .
\eqe
\end{lem}
\begin{proof}
We first reduce to the case when $\BB r = 1$. Indeed, by the scale and translation invariance of the law of $h$, modulo additive constant, we have $h(\BB r\cdot) - h_{\BB r}(0) \eqD h$. Moreover, from the definition~\eqref{eqn-lfpp-approx} it is easily seen that
\eqb
\wt D_{h(\BB r\cdot) - h_{\BB r}(0)}^\delta\left(\cdot,\cdot ;   \BB S \right) = e^{-\xi h_{\BB r}(0)} \wt D_h^{\delta\BB r} \left( \cdot,\cdot ;\BB r\BB S \right) .
\eqe
Hence $e^{-\xi h_{\BB r}(0)} \wt D_h^{\delta\BB r} \left( \cdot,\cdot ;\BB r\BB S \right) \eqD \wt D_h^\delta(\cdot,\cdot;\BB S)$, so we only need to prove the lemma when $\BB r =1$, i.e., we need to show that with probability tending to 1 as $\delta \rta 0$, we have
\eqb \label{eqn-lfpp-dist-show}
\wt D_h^{\delta }\left( \bdy_{\op{L}}^{\delta } \BB S , \bdy_{\op{R}}^{\delta } \BB S ; \BB S  \right) = \delta^{-\xi Q + o_\delta(1) } .
\eqe

This follows from the LFPP distance exponent computation in~\cite{dg-lqg-dim}. 
To be more precise, \cite[Theorem 1.5]{dg-lqg-dim} shows that for continuum LFPP defined using the circle average process of the GFF, as in~\eqref{eqn-lfpp}, the $\delta$-LFPP distance between the left and right boundaries of $\BB S$ is of order $\delta^{1 - \xi Q + o_\delta(1)}$ with probability tending to 1 as $\delta \rta 0$. 
Combining this with~\cite[Lemma 3.7]{dg-lqg-dim} shows that the same is true for continuum LFPP defined using the white-noise approximation $\{\wh h_\delta\}_{\delta > 0}$, as defined in~\cite[Equation (3.1)]{dg-lqg-dim}, in place of the circle average process. 
The same argument as in the proof of~\cite[Proposition 3.16]{dg-lqg-dim} then shows that~\eqref{eqn-lfpp-dist-show} holds if we replace the circle average by the white-noise approximation in the definition of $\wt D_h^\delta$ (here we note that the definition of discretized LFPP in~\cite[Equation (3.32)]{dg-lqg-dim} has an extra factor of $\delta$ as compared to~\eqref{eqn-lfpp-approx}, which is why we get $\delta^{-\xi Q + o_\delta(1)}$ instead of $\delta^{1-\xi Q+ o_\delta(1)}$). 
The desired formula~\eqref{eqn-lfpp-dist-show} now follows by combining this with the uniform comparison of $h_\delta$ and $\wh h_\delta$ from~\cite[Lemma 3.7]{dg-lqg-dim}. 
\end{proof}

For the proof of Theorem~\ref{thm-metric-scaling} (and at several later places in this section) we will use the following terminology.

\begin{defn}[Distance around an annulus] \label{def-around-annulus}
For a set $A\subset\BB C$ with the topology of a an annulus, we define the \emph{$D_h$-distance around $A$} to be the infimum of the $D_h$-lengths of the paths in $A$ which disconnect the inner and outer boundaries of $A$. 
\end{defn}  
  
\begin{proof}[Proof of Theorem~\ref{thm-metric-scaling}] 
\noindent\textit{Step 1: estimates for $D_h$.}
For $z\in \ep\BB Z^2$, we write $S_z^\ep$ for the square of side length $\ep$ centered at $z$ and $B_\ep(S_z^\ep)$ for the $\ep$-neighborhood of this square. 
Fix $\zeta\in (0,1)$. By Proposition~\ref{prop-two-set-dist} and a union bound over all $z \in (\BB r  \BB S ) \cap (\delta \BB r \BB Z^2)$, it holds with superpolynomially high probability as $\delta \rta 0$ that (in the terminology of Definition~\ref{def-around-annulus}) 
\eqb  \label{eqn-square-around}
\left(\text{$D_h$-distance around $B_{\delta \BB r}(S_z^{\delta\BB r}) \setminus S_z^{\delta\BB r}$} \right) 
\leq \delta^{-\zeta} \frk c_{\delta \BB r} e^{\xi h_{\delta\BB r}(z)}  ,\quad
\forall z \in (\BB r  \BB S ) \cap (\delta \BB r \BB Z^2) .
\eqe
Similarly, it holds with superpolynomially high probability as $\delta\rta 0$ that
\eqb  \label{eqn-square-across}
D_h\left( S_z^{\delta\BB r} , \bdy B_{\delta\BB r}(S_z^{\delta\BB r}) \right)
\geq \delta^{ \zeta} \frk c_{\delta \BB r} e^{\xi h_{\delta\BB r}(z)}  ,\quad
\forall z \in (\BB r  \BB S ) \cap (\delta \BB r \BB Z^2) .
\eqe
Henceforth assume that~\eqref{eqn-square-around} and~\eqref{eqn-square-across} both hold. 
\medskip 

\noindent\textit{Step 2: lower bound for $\frk c_{\delta \BB r} / \frk c_{\BB r}  $.}
Let $\pi : [0,|\pi|]_{\BB Z} \rta  (\BB r \BB S) \cap (\delta\BB r \BB Z^2)$ be a path in $(\BB r \BB S) \cap (\delta\BB r \BB Z^2)$ (with the graph structure defined by \eqref{eq:graph}) from $\bdy_{\op{L}}^{\delta\BB r} (\BB r \BB S)$ to $\bdy_{\op{R}}^{\delta\BB r} (\BB r \BB S)$ for which the sum in~\eqref{eqn-lfpp-approx} equals $\wt D_h^{\delta\BB r}\left( \bdy_{\op{L}}^{\delta\BB r} (\BB r \BB S) , \bdy_{\op{R}}^{\delta\BB r} (\BB r \BB S)  ; \BB r \BB S \right)$.  
For each $j  \in [0,|\pi|]_{\BB Z}$, let $P_j$ be a path in $B_{\delta \BB r}(S_{\pi(j)}^{\delta\BB r}) \setminus S_{\pi(j)}^{\delta\BB r}$ which disconnects the inner and outer boundaries of  $B_{\delta \BB r}(S_{\pi(j)}^{\delta\BB r}) \setminus S_{\pi(j)}^{\delta\BB r}$ and whose $D_h$-length is at most $2 \delta^{-\zeta} \frk c_{\delta \BB r} e^{\xi h_{\delta\BB r}(z)}$. Such a path exists by~\eqref{eqn-square-around}. 

We have $P_j \cap P_{j-1} \not=\emptyset$ for each $j\in [0,|\pi|]_{\BB Z}$, so the union of the $P_j$'s is connected and contains a path between the left and right boundaries of $\BB r\BB S$. 
Therefore, the triangle inequality implies that
\allb \label{eqn-lfpp-lower}
D_h\left( \BB r \bdy_{\op{L}} \BB S , \BB r \bdy_{\op{R}} \BB S \right) 
\leq \sum_{j=0}^{|\pi|} \left(\text{$D_h$-length of $P_j$}\right)
&\leq  2\delta^{-\zeta} \frk c_{\delta \BB r}  \sum_{j=0}^{|\pi|} e^{\xi h_{\delta \BB r}(0)}  \notag \\
&= 2\delta^{-\zeta}  \frk c_{\delta \BB r}  \wt D_h^{\delta\BB r}\left( \bdy_{\op{L}}^{\delta\BB r} (\BB r \BB S) , \bdy_{\op{R}}^{\delta\BB r} (\BB r \BB S) ; \BB r\BB S  \right) .
\alle
By Axiom~\ref{item-metric-coord}, the left side of~\eqref{eqn-lfpp-lower} is at least $\delta^{\zeta} \frk c_{\BB r} e^{\xi h_{\BB r}(0)}$ with probability tending to 1 as $\delta \rta 0$, uniformly in $\BB r$. 
By Lemma~\ref{lem-lfpp-dist}, the right side of~\eqref{eqn-lfpp-lower} is at most $\delta^{-\xi Q - 2\zeta  }  \frk c_{\delta \BB r} e^{\xi h_{\BB r}(0)}$ with probability tending to 1 as $\delta \rta 0$, uniformly in $\BB r$. 
Combining these relations and sending $\zeta \rta 0$ shows that $\frk c_{\BB r} \leq \delta^{-\xi Q - o_\delta(1)} \frk c_{\delta \BB r}$, as desired. 
\medskip

\noindent\textit{Step 3: upper bound for $\frk c_{\delta \BB r} / \frk c_{\BB r} $.}
Let $P : [0,|P|] \rta \BB S$ be a path between the left and right boundaries of $\BB r \BB S$ with $D_h$-length at most $2D_h\left( \BB r \bdy_{\op{L}} \BB S , \BB r \bdy_{\op{R}} \BB S ; \BB r\BB S \right) $. 
We will use $P$ to construct a path in $(\BB r \BB S) \cap (\delta\BB r \BB Z^2)$ from $\bdy_{\op{L}}^{\delta\BB r} (\BB r \BB S)$ to $\bdy_{\op{R}}^{\delta\BB r} (\BB r \BB S)$ for which the sum in~\eqref{eqn-lfpp-approx} can be bounded above.

To this end, let $\tau_0 = 0$ and let $z_0 \in (\BB r \BB S) \cap (\delta\BB r \BB Z^2)$ be chosen so that $P(0) \in S_{z_0}^{\delta\BB r}$. 
Inductively, suppose $j\in \BB N$, a time $\tau_{j-1} \in [0,|P|]$, and a point $z_{j-1} \in (\BB r \BB S) \cap (\delta\BB r \BB Z^2)$ have been defined in such a way that $P(\tau_{j-1}) \in S_{z_{j-1}}^{\delta\BB r}$.
Let $\tau_j$ be the first time after $\tau_{j-1}$ at which $P$ exits $B_{\delta\BB r}(S_{z_{j-1}}^{\delta\BB r})$, if such a time exists, and otherwise set $\tau_j = |P|$. 
Let $z_j \in  (\BB r \BB S) \cap (\delta\BB r \BB Z^2)$ be chosen so that $P(\tau_j) \in S_{z_j}^{\delta\BB r}$.
Let $J$ be the smallest $j\in\BB N$ for which $\tau_j = |P|$, and note that $P(|P|) \in S_{z_j}^{\delta\BB r}$. 

Successive squares $S_{z_{j-1}}^{\delta \BB r}$ and $S_{z_j}^{\delta\BB r}$ necessarily share a vertex. Hence $z_{j-1}$ and $z_j$ lie at $(\BB r \BB S) \cap (\delta\BB r \BB Z^2)$-graph distance 1 from one another, so $\pi(j) := z_j$ for $j\in [0,J]_{\BB Z}$ is a path from $\bdy_{\op{L}}^{\delta\BB r} (\BB r \BB S)$ to $\bdy_{\op{R}}^{\delta\BB r} (\BB r \BB S)$ in $(\BB r \BB S) \cap (\delta\BB r \BB Z^2)$. 
 
We will now bound $\sum_{j=0}^J  e^{\xi  h_{\delta\BB r}(\pi(j))}$.  
For each $j\in [1,J]_{\BB Z}$, the path $P$ crosses between the inner and outer boundaries of $B_{\delta\BB r}(S_{z_{j-1}}^{\delta\BB r}) \setminus S_{z_{j-1}}^{\delta\BB r}$ between time $\tau_{j-1}$ and time $\tau_j$.  
By~\eqref{eqn-square-across}, for each $j\in [1,J]_{\BB Z}$, 
\allb \label{eqn-square-pair-dist}
D_h\left(P(\tau_{j-1}) , P(\tau_{j}) \right) 
 \geq \delta^\zeta \frk c_{\delta\BB r} e^{\xi h_{\delta\BB r}(\pi( j))} .
\alle 
Using~\eqref{eqn-square-pair-dist} and the definition of $P$, we therefore have 
\allb \label{eqn-lfpp-upper}
\sum_{j=0}^J  e^{\xi  h_{\delta\BB r}(\pi(j))}
&\leq \delta^{-\zeta} \frk c_{\delta\BB r}^{-1}     \sum_{j=0}^J D_h\left(P(\tau_{j-1}) , P(\tau_{j}) \right)  \notag\\
&\leq \delta^{-\zeta}\frk c_{\delta\BB r}^{-1} D_h\left( \BB r \bdy_{\op{L}} \BB S , \BB r \bdy_{\op{R}} \BB S \right)  .
\alle
By Axiom~\ref{item-metric-coord}, the right side of~\eqref{eqn-lfpp-upper} is at most $\delta^{-2\zeta} \frk c_{\delta \BB r}^{-1} \frk c_{\BB r} e^{\xi h_{\BB r}(0)}$
with probability tending to 1 as $\delta \rta 0$, uniformly in $\BB r$. 
By Lemma~\ref{lem-lfpp-dist}, the left side of~\eqref{eqn-lfpp-lower} is at least $\delta^{-\xi Q  - \zeta}   e^{\xi h_{\BB r}(0)}$ with probability tending to 1 as $\delta \rta 0$, uniformly in $\BB r$. 
Combining these relations and sending $\zeta \rta 0$ shows that $\frk c_{\delta \BB r}^{-1} \frk c_{\BB r} \geq \delta^{-\xi Q - o_\delta(1)} $.
\end{proof} 

Theorem~\ref{thm-metric-scaling} has the following useful corollary.
 
\begin{lem} \label{lem-infinite-dist}
Let $h$ be a whole-plane GFF normalized so that $h_1(0) = 0$. 
Almost surely, for every compact set $K\subset\BB C$ we have $\lim_{r \rta\infty} D_h(K,\bdy B_r(0)) = \infty$. 
In particular, every closed, $D_h$-bounded subset of $\BB C$ is compact. 
\end{lem}
\begin{proof}
By tightness across scales (Axiom~\ref{item-metric-coord}), there exists $a > 0$ such that for each $r > 0$, $\BB P\left[ D_h(B_r(0) , B_{2r}(0))  \geq a \frk c_r e^{\xi h_r(0)} \right] \geq 1/2$.
By the locality of $D_h$ (Axiom~\ref{item-metric-local}) and since $\sigma\left(\bigcap_{r > 0} h|_{\BB C\setminus B_r(0)} \right)$ is trivial, a.s.\ there are infinitely many $k \in \BB N$ for which $D_h(B_{2^k}(0) , B_{2^{k+1}}(0))  \geq a \frk c_{2^k} e^{\xi h_{2^k}(0)} $. By Theorem~\ref{thm-metric-scaling}, $\frk c_r = r^{\xi Q + o_r(1)}$. Since $t\mapsto h_{e^t}(0)$ is a standard linear Brownian motion~\cite[Section 3.1]{shef-kpz}, we get that a.s.\ $\lim_{r\rta\infty} \frk c_r e^{\xi h_r(0)} = \infty$. 
Hence a.s.\ $\limsup_{k\rta\infty} D_h(B_{2^k}(0) , B_{2^{k+1}}(0)) =\infty$. 
Since $D_h$ is a length metric, for any $r\geq 2^{k+1}$ and any compact set $K\subset B_{2^k}(0)$, we have $D_h(K , \bdy B_r(0)) \geq D_h(B_{2^k}(0) , B_{2^{k+1}}(0))$. 
We thus obtain the first assertion of the lemma. 
The first assertion (applied with $K$ equal to a single point, say) implies that any $D_h$-bounded subset of $\BB C$ must be contained in a Euclidean-bounded subset of $\BB C$, which must be compact since $D_h$ induces the Euclidean topology on $\BB C$. 
\end{proof}

\subsection{Moment bound for diameters}
\label{sec-diam-moment}

In this section we will prove the following more quantitative version of the moment bound from Theorem~\ref{thm-moment}, which is required to be uniform across scales.

\begin{prop} \label{prop-diam-moment} 
Let $U\subset \BB C$ be open and let $K\subset U$ be a compact connected set with more than one point. 
For each $p\in (-\infty, 4d_\gamma/\gamma^2 )$, there exists $C_p > 0$ which depends on $U$ and $K$ but not on $\BB r$ such that for each $\BB r > 0$, 
\eqb \label{eqn-diam-moment}
\BB E\left[ \left( \frk c_{\BB r}^{-1} e^{-\xi h_{\BB r}(0)}  \sup_{z,w\in \BB r K} D_h(z,w ; \BB r U ) \right)^p \right]  \leq C_p  .
\eqe
\end{prop}
  
We will deduce Proposition~\ref{prop-diam-moment} from the following variant, which allows us to bound internal $D_h$-distances all the way up to the boundary of a square. Recall that $\BB S := (0,1)^2$. 

\begin{prop} \label{prop-internal-moment}  
For each $p\in (-\infty, 4d_\gamma/\gamma^2 )$, there is a constant $C_p  > 0$ such that for each $\BB r >0$,
\eqb \label{eqn-internal-moment}
\BB E\left[\left( \frk c_{\BB r}^{-1} e^{-\xi h_{\BB r}(0)} \sup_{z,w\in \BB r \BB S} D_h\left(z,w ; \BB r\BB S\right) \right)^p \right] \leq C_p. 
\eqe
\end{prop}

\begin{proof}[Proof of Proposition~\ref{prop-diam-moment}, assuming Proposition~\ref{prop-internal-moment}]
For $p < 0$, the bound~\eqref{eqn-diam-moment} follows from the lower bound of Proposition~\ref{prop-two-set-dist}. 
Now assume $p\in (0,4d_\gamma/\gamma^2)$. 
We can cover $K$ by finitely many Euclidean squares $S_1,\dots,S_n$ which are contained in $U$, chosen in a manner depending only on $K$ and $U$. 
For $k=1,\dots,n$, let $u_k$ be the bottom left corner of $S_k$ and let $\rho_k$ be its side length. 
Proposition~\ref{prop-internal-moment} together with Axiom~\ref{item-metric-translate} shows that there is a constant $\wt C_p > 0$ depending only on $p$ such that for each $k=1,\dots,n$, 
\eqb \label{eqn-diam-moment0}
\BB E\left[\left( \frk c_{\BB r \rho_k }^{-1} e^{-\xi h_{\BB r \rho_k}(\BB r u_k)} \sup_{z,w\in \BB r S_k} D_h\left(z,w ; \BB r S_k \right) \right)^p \right] \leq \wt C_p. 
\eqe
We apply the Gaussian tail bound to bound each of the Gaussian random variables $h_{\BB r \rho_k}(\BB r u_k) - h_{\BB r }(0)$ (which have constant order variance) and Theorem~\ref{thm-metric-scaling} to compare $ \frk c_{\BB r \rho_k }$ to $\frk c_{\BB r}$ up to a constant-order multiplicative error.
This allows us to deduce~\eqref{eqn-diam-moment} from~\eqref{eqn-diam-moment0}.
\end{proof}

To prove Proposition~\ref{prop-internal-moment}, we first use the upper bound in Proposition~\ref{prop-two-set-dist} and a union bound to build paths between the two shorter sides of each $2^{-n}\BB r \times 2^{-n-1}\BB r$ or $2^{-n-1}\BB r\times 2^{-n}\BB r$ rectangle with corners in $2^{-n-1}\BB r\BB Z^2$ which is contained in $\BB S$. 
We then string together such paths at all scales (in the manner illustrated in Figure~\ref{fig-diam}) to get a bound for the internal $D_h$-diameter of $\BB r\BB S$. 
The following lemma is needed to control the circle average terms which appear when we apply Proposition~\ref{prop-two-set-dist}.

\begin{lem} \label{lem-circle-avg-all}
Fix $R > 0$ and $q >2$. 
For $C>1$ and $\BB r > 0$, it holds with probability $1 - C^{-q-\sqrt{q^2-4} + o_C(1)}$ as $C\rta\infty$, at a rate which is uniform in $\BB r$, that 
\eqb \label{eqn-circle-avg-all}
  \sup\left\{ |h_{2^{-n}\BB r}(w) - h_{\BB r}(0)|  :  w\in B_{R\BB r}(0) \cap \left(2^{-n-1} \BB r \BB Z^2 \right)  \right\} \leq   \log(C 2^{q n})  ,\quad\forall n \in \BB N_0 .
\eqe
\end{lem} 

When we apply Lemma~\ref{lem-circle-avg-all}, we will take $q$ to be a little bit less than $Q = 2/\gamma+\gamma/2$. The fact that $Q + \sqrt{Q^2-4} = 4/\gamma$ is the reason why $\gamma$ (instead of just $\xi$) appears in our moment bounds. 

\begin{proof}[Proof of Lemma~\ref{lem-circle-avg-all}]
To lighten notation, define the event
\eqb
E_{\BB r}^n := \left\{ \sup\left\{ |h_{2^{-n}\BB r}(w) - h_{\BB r}(0)|  :  w\in B_{R\BB r}(0) \cap \left(2^{-n-1} \BB r \BB Z^2 \right)  \right\} \leq   \log(C 2^{q n}) \right\} .
\eqe
We want a lower bound for the probability that $E_{\BB r}^n$ occurs for every $n\in\BB N_0$ simultaneously. 

Fix $\zeta >0$ (which we will eventually send to $0$) and a partition $\zeta = \alpha_0 < \cdots < \alpha_N = 1/\zeta$ of $[\zeta ,1/\zeta]$ with $\max_{k=1,\dots,N} (\alpha_k - \alpha_{k-1}) \leq\zeta$. 
We will separately bound the probability of $E_{\BB r}^n$ for $2^n \in [C^{\alpha_{k-1} } , C^{\alpha_k}]$ for $k=1,\dots,N$, for $2^n \geq C^{1/\zeta}$, and for $2^n \leq C^\zeta$. 
 
By Lemma~\ref{lem-circle-avg-tail} applied with $\ep = 2^{-n}$, $\nu = 0$, and $ q + 1/\alpha_k $ in place of $q$, we find that for each $k=1,\dots,N$ and each $n\in\BB N_0$ with $2^n \in [C^{\alpha_{k-1} } , C^{\alpha_k}]$,
\allb
 \BB P\left[ (E_{\BB r}^n)^c \right]  
 &\leq \BB P\left[ \sup\left\{ |h_{2^{-n}\BB r}(w) - h_{\BB r}(0)|  :  w\in B_{R\BB r}(0) \cap \left(2^{-n-1} \BB r \BB Z^2 \right)  \right\}   >  \left(q + \frac{1}{\alpha_k}\right) \log( 2^n) \right] \notag \\
& \leq  2^{-n\left( \tfrac{(q+1/\alpha_k)^2}{2} -2    \right)} 
 \leq C^{- \alpha_{k-1} \left(  \tfrac{(q + 1/\alpha_k)^2}{2}  - 2 \right) }
 \leq C^{2\alpha_k -  \tfrac{(q \alpha_k + 1)^2}{2\alpha_k }    + o_\zeta(1)  }  
\alle 
with the rate of the $o_\zeta(1)$ depending only on $q$. Note that in the last inequality, we have done some trivial algebraic manipulations then used that $\alpha_k - \alpha_{k-1} \leq \zeta$ (which is what produces the $o_\zeta(1)$). 
By a union bound over logarithmically many (in $C$) values of $n\in\BB N_0$ with $2^n \in [C^{\alpha_{k-1} } , C^{\alpha_k}]$, we get
\allb \label{eqn-circle-avg-max-med}
 \BB P\left[ E_{\BB r}^n  ,\: \text{$\forall n\in\BB N_0$ with $C^{\alpha_{k-1}} \leq 2^n \leq C^{\alpha_k}$} \right] 
\geq 1 - C^{2\alpha_k - \tfrac{(q \alpha_k + 1)^2}{2\alpha_k }   + o_\zeta(1)  +o_C(1)  }  .
\alle
 
For $n\in\BB N_0$ with $2^n \geq C^{1/\zeta}$, Lemma~\ref{lem-circle-avg-tail} applied with $\ep = 2^{-n}$, $\nu = 0$, and $q+\zeta$ in place of $q$ gives 
\eqbn
\BB P\left[ (E_{\BB r}^n)^c \right] \leq 2^{-n\left( (q+ \zeta)^2/2 -2    \right)} .
\eqen
Summing this estimate over all such $n$ shows that
\eqb \label{eqn-circle-avg-max-large}
\BB P\left[ E_{\BB r}^n ,\: \text{$\forall n\in\BB N$ with $2^n \geq C^{1/\zeta}$} \right] \geq 1 -  C^{- \frac{(q+\zeta)^2 - 4}{2\zeta}   + o_C(1)}  . 
\eqe

Finally, if $n\in\BB N_0$ and $2^n\leq C^\zeta$, then the Gaussian tail bound and a union bound, applied as in the proof of Lemma~\ref{lem-circle-avg-tail}, shows that $\BB P[(E_{\BB r}^n)^c ]\leq C^{2\zeta - (q\zeta + 1)^2/(2\zeta) + o_C(1)}$ (in fact, if $2^n$ is of constant order, this probability will decay superpolynomially in $C$ due to the Gaussian tail bound).
By a union bound over a logarithmic number (in $C$) of such values of $n$ we get  
\eqb \label{eqn-circle-avg-max-small}
\BB P\left[  E_{\BB r}^n ,\: \text{$\forall n\in\BB N$ with $2^n \leq C^\zeta$} \right] 
\geq 1 - C^{2\zeta - \tfrac{(q\zeta + 1)^2}{2\zeta} + o_C(1)} .
\eqe

The quantity $2\alpha - (q\alpha +1)^2/(2\alpha) $ is maximized over all $\alpha >0$ when $\alpha =  (q^2-4)^{-1/2}$, in which case it equals $ -( q + \sqrt{q^2-4} )$. 
Consequently, by combining the estimates~\eqref{eqn-circle-avg-max-med}, \eqref{eqn-circle-avg-max-large}, and~\eqref{eqn-circle-avg-max-small}, we get that if $\zeta$ is chosen sufficiently small relative to $q$, then
\eqb
\BB P\left[ E_{\BB r}^n ,\: \forall n\in\BB N_0 \right] \geq  1 - C^{- q - \sqrt{q^2-4}  + o_\zeta(1) + o_C(1)} .
\eqe
Sending $\zeta \rta 0$ now concludes the proof.
\end{proof}

\begin{figure}[ht!]
\begin{center}
\includegraphics[scale=1]{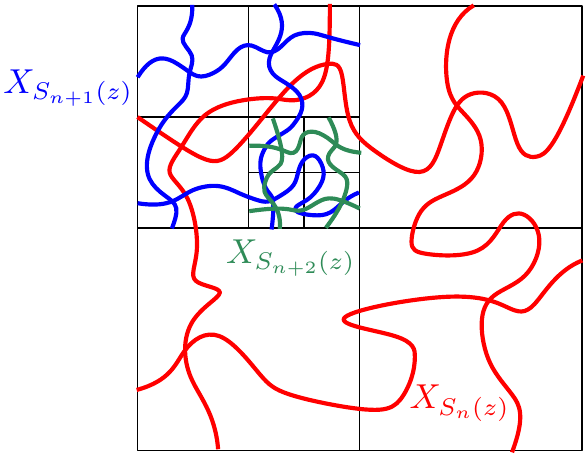} 
\caption{\label{fig-diam} Three of the sets $X_{S_n(z)}$ for dyadic squares containing $z$ used in the proof of Proposition~\ref{prop-internal-moment}. As $n\rta \infty$, the $D_h$-diameter of $S_n(z)$ shrinks to zero (by the continuity of $(z,w) \mapsto D_h(z,w)$), so the distance from $z$ to $X_{S_N(z)}$ is bounded above by the sum over all $n\geq N$ of the $D_h$-lengths of the four paths which comprise $X_{S_n(z)}$. 
}
\end{center}
\end{figure}

\begin{proof}[Proof of Proposition~\ref{prop-internal-moment}]
For $p < 0$, the bound~\eqref{eqn-internal-moment} follows from the lower bound of Proposition~\ref{prop-two-set-dist}. 
We will bound the positive moments up to order $4d_\gamma/\gamma^2$.  
\medskip

\noindent\textit{Step 1: constructing short paths across rectangles.}
Fix $q\in (2,Q)$ which we will eventually send to $Q$.
By Lemma~\ref{lem-circle-avg-all} it holds with probability $1 -  C^{-q-\sqrt{q^2-4} + o_C(1)} $ that 
\eqb \label{eqn-use-circle-avg-all}
 \sup\left\{ |h_{2^{-n}\BB r}(w) - h_{\BB r}(0)|  :  w\in \BB r\BB S \cap \left(2^{-n-1} \BB r \BB Z^2 \right)  \right\} \leq   \log(C 2^{q n})  ,\quad\forall n \in \BB N_0 .
\eqe  

Now fix $\zeta\in (0,Q-q)$, which we will eventually send to zero.
For $n\in\BB N_0$, let $\mcl R_{\BB r}^n$ be the set of open $2^{-n} \BB r \times 2^{-n-1}\BB r$ or $2^{-n-1}\BB r \times 2^{-n }\BB r$ rectangles $R\subset\BB r\BB S$ with corners in $2^{-n-1}\BB r\BB Z^2$. 
For $R\in\mcl R_{\BB r}^n$ let $w_R$ be the bottom-left corner of $R$. 

Let 
\eqb
N_C := \lfloor \log_2 C^\zeta \rfloor .
\eqe 
By the upper bound of Proposition~\ref{prop-two-set-dist} (applied with $2^{-n}\BB r$ in place of $\BB r$ and with $A = 2^{\zeta\xi n}$), Axiom~\ref{item-metric-translate}, and a union bound over all $R\in \mcl R_{\BB r}^n$ and all $n\geq  N_C$, we get that except on an event of probability decaying faster than any negative power of $C$ (the rate of decay depends on $\zeta$), the following is true. 
For each $n\geq N_C$ and each $R\in\mcl R_{\BB r}^n$, the distance between the two shorter sides of $  R$ w.r.t.\ the internal metric $D_h(\cdot,\cdot ; R)$ is at most $2^{\zeta \xi n} \frk c_{2^{-n} \BB r} e^{\xi h_{2^{-n}\BB r}(w_R)}$. 

Combining this with~\eqref{eqn-use-circle-avg-tail} shows that with probability $1 -  C^{-q-\sqrt{q^2-4} + o_C(1)} $, it holds for each $n\geq N_C$ and each $R\in\mcl R_{\BB r}^n$ that there is a path $P_R$ in $R$ between the two shorter sides of $R$ with $D_h$-length at most $ C^\xi 2^{(q+\zeta) \xi n} \frk c_{2^{-n}\BB r} e^{\xi h_{\BB r}(0)} $. 
By applying Theorem~\ref{thm-metric-scaling} to bound $\frk c_{2^{-n}\BB r}$, we get that in fact 
\eqb
\left(\text{$D_h$-length of $P_R$}\right)\leq C^\xi 2^{-(Q-q- \zeta ) \xi n + o_n(n)}  \frk c_{ \BB r} e^{\xi h_{\BB r}(0)} .
\eqe
Henceforth assume that such paths $P_R$ exist. We will establish an upper bound for the $D_h$-diameter of $\BB r\BB S$. 
\medskip

\noindent\textit{Step 2: stringing together paths in rectangles.}
For each square $S\subset \BB r \BB S$ with side length $2^{-n} \BB r $ and corners in $2^{-n}\BB r\BB S$, there are exactly four rectangles in $\mcl R_{\BB r}^n$ which are contained in $S$. 
If $n \geq N_C$, let $X_S$ be the $\#$-shaped region which is the union of the paths $P_R$ for these four rectangles, as illustrated in Figure~\ref{fig-diam}. 
If $S'$ is one of the four dyadic children of $S$, then $X_S\cap X_{S'}\not=\emptyset$.
Since the four paths which comprise $X_S$ have $D_h$-length at most $  C^\xi 2^{-(Q-q- \zeta ) \xi n + o_n(n)} e^{\xi h_{\BB r}(0)} \frk c_{ \BB r} e^{\xi h_{\BB r}(0)}$, this means that each point of $X_S$ can be joined to $X_{S'}$ by a path in $S$ of $D_h$-length at most $  C^\xi 2^{-(Q-q- \zeta ) \xi n + o_n(n)} \frk c_{ \BB r} e^{\xi h_{\BB r}(0)}  $. 

Since the metric $D_h$ is a continuous function on $\BB C\times\BB C$, if $z\in \BB r \BB S$ and we let $S_n(z)$ for $n\in\BB N_0$ be the square of side length $2^{-n} \BB r$ with corners in $2^{-n}\BB r\BB Z^2$ which contains $z$, so that $S_0(z) = \BB S$, then the $D_h$-diameter of $S_n(z)$ tends to zero as $n\rta\infty$. 
Consequently, 
\eqbn
\sup_{w \in S_{N_C}(z)} D_h\left( z , w  ; \BB r \BB S \right) 
\leq C^\xi \frk c_{ \BB r} e^{\xi h_{\BB r}(0)}    \sum_{n=N_C}^\infty 2^{-(Q-q - \zeta) \xi n + o_n(n)}
\leq O_C(C^\xi) \frk c_{ \BB r} e^{\xi h_{\BB r}(0)}  .
\eqen
Since this holds for every $z\in \BB r \BB S$, we get that with probability at least $1 -  C^{-q-\sqrt{q^2-4} o_C(1) } $, for each $n\geq N_C$, each $2^{-n}\BB r\times 2^{-n}\BB r$ square $S\subset \BB r\BB S$ with corners in $2^{-n}\BB r \BB Z^2$ has $D_h(\cdot,\cdot; \BB r \BB S)$-diameter at most $O_C(C^\xi)\frk c_{ \BB r} e^{\xi h_{\BB r}(0)} $. 
\medskip

\noindent\textit{Step 3: conclusion.}
Since $2^{ N_C} \leq C^\zeta$, we can use the triangle inequality to get that if the event at the end of the preceding step occurs, then the $D_h(\cdot,\cdot;\BB r \BB S)$-diameter of $\BB r \BB S$ is at most $O_C(C^{\xi + \zeta}) \frk c_{ \BB r} e^{\xi h_{\BB r}(0)} $. 
Setting $\wt C := C^{\xi + \zeta}$, then sending $\zeta \rta 0$, shows that
\eqbn
\BB P\left[ \frk c_{ \BB r}^{-1} e^{-\xi h_{\BB r}(0)}  \sup_{z,w\in\BB r \BB S} D_h(z,w; \BB r \BB S) > \wt C \right] \leq \wt C^{-\xi^{-1} (  q + \sqrt{q^2-4}  ) + o_{\wt C}(1)} .
\eqen
By sending $q\rta Q$ and noting that $Q + \sqrt{Q^2-4} =4/\gamma$, we get
\eqbn
\BB P\left[ \frk c_{ \BB r}^{-1} e^{-\xi h_{\BB r}(0)}  \sup_{z,w\in\BB r \BB S} D_h(z,w; \BB r \BB S) > \wt C \right] \leq \wt C^{-\frac{4}{\gamma\xi} + o_{\wt C}(1)}  = \wt C^{-\frac{4d_\gamma}{\gamma^2} + o_{\wt C}(1)} .
\eqen
For $p \in (0, 4d_\gamma/\gamma^2)$, we can multiply this last estimate by $\wt C^{p-1}$ and integrate to get the desired $p$th moment bound~\eqref{eqn-internal-moment}.  
\end{proof}

\subsection{Pointwise distance bounds}
\label{sec-ptwise-dist}

In this subsection we will prove the following more quantitative versions of Theorems~\ref{thm-pt-to-circle-moment} and~\ref{thm-pt-to-pt-moment}, which are required to be uniform across scales. Recall that $h$ is a whole-plane GFF normalized so that $h_1(0) =0$. 

\begin{prop}[Distance from a point to a circle] \label{prop-pt-to-circle-moment}
Let $\alpha \in \BB R$ and let $h^\alpha := h - \alpha\log|\cdot|$. 
If $\alpha \in (-\infty ,Q)$, then for each $p\in (-\infty , \frac{2d_\gamma}{\gamma} (Q-\alpha))$, there exists $C_p > 0$ such that for each $\BB r > 0$,
\eqb\label{eqn-pt-to-circle-moment} 
\BB E\left[ \left( \frk c_{\BB r}^{-1} \BB r^{ \alpha\xi} e^{- \xi h_{\BB r}(0)}  D_{h^\alpha}\left(  0 , \bdy B_{\BB r}(0)  \right)  \right)^p \right] \leq C_p .
\eqe     
If $\alpha >Q$, then a.s.\ $D_{h^\alpha}(0,z) = \infty$ for every $z\in\BB C\setminus \{0\}$. 
\end{prop}
 
\begin{prop}[Distance between two points] \label{prop-pt-to-pt-moment}
Let $\alpha , \beta \in \BB R$, let $z,w\in \BB C$ be distinct, and let $h^{\alpha,\beta} := h - \alpha \log|\cdot - z| - \beta \log|\cdot - w|$. 
Set $\BB r := |z-w|/2$. 
If $\alpha,\beta \in (-\infty,Q)$, then for each $p\in \left(-\infty ,  \frac{2d_\gamma}{\gamma} (Q- \max\{\alpha,\beta\} ) \right)$, there exists $C_p > 0$ such that for each choice of $z,w$ as above,
\eqb\label{eqn-pt-to-pt-moment}
\BB E\left[ \left( \frk c_{\BB r}^{-1} \BB r^{ \alpha\xi} e^{- \xi h_{\BB r}(z)}  D_{h^\alpha}\left(  z,w ; B_{8\BB r}(z)  \right)  \right)^p \right] \leq C_p .
\eqe   
If either $\alpha > Q$ or $\beta > Q$, then a.s.\ $D_{h^{\alpha,\beta}}(z,w) = \infty$.
\end{prop}

Propositions~\ref{prop-pt-to-circle-moment} and~\ref{prop-pt-to-pt-moment} are immediate consequences of the following sharper distance estimates and a calculation for the standard linear Brownian motion $t\mapsto h_{\BB r e^{-t}}(0) - h_{\BB r}(0)$.

\begin{prop}  \label{prop-dist-to-int} 
Assume that we are in the setting of Proposition~\ref{prop-pt-to-circle-moment}. 
If $\alpha \in (-\infty ,Q)$, then there is a deterministic function $\psi : [0,\infty) \rta [0,\infty)$ which is bounded in every neighborhood of 0 and satisfies $\lim_{t\rta\infty} \psi(t)/t = 0$, depending only on $\alpha$ and the choice of metric $D$,\footnote{At this point we do not know that the weak LQG metric $D : h\mapsto D_h$ is unique (it will be proven that this metric is unique up to a deterministic multiplicative constant in~\cite{gm-uniqueness}). When we say that something is allowed to depend on the choice of $D$, we mean that it is allowed to depend on which particular weak LQG metric we are looking at.}  such that the following is true. 
For each $\BB r > 0$,  
it holds with superpolynomially high probability as $C \rta \infty$, at a rate which is uniform in the choice of $\BB r$, that
\eqb \label{eqn-dist-to-int}
C^{-1} \frac{\frk c_{\BB r} }{ \BB r^{ \alpha\xi} } \int_0^\infty e^{\xi h_{\BB r e^{-t}}(0)  - \xi (Q-\alpha) t  -  \psi(t)  } \, dt 
\leq D_{h^\alpha}\left(  0 , \bdy B_{\BB r}(0)  \right) 
\leq C  \frac{\frk c_{\BB r} }{ \BB r^{ \alpha\xi} } \int_0^\infty e^{\xi h_{\BB r e^{-t}}(0)  - \xi (Q-\alpha) t +  \psi(t)  } \, dt  
\eqe  
and the $D_{h^\alpha}$-distance around the annulus $B_{\BB r}(0)\setminus B_{\BB r/e}(0)$ (Definition~\ref{def-around-annulus}) is at most the right side of~\eqref{eqn-dist-to-int}.
If $\alpha >Q$, then a.s.\ $D_{h^\alpha}(0,z) = \infty$ for every $z\in\BB C\setminus \{0\}$.  
\end{prop}

\begin{prop}  \label{prop-dist-to-int-pt}
Assume that we are in the setting of Proposition~\ref{prop-pt-to-pt-moment}. 
If $\alpha,\beta \in (-\infty,Q)$, then there is a deterministic function $\psi : [0,\infty) \rta [0,\infty)$ which is bounded in every neighborhood of 0 and satisfies $\lim_{t\rta\infty} \psi(t)/t = 0$, depending only on $\alpha$ and the choice of metric $D$, such that the following is true. With superpolynomially high probability as $C\rta\infty$, at a rate which is uniform in the choice of $z$ and $w$,  
\eqb \label{eqn-dist-to-int-pt-lower}
  D_{h^{\alpha,\beta}}\left( z , w   \right)  \geq C^{-1}  \frac{\frk c_{\BB r} }{ \BB r^{ \alpha\xi} }   \int_0^\infty \left( e^{\xi h_{\BB r e^{-t}}(z)  - \xi (Q-\alpha) t  - \psi(t)}   + e^{\xi h_{\BB r e^{-t}}(w)  - \xi (Q-\beta) t  - \psi(t)}  \right) \,dt    
\eqe
and
\eqb \label{eqn-dist-to-int-pt-upper}
  D_{h^{\alpha,\beta}}\left( z , w  ; B_{8\BB r }(z)   \right)  \leq C \frac{\frk c_{\BB r} }{ \BB r^{ \alpha\xi} }   \int_0^\infty \left( e^{\xi h_{\BB r e^{-t}}(z)  - \xi (Q-\alpha) t   + \psi(t)  }  + e^{\xi h_{\BB r e^{-t}}(w)  - \xi (Q-\beta) t  + \psi(t) }  \right) \,dt      .
\eqe
If either $\alpha > Q$ or $\beta > Q$, then a.s.\ $D_{h^{\alpha,\beta}}(z,w) = \infty$.
\end{prop}
 
\begin{remark} \label{remark-sublinear-error}
It will be shown in~\cite{gm-uniqueness} that every weak LQG metric is a strong LQG metric, so in particular it satisfies Axiom~\ref{item-metric-coord} with $\frk c_r = r^{\xi Q}$.
Once this is established, our proof shows that Propositions~\ref{prop-dist-to-int} and~\ref{prop-dist-to-int-pt} hold with $\psi(t) = 0$. 
\end{remark}

\begin{proof}[Proof of Proposition~\ref{prop-pt-to-circle-moment}, assuming Proposition~\ref{prop-dist-to-int}]
For $t\geq 0$, let $B_t := h_{\BB r e^{-t}}(0) - h_{\BB r}(0)$. 
Then $B$ is a standard linear Brownian motion~\cite[Section 3.1]{shef-kpz}. By Proposition~\ref{prop-dist-to-int}, for each $\zeta \in (0,1)$, it holds with superpolynomially high probability as $C\rta\infty$, uniformly over the choice of $\BB r$, that
\eqb \label{eqn-use-dist-to-int}
C^{-\zeta} \int_0^\infty e^{\xi B_t - (Q-\alpha) \xi t  - \zeta t } \,dt \leq 
\frk c_{\BB r}^{-1} \BB r^{ \alpha\xi} e^{- \xi h_{\BB r}(0)}  D_{h^\alpha}\left(  0 , \bdy B_{\BB r}(0)  \right)
\leq C^\zeta \int_0^\infty e^{\xi B_t - (Q-\alpha) \xi t + \zeta t } \,dt .
\eqe

To prove the proposition, we will use an exact formula for the laws of the integrals appearing in~\eqref{eqn-use-dist-to-int}. 
To write down such a formula, let $\wt B_s := \xi B_{s/\xi^2}$. Then $\wt B$ is a standard linear Brownian motion and $B_t = \xi^{-1} \wt B_{\xi^2 t}$.
Making the change of variables $t = s/\xi^2$ gives
\eqb \label{eqn-bm-coord-change}
\int_0^\infty e^{\xi B_t - (Q-\alpha) \xi t + \zeta t} \,dt = \frac{1}{\xi^2} \int_0^\infty e^{  \wt B_s - (Q-\alpha) s/\xi + \zeta s / \xi^2 } \,ds .
\eqe
It is shown in~\cite{dufresne-perpetuity} (see also~\cite[Example 3.3]{urbanik-functionals} with $c = (Q-\alpha)/\xi - \zeta/\xi^2$) that  
\eqb \label{eqn-bm-density}
\BB P\left[ \int_0^\infty e^{  \wt B_s - (Q-\alpha) s/\xi + \zeta s / \xi^2} \,ds \in \, dx \right] 
= b x^{-2(Q-\alpha)/\xi + 2\zeta /\xi^2  - 1} e^{-2/x} ,\quad\forall x \geq 0  ,
\eqe
where $b$ is a normalizing constant depending only on $Q,\alpha,\xi$. 
Combining the upper bound in~\eqref{eqn-use-dist-to-int} with~\eqref{eqn-bm-coord-change} and the upper tail asymptotics of the density \eqref{eqn-bm-density}, then sending $\zeta \rta 0$, shows that
\eqb\label{eqn-p2b}
\BB P\left[ \frk c_{\BB r}^{-1} \BB r^{ \alpha\xi} e^{- \xi h_{\BB r}(0)}  D_{h^\alpha}\left(  0 , \bdy B_{\BB r}(0)  \right) > C \right] 
\leq C^{-2(Q-\alpha)/\xi - o_C(1)} ,
\eqe
uniformly in $\BB r$. 
Recall that $\xi=\gamma/d_\gamma$. Multiplying both sides of~\eqref{eqn-p2b} by $p C^{p-1}$ and integrating gives the desired bound for positive moments from~\eqref{eqn-pt-to-circle-moment}. We similarly obtain the desired bound for negative moments using the lower bound in~\eqref{eqn-use-dist-to-int} and the exponential lower tail of the density~\eqref{eqn-bm-density}. 
\end{proof}

\begin{proof}[Proof of Proposition~\ref{prop-pt-to-pt-moment}, assuming Proposition~\ref{prop-dist-to-int-pt}]
The bound for positive moments in~\eqref{eqn-pt-to-pt-moment} is obtained in essentially the same way as the analogous bound in Proposition~\ref{prop-pt-to-circle-moment}. 
We apply the upper bound in Proposition~\ref{prop-dist-to-int-pt} and use the exact formula~\eqref{eqn-bm-density} to bound the integral of each of the two summands appearing on the right side of~\eqref{eqn-dist-to-int-pt-upper}, then multiply the resulting tail estimate by $p C^{p-1}$ and integrate. We use that $h_{\BB r}(z) - h_{\BB r}(w)$ is Gaussian with constant-order variance to get an estimate which depends only on $h_{\BB r}(z)$, not $h_{\BB r}(w)$. 
The bound for negative moments in~\eqref{eqn-pt-to-pt-moment} can similarly be extracted from the lower bound in Proposition~\ref{prop-dist-to-int-pt}, or can be deduced from Proposition~\ref{prop-pt-to-circle-moment} and the fact that a path from $z$ to $w$ must cross $\bdy B_{\BB r}(z)$. 
\end{proof}

It remains only to prove Propositions~\ref{prop-dist-to-int} and~\ref{prop-dist-to-int-pt}. 
We will prove Proposition~\ref{prop-dist-to-int} by applying Proposition~\ref{prop-two-set-dist} to bound the distances across and around concentric annuli surrounding 0 with dyadic radii, then summing over all of these annuli (see Figure~\ref{fig-dist-to-int} for an illustration). We will then deduce Proposition~\ref{prop-dist-to-int-pt} from Proposition~\ref{prop-dist-to-int} by considering two overlapping Euclidean disks centered at $z$ and $w$, respectively. For this purpose the statement concerning the $D_h$-distance around $B_{\BB r}(0) \setminus B_{\BB r/e}(0)$ is essential to link up paths in these two disks.

\begin{figure}[ht!]
\begin{center}
\includegraphics[scale=1]{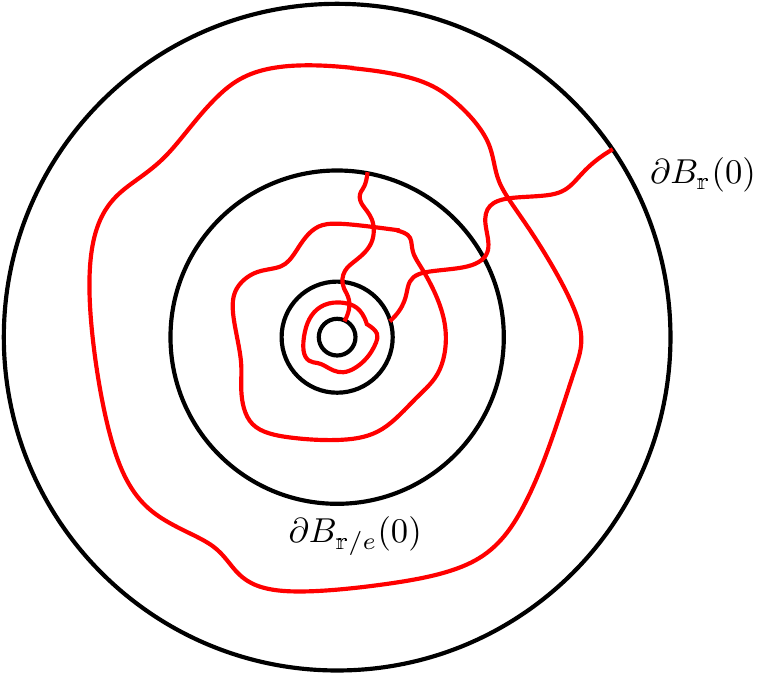} 
\caption{\label{fig-dist-to-int} To prove Proposition~\ref{prop-dist-to-int}, we use Proposition~\ref{prop-two-set-dist} to show that with high probability, the following bounds hold simultaneously for each $k\in\BB N_0$: a lower bound for the $D_h$-distance across the annulus $B_{\BB r e^{-k}}(0)\setminus B_{\BB r e^{-k-1}}(0)$; an upper bound for the $D_h$-distance around this annulus; and a lower bound for the $D_h$-distance across the larger annulus $B_{\BB r e^{-k}}(0)\setminus B_{\BB r e^{-k-2}}(0)$. 
Summing the lower bounds for the distances across these annuli leads to the lower bound in~\eqref{eqn-dist-to-int}. 
The paths involved in the upper bounds are shown in red in the figure. Concatenating all of these paths gives a path from 0 to $\bdy B_{\BB r}(0)$, which leads to the upper bound in~\eqref{eqn-dist-to-int}. 
}
\end{center}
\end{figure}

\begin{proof}[Proof of Proposition~\ref{prop-dist-to-int}] 
See Figure~\ref{fig-two-set-dist} for an illustration. The proof is divided into four steps.
\begin{enumerate}
\item We apply Proposition~\ref{prop-two-set-dist} in the annuli $\BB A_{\BB r e^{-k-1} , \BB r e^{-k}}$ for $k\in\BB N_0$ to prove upper and lower bounds for $D_h(0,\bdy B_{\BB r}(0))$ in terms of sums over such annuli.
\item Using Brownian motion estimates, we convert from sums over annuli to integrals of quantities of the form $ e^{\xi h_{\BB r e^{-t}}(z)  - \xi (Q-\alpha) t   + o_t(t) }$.
\item We show that the contribution of the small error terms in our estimates coming from sums/integrals at superpolynomially small scales is negligible.
\item We put the above pieces together to conclude the proof. 
\end{enumerate}
\medskip

\noindent\textit{Step 1: applying Proposition~\ref{prop-two-set-dist} at exponential scales.} 
We will apply Proposition~\ref{prop-two-set-dist} and take a union bound over exponential scales.
In this step we allow any value of $\alpha\in\BB R$.  

Fix a small parameter $\zeta \in (0,1)$, which we will eventually send to zero. 
By Proposition~\ref{prop-two-set-dist} and Axiom~\ref{item-metric-f} (to deal with the addition of $-\alpha\log|\cdot|$) and a union bound over all $k\in [0,C^{1/\zeta}]_{\BB Z}$, we find that with superpolynomially high probability as $C\rta\infty$, the following is true for each $k \in [0 ,  C^{1/\zeta}]_{\BB Z}$. 
\begin{enumerate}
\item The $D_{h^\alpha}$-distance from $\bdy B_{\BB r e^{-k-1}}(0)$ to $\bdy B_{ \BB r e^{-k} }(0)$ is at least $C^{-1} \frk c_{ \BB r e^{-k} } \BB r^{-\xi\alpha} \exp\left( \xi h_{\BB r e^{-k} }(0)  + \xi \alpha k   \right)  $. \label{item-annulus-perc-lower-k}
\item There is a path from $\bdy B_{ \BB r e^{-k-2} }(0)$ to $\bdy B_{ \BB r e^{-k} }(0)$ which has $D_{h^\alpha}$-length at most\\ 
$C \frk c_{ \BB r e^{-k} } \BB r^{-\xi\alpha} \exp\left(\xi h_{\BB r e^{-k} }(0)  + \xi \alpha k \right)  $. Moreover, there is also a path in $B_{ \BB r e^{-k} }(0) \setminus \ol{B_{ \BB r e^{-k-1} }(0) }$ which disconnects $\bdy  B_{ \BB r e^{-k-1} }(0)$ from $\bdy B_{ \BB r e^{-k} }(0)$ and which has $D_{h^\alpha}$-length at most\\
 $C \frk c_{ \BB r e^{-k} } \BB r^{-\xi\alpha} \exp\left(\xi h_{\BB r e^{-k} }(0)  + \xi \alpha k \right)  $. \label{item-annulus-perc-upper-k}
\end{enumerate}
To deal with the scales for which $k \geq C^{ 1/\zeta}$, we apply Proposition~\ref{prop-two-set-dist} with $ k^\zeta$ in place of $C$ and take a union bound over all such values of $k$ to find that superpolynomially high probability as $C\rta\infty$,  the above two conditions hold for each $k \in [0 , C^{1/\zeta}]_{\BB Z}$, and furthermore the following condition holds for each integer $k\geq C^{ 1/\zeta}$. 
\begin{enumerate}[2$'$.]
\setcounter{enumi}{1}
\item There is a path from $\bdy B_{\BB r e^{-k-2}}(0)$ to $\bdy B_{\BB r e^{-k} }(0)$ which has $D_{h^\alpha}$-length at most\\
 $k^\zeta\frk c_{\BB r e^{-k}} \BB r^{-\xi\alpha} \exp\left( \xi h_{ \BB r e^{-k} }(0)  + \xi \alpha k \right)$. 
Moreover, there is also a path in $B_{\BB r e^{-k}}(0) \setminus \ol{B_{\BB r e^{-k-1}}(0) }$ which disconnects $\bdy  B_{\BB r e^{-k-1}}(0)$ from $\bdy B_{\BB r e^{-k}}(0)$ and which has $D_{h^\alpha}$-length at most\\
 $k^\zeta\frk c_{\BB r e^{-k}} \BB r^{-\xi\alpha} \exp\left( \xi h_{ \BB r e^{-k} }(0)  + \xi \alpha k \right)$.  \label{item-annulus-perc-small}
\end{enumerate}

\newcommand{\refsmall}{\hyperref[item-annulus-perc-small]{$2'$}}
 
Henceforth assume that conditions~\ref{item-annulus-perc-lower-k} and~\ref{item-annulus-perc-upper-k} hold for each $k \in [0,C^{1/\zeta}]_{\BB Z}$ and condition~\refsmall\, holds for each integer $k\geq C^{1/\zeta}$, which happens with superpolynomially high probability as $C\rta\infty$.

Any path from $0$ to $\bdy B_{\BB r}(0)$ must cross each of the annuli $B_{ \BB r e^{-k }}(0) \setminus B_{ \BB r e^{-k -1}}(0)$ for $k  \in [0 , C^{1/\zeta} ]_{\BB Z}$.  
Furthermore, the union of $\{0\}$ and the paths from conditions~\ref{item-annulus-perc-upper-k} and~\refsmall\, for all $k \in \BB N_0$ contains a path from $0$ to $\bdy B_{\BB r}(0)$. 
By Theorem~\ref{thm-metric-scaling}, there is a deterministic function $\phi : [0,\infty) \rta [0,\infty)$ with $\phi(k) = o_k(k)$,  depending only on the choice of metric $D$, such that
\eqb \label{eqn-scaling-const-error}
e^{-\xi Q k - \phi(k)} \frk c_{\BB r} \leq \frk c_{\BB r e^{-k}} \leq e^{-\xi Q k + \phi(k)} \frk c_{\BB r} ,\quad\forall \BB r  >0 .
\eqe  
Summing the bounds from conditions~\ref{item-annulus-perc-lower-k} and~\ref{item-annulus-perc-upper-k} over all $k  \in [0, C^{1/\zeta}]_{\BB Z}$ and the bounds from condition~\refsmall\, over all integers $k\geq C^{1/\zeta}$ and plugging in~\eqref{eqn-scaling-const-error} shows that with superpolynomially high probability as $C\rta\infty$,
\allb \label{eqn-dist-to-sum-annulus}
&C^{-1} \frac{\frk c_{\BB r} }{ \BB r^{ \alpha\xi} } \sum_{k=0}^{\lfloor C^{1/\zeta}\rfloor} e^{\xi h_{\BB r e^{-k}}(0)  - \xi (Q-\alpha) k  - \phi(k) }   \leq   
D_{h^\alpha}\left(   0 , \bdy B_{\BB r}(0) \right) \notag\\ 
&\qquad \leq C  \frac{\frk c_{\BB r} }{ \BB r^{ \alpha\xi} }  \sum_{k=0}^{\lfloor C^{1/\zeta}\rfloor} e^{\xi h_{ \BB r e^{-k}}(0)  - \xi (Q-\alpha) k + \phi(k)  } 
+ \frac{\frk c_{\BB r} }{ \BB r^{ \alpha\xi} } \sum_{k=\lfloor C^{1/\zeta}\rfloor  +1}^\infty k^\zeta e^{\xi h_{ \BB r e^{-k}}(0)  - \xi (Q-\alpha) k + \phi(k)  } .
\alle
Furthermore, by condition~\ref{item-annulus-perc-upper-k}  for $k = 0$ the $D_{h^\alpha}$-distance around $B_{\BB r}(0) \setminus B_{\BB r/e}(0)$ is at most the right side of~\eqref{eqn-dist-to-sum-annulus}. 
\medskip

\noindent\textit{Step 2: from summation to integration.}
We now want to convert from sums to integrals in~\eqref{eqn-dist-to-sum-annulus}. 
Since $t\mapsto h_{\BB r e^{-t} }(0) - h_{\BB r}(0)$ is a standard linear Brownian motion~\cite[Section 3.1]{shef-kpz}, the Gaussian tail bound and the union bound show that with superpolynomially high probability as $C\rta\infty$, 
\eqb  \label{eqn-sum-int-bm}
\sup_{t\in [k ,k + 1]} |h_{\BB r e^{-t}}(0) - h_{\BB r e^{-k}}(0) | \leq \frac{1}{  \xi} \log C ,\quad \forall k \in \left[ 0 , C^{1/\zeta} \right]_{\BB Z} .
\eqe
Let $\psi(t) := \phi(\lfloor t \rfloor)$, where $\phi$ is as in~\eqref{eqn-scaling-const-error}. Then $\psi(t) = o_t(t)$ and if~\eqref{eqn-sum-int-bm} holds, then for each $k\in [0, C^{1/\zeta}]_{\BB Z}$, 
\allb \label{eqn-dist-to-int-pre-sum}
e^{\xi h_{ \BB r e^{-k}}(0)  - \xi (Q-\alpha) k  -\phi(k) } 
&\geq C^{-1} \int_k^{k+1} e^{\xi h_{\BB r e^{-t}}(0)  - \xi (Q-\alpha) t -\psi(t) } \, dt 
\quad \text{and} \notag\\
e^{\xi h_{ \BB r e^{-k}}(0)  - \xi (Q-\alpha) k  +\phi(k) } 
&\leq C \int_k^{k+1} e^{\xi h_{\BB r e^{-t}}(0)  - \xi (Q-\alpha) t  + \psi(t) } \, dt .
\alle 
By summing~\eqref{eqn-dist-to-int-pre-sum} over all $k\in [0, C^{1/\zeta} ]_{\BB Z}$, we obtain 
\allb \label{eqn-dist-to-int-bm}
 \sum_{k=0}^{\lfloor C^{1/\zeta}\rfloor } e^{\xi h_{ \BB r e^{-k}}(0)  - \xi (Q-\alpha) k  - \phi(k) } 
 &\geq C^{-1} \int_0^{\lfloor C^{1/\zeta}\rfloor  + 1 } e^{\xi h_{\BB r e^{-t}}(0)  - \xi (Q-\alpha) t  - \psi(t) } \, dt     \quad \text{and} \quad \notag\\
 \sum_{k=0}^{\lfloor C^{1/\zeta}\rfloor } e^{\xi h_{ \BB r e^{-k}}(0)  - \xi (Q-\alpha) k  + \phi(k) }  
&\leq C \int_0^{\lfloor C^{1/\zeta}\rfloor  +1 } e^{\xi h_{ \BB re^{-t} }(0)  - \xi (Q-\alpha) t  +  \psi(t) } \, dt  .
\alle 
\medskip

\noindent\textit{Step 3: bounding the sum of the small scales.}
To deduce our desired bounds from~\eqref{eqn-dist-to-sum-annulus} and~\eqref{eqn-dist-to-int-bm}, we now need an upper bound for $\int_{\lfloor C^{1/\zeta} \rfloor}^\infty e^{\xi h_{\BB r e^{-t} }(0)  - \xi (Q-\alpha) t + \psi(t) } \, dt $ and an upper bound for the second sum on the right side of~\eqref{eqn-dist-to-sum-annulus}. 
This is the only step where we need to assume that $\alpha  <Q$. 

Since $t\mapsto h_{ \BB r e^{-t}}(0) - h_{\BB r}(0)$ is a standard linear Brownian motion and for $q\in (0,1]$, $x\mapsto x^q$ is concave, hence subadditive, if $q \in (0,1]$ is chosen small enough that $\xi q (Q-\alpha) - \xi^2 q^2/2  > 0$, then
\alb
\BB E\left[ \left( \int_{\lfloor C^{1/\zeta}\rfloor}^\infty e^{\xi h_{ \BB r e^{-t}}(0)  - \xi (Q-\alpha) t + \psi(t)  } \,dt  \right)^q \right]
&\preceq e^{q h_{\BB r}(0)} \int_{\lfloor C^{1/\zeta}\rfloor}^\infty \exp\left(  - \left(  \xi q  (Q-\alpha) -  \frac{\xi^2 q^2}{2}  \right) t  + o_t(t) \right) \,dt  \\
&\preceq  e^{q h_{\BB r}(0)} \exp\left( - \frac12 \left(  \xi q (Q-\alpha) -  \frac{\xi^2 q^2 }{2}  \right) C^{1/\zeta} \right) ,
\ale 
where here the $o_t(t)$ and the implicit constants in $\preceq$ do not depend on $C$ or $\BB r$.
Therefore, the Chebyshev inequality shows that 
\eqb
\BB P\left[\int_{\lfloor C^{1/\zeta}\rfloor}^\infty e^{\xi h_{ \BB r e^{-t}}(0)  - \xi (Q-\alpha) t  + \psi(t)  } \,dt  > e^{ \xi h_{\BB r}(0) - C^{1/(2\zeta)}} \right]
\eqe
decays faster than any negative power of $C$. On the other hand, it is easily seen from the Gaussian tail bound that 
\eqb
\BB P\left[   \int_0^{\lfloor C^{1/\zeta}\rfloor} e^{\xi h_{\BB r e^{-t}}(0)  - \xi (Q-\alpha) t  + \psi(t)  } \, dt   < e^{\xi h_{\BB r}(0) -C^{1/(2\zeta)}}  \right]
\eqe
decays faster than any negative power of $C$. 
Hence with superpolynomially high probability as $C\rta\infty$,
\eqb \label{eqn-dist-to-int-lower-error}
 \int_0^\infty e^{\xi h_{\BB r e^{-t}}(0)  - \xi (Q-\alpha) t   + \psi(t)  } \, dt  \leq 2 \int_0^{\lfloor C^{1/\zeta}\rfloor} e^{\xi h_{\BB r e^{-t}}(0)  - \xi (Q-\alpha) t  +  \psi(t)  } \, dt .
\eqe
Similarly, we get that with superpolynomially high probability as $C\rta\infty$,  
\eqb \label{eqn-dist-to-int-upper-error}
\sum_{k= \lfloor C^{1/\zeta} \rfloor +1}^\infty  k^\zeta e^{\xi h_{\BB r e^{-k}}(0)  - \xi (Q-\alpha) k  - \phi(k)   } \leq    \int_0^\infty e^{\xi h_{ \BB r e^{-t}}(0)  - \xi (Q-\alpha) t  -  \psi(t)  } \, dt .
\eqe
\medskip

\noindent\textit{Step 4: conclusion.}
By applying \eqref{eqn-dist-to-int-bm}, \eqref{eqn-dist-to-int-lower-error}, and~\eqref{eqn-dist-to-int-upper-error} to bound the left and right sides of~\eqref{eqn-dist-to-sum-annulus}, we get that if $\alpha < Q$, then with superpolynomially high probability, the bounds~\eqref{eqn-dist-to-int} as well as the bound stated just below~\eqref{eqn-dist-to-int} (here we use the sentence just below~\eqref{eqn-dist-to-sum-annulus}) all hold with $2C^2$, say, in place of $C$. Since we are claiming that these bounds hold with superpolynomially high probability as $C\rta\infty$, this is sufficient. 

Finally, we consider the case when $\alpha  > Q$. 
Since $h_{\BB r e^{-t}}(0) - h_{\BB r}(0)$ evolves as a standard linear Brownian motion, for each $\beta \in (0, \alpha - Q)$ it is a.s.\ the case that the summand $e^{\xi h_{\BB r e^{-k}}(0) - \xi(Q-\alpha) k - \phi(k)}$ in the lower bound in~\eqref{eqn-dist-to-sum-annulus} is bounded below by $e^{\beta k}$ for large enough $k$. (How large is random). 
Since~\eqref{eqn-dist-to-sum-annulus} holds with superpolynomially high probability as $C\rta\infty$, the Borel-Cantelli lemma combined with the preceding sentence shows that a.s.\ for large enough (random) $C > 1$, we have $D_{h^\alpha}\left(   0 , \bdy B_{\BB r}(0) \right) \geq C^{-1} e^{\beta \lfloor C^{1/\zeta} \rfloor}$, which tends to $\infty$ as $C\rta\infty$. 
This shows that a.s.\ $D_{h^\alpha}(0,\bdy B_{\BB r}(0)) = \infty$. 
Since this holds a.s.\ for each rational $\BB r > 0$, it follows that a.s.\ $D_{h^\alpha}(0,z) = \infty$ for every $z \in \BB C\setminus\{0\}$. 
\end{proof}

\begin{proof}[Proof of Proposition~\ref{prop-dist-to-int-pt}]
We first observe that by Axiom~\ref{item-metric-translate}, Proposition~\ref{prop-dist-to-int} still holds with 0 replaced by any $z\in\BB C$, with the rate of convergence as $C\rta\infty$ uniform in $z$ and $\BB r$. 
 Applying the lower bound of Proposition~\ref{prop-dist-to-int} with each of $z$ and $w$ in place of $0$ immediately gives~\eqref{eqn-dist-to-int-pt-lower} since any path from $z$ to $w$ must contain disjoint sub-paths from $z$ to $\bdy B_{\BB r/2}(z)$ and from $w$ to $B_{\BB r/2}(w)$. 
Moreover, by comparing the local behavior of $D_{h^{\alpha,\beta}}$ near $z$ and near $w$ to $D_{h^\alpha}$ and $D_{h^\beta}$, respectively, we get that  
a.s.\ $D_{h^{\alpha,\beta}}(z,w) = \infty$ if either $\alpha >Q$ or $\beta >Q$.  

It remains to prove~\eqref{eqn-dist-to-int-pt-upper}. Assume $\alpha < Q$. We first apply Proposition~\ref{prop-dist-to-int} with $8\BB r$ in place of $\BB r$ to find that with superpolynomially high probability as $C\rta\infty$, there is a path $P_{z,1}$ from $z$ to $\bdy B_{8\BB r}(z)$ and a path $P_{z,2}$ in $B_{\BB r}(z)\setminus B_{8\BB r/e}(z)$ which disconnects $\bdy B_{8\BB r / e}(z)$ from $\bdy B_{8\BB r}(z)$ which each have $D_h$-length at most 
\eqbn
\int_{-\log 8}^\infty   e^{\xi h_{\BB r e^{-t}}(z)  - \xi (Q-\alpha) t  + \psi(t) }   \,dt ;
\eqen
and the same is true with $w$ in place of $z$. Since $w \in B_{8\BB r/e}(z)$, the union of the paths $P_{z,1} , P_{z,2},$ and $ P_{w,1}$ contains a path from $z$ to $w$ in $B_{8\BB r}(z)$. This gives~\eqref{eqn-dist-to-int-pt-upper} but with $-\log 8$ instead of $0$ in the lower bound of integration for the integral on the right. 

To get the estimate with the desired lower bound of integration, we use that $t\mapsto h_{\BB r e^{-t}}(z) - h_{\BB r}(z)$ is a standard two-sided linear Brownian motion. In particular, two applications of the Gaussian tail bound show that with superpolynomially high probability as $C\rta\infty$, 
\eqbn 
\sup_{t \in [-\log 8 , 0]}  h_{\BB r e^{-t}}(z)   
\leq   \inf_{t \in [0 , \log 2]}   h_{\BB r e^{-t}}(z)    +   \log C .
\eqen
Therefore, with superpolynomially high probability as $C\rta\infty$, 
\alb
\int_{-\log 8}^\infty   e^{\xi h_{\BB r e^{-t}}(z)  - \xi (Q-\alpha) t  + \psi(t) }   \,dt 
&\leq \int_0^\infty   e^{\xi h_{\BB r e^{-t}}(z)  - \xi (Q-\alpha) t  +  \psi(t) } \,dt \notag \\ 
&\qquad + C^\xi \int_0^{\log 2}   e^{\xi h_{\BB r e^{-t}}(z)  - \xi (Q-\alpha) t  + \psi(t) }  \,dt .
\ale
Combining this with the analogous estimate with $w$ in place of $z$ and the aforementioned analog of~\eqref{eqn-dist-to-int-pt-upper} with $-\log 8$ instead of 0 in the lower bound of integration gives~\eqref{eqn-dist-to-int-pt-upper}.
\end{proof}

Although it is not needed for the proofs of Propositions~\ref{prop-dist-to-int} and~\ref{prop-dist-to-int-pt}, we record the following generalization of Proposition~\ref{prop-diam-moment} which tells us in particular that $D_{h^\alpha}$ induces the Euclidean topology on $\BB C$ when $Q>2$ and $\alpha <Q$ (which is a stronger statement than just that $D_{h^\alpha}(0,z) < \infty$ for every $z\in\BB C$).

\begin{prop} \label{prop-diam-moment-alpha}
Let $h$, $\alpha$, $h^\alpha$, and $D_{h^\alpha}$ be as in Proposition~\ref{prop-dist-to-int}. 
If $Q = 2/\gamma + \gamma/2 > 2$ and $\alpha \in (-\infty, Q)$, then for each $ -\infty < p < \min\{ \frac{4 d_\gamma}{\gamma^2} , \frac{2 d_\gamma }{\gamma}(Q-\alpha)  \}$, there exists $C_{\alpha,p} >  0$ such that for each $\BB r > 0$, 
\eqb \label{eqn-diam-moment-alpha}
\BB E\left[\left( e^{-\xi h_{\BB r}(0)}  \frk c_{\BB r}^{-1} \BB r^{ \alpha\xi} \sup_{z,w\in B_{\BB r}(0) } D_{h^\alpha}(z,w) \right)^p \right] \leq C_{\alpha,p} .
\eqe
In particular, a.s.\ $D_{h^\alpha}$ induces the Euclidean topology on $\BB C$. 
\end{prop}

We note that the range of moments $ -\infty < p < \min\{ \frac{4 d_\gamma}{\gamma^2} , \frac{2 d_\gamma }{\gamma}(Q-\alpha)  \}$ for the $D_{h^\alpha}$-diameter of $\BB D$ appearing in Proposition~\ref{prop-diam-moment-alpha} is the same as the range of moments for the $\mu_{h^\alpha}$-mass of $\BB D$, but scaled by $d_\gamma$; see, e.g.,~\cite[Lemma A.3]{ghs-dist-exponent}. This is natural from the perspective that $d_\gamma$ is the scaling exponent relating $\gamma$-LQG distances and areas. 

\begin{proof}[Proof of Proposition~\ref{prop-diam-moment-alpha}]
On $B_{\BB r}(0) \setminus B_{\BB r/2}(0)$,  we have that $-\alpha\log|\cdot|$ is bounded above and below by $-\alpha\log\BB r$ times constants depending only on $\alpha$. Therefore, the existence of negative moments is immediate from Axiom~\ref{item-metric-f} and Proposition~\ref{prop-diam-moment} applied with $U = \BB D\setminus B_{1/2}(0)$. 

To get the desired positive moments, for $k\in \BB N_0$ let $A_k $ be the annulus $B_{\BB r e^{-k }}(0) \setminus B_{\BB r e^{-k-1}}(0)$.
The random variable $h_{\BB r e^{-k}}(0) - h_{\BB r}(0)$ is Gaussian with variance $k$, so for $p > 0$, 
\eqb \label{eqn-exp-circle-moment}
\BB E\left[ e^{ p \xi  \left( h_{\BB r e^{-k}} (0) - h_{\BB r}(0) \right) }  \right] = e^{p^2\xi^2 k/2 } ,\quad\forall p > 0 .
\eqe 
By Proposition~\ref{prop-diam-moment} (applied with $K = A_0$, $U = \BB C$, and $\BB r e^{-k}$ in place of $\BB r$),  
\eqb \label{eqn-scaled-diam-moment}
\BB E\left[ \left(  \frk c_{\BB r e^{-k}}^{-1} e^{-\xi h_{\BB r e^{-k}}(0)} e^{-\alpha \xi k} \BB r^{\alpha k} \sup_{z,w\in A_k}   D_{h^\alpha}(z,w)       \right)^p \right] \preceq 1 ,\quad \forall p   < \frac{4d_\gamma}{\gamma^2} .
\eqe
By~\eqref{eqn-exp-circle-moment} and~\eqref{eqn-scaled-diam-moment} and since $(h - h_{\BB r e^{-k}}(0))|_{A_k}$ is independent from $h_{\BB r e^{-k}}(0) - h_{\BB r}(0)$, we find that for $p\in (0,4d_\gamma/\gamma^2)$,
\allb \label{eqn-annulus-moment}
&\BB E\left[ \left( e^{-\xi h_{\BB r}(0)}  \frk c_{\BB r}^{-1} \BB r^{ \alpha\xi} \sup_{z,w\in A_k } D_{h^\alpha}(z,w) \right)^p \right] \notag\\ 
&\qquad= \left(\frac{\frk c_{\BB r e^{-k}}}{\frk c_{\BB r}} \right)^p e^{ p \alpha \xi k} 
\BB E\left[ e^{ p \xi  \left( h_{\BB r e^{-k}} (0) - h_{\BB r}(0) \right) }  \right] 
\BB E\left[ \left(  \frk c_{\BB r e^{-k}}^{-1} e^{-\xi h_{\BB r e^{-k}}(0)} e^{-\alpha \xi k} \BB r^{\alpha k} \sup_{z,w\in A_k}   D_{h^\alpha}(z,w)       \right)^p \right]  \notag\\ 
&\qquad\leq \exp\left(    - \left( \xi p (Q-\alpha) - \frac{p^2 \xi^2}{2}   \right) k  + o_k(k)  \right)  ,
\alle  
at a rate depending only on $\alpha, p$. Note that in the last line we used Theorem~\ref{thm-metric-scaling} to bound $\frk c_{\BB r e^{-k}} / \frk c_{\BB r}$. 

The quantity inside the exponential on the right side of~\eqref{eqn-annulus-moment} is negative provided $p < \min\{ \frac{4 d_\gamma}{ \gamma^2} , \frac{2 d_\gamma}{\gamma}(Q-\alpha)  \}$ (recall that $\xi = \gamma/d_\gamma$). 
For $0 < p  < \min\{1 ,\frac{2 d_\gamma }{\gamma}(Q-\alpha)\} $, the function $x\mapsto x^p$ is concave, hence subadditive, so summing~\eqref{eqn-annulus-moment} over all $k \in \BB N_0$ gives
\allb \label{eqn-diam-alpha-sum}
\BB E\left[ \left( e^{-\xi h_{\BB r}(0)}  \frk c_{\BB r}^{-1} \BB r^{ \alpha\xi} \sup_{z,w\in B_{\BB r}(0) } D_{h^\alpha}(z,w) \right)^p \right] 
&\leq \sum_{k=0}^\infty\BB E\left[ \left( e^{-\xi h_{\BB r}(0)}  \frk c_{\BB r}^{-1} \BB r^{ \alpha\xi} \sup_{z,w\in A_k } D_{h^\alpha}(z,w) \right)^p \right] \notag \\
&\preceq \sum_{k= 0}^\infty \exp\left(    - \left( \xi p (Q-\alpha) - \frac{p^2 \xi^2}{2}   \right) k  + o_k(k)  \right)  \notag \\
&\preceq 1   .
\alle 
This gives~\eqref{eqn-diam-moment-alpha} in the case when  $0 < p  < \min\{1 ,\frac{2 d_\gamma}{\gamma}(Q-\alpha)\} $.  
In the case when $1 \leq p <  \min\{ \frac{4 d_\gamma}{ \gamma^2} , \frac{2 d_\gamma}{\gamma}(Q-\alpha)  \}$,~\eqref{eqn-diam-moment-alpha} follows from a similar calculation with the triangle inequality for the $L^p$ norm used in place of sub-additivity. 

Finally, we know that the restriction of $D_{h^\alpha}$ to $\BB C\setminus \{0\}$ induces the Euclidean topology (see the discussion just above Theorem~\ref{thm-pt-to-circle-moment}), so to check that that $D_{h^\alpha}$ induces the Euclidean topology, we need to show that a.s.\ $\sup_{z,w\in B_{e^{-k}}(0)} D_{h^\alpha}(z,w) \rta 0$ as $k \rta\infty$.
This follows from the bound~\eqref{eqn-diam-alpha-sum} applied with $\BB r =1$ and the Borel-Cantelli lemma. 
\end{proof}

\subsection{H\"older continuity}
\label{sec-holder}

We will prove the following more quantitative version of Theorem~\ref{thm-optimal-holder} which is required to be uniform across scales.

\begin{prop} \label{prop-holder-uniform}
Fix a compact set $K\subset\BB C$ and exponents $\chi \in (0,\xi(Q-2))$ and $\chi' > \xi(Q+2)$. For each $\BB r > 0$, it holds with polynomially high probability as $\ep\rta 0$, at a rate which is uniform in $\BB r$, that
\eqb \label{eqn-holder-uniform}
\left|\frac{u-v}{\BB r} \right|^{\chi'} \leq  \frk c_{\BB r}^{-1} e^{-\xi h_{\BB r}(0)}  D_h\left( u,v  \right) \leq \left|\frac{u-v}{\BB r} \right|^\chi , 
 \quad\forall u,v\in \BB r K \: \text{with} \: |u-v| \leq \ep \BB r .
\eqe
\end{prop}

We will actually prove a slightly stronger version of the upper bound for $D_h$ in Proposition~\ref{prop-holder-uniform}, which bounds internal distances relative to a small neighborhood of $u$ instead of just distances along paths in all of $\BB C$; see Lemma~\ref{lem-holder-upper} just below. 
This stronger version is used in~\cite{gm-uniqueness}. 

For the proof of Proposition~\ref{prop-holder-uniform}, we assume that $Q >2$ and we fix a compact set $K\subset\BB C$. 
The basic idea of the proof of the upper bound in~\eqref{eqn-holder-uniform} is to apply Proposition~\ref{prop-diam-moment} to Euclidean balls of radius $\ep$ and take a union bound over many such Euclidean balls which cover $K$.
The basic idea for the proof of the lower bound in~\eqref{eqn-holder-uniform} is to apply the lower bound in Proposition~\ref{prop-two-set-dist} to lower bound the $D_h$-distance across Euclidean annuli of the form $B_{2\ep}(z) \setminus B_\ep(z)$, then take a union bound over many such annuli whose inner balls cover $K$. 
We first prove an upper bound for $D_h$-distances in terms of Euclidean distances. For this purpose we will use the following consequence of Propositions~\ref{prop-diam-moment} and~\ref{prop-internal-moment}.

\begin{lem} \label{lem-ep-diam}   
For each $s  \in (0,\xi Q)$, each $\BB r>0$, and each $z \in \BB r K $,
\eqb \label{eqn-ep-diam}
\BB P\left[ \sup_{u,v \in B_{\ep \BB r }(z)} D_h\left( u,v ; B_{2\ep \BB r}(z) \right) \leq \ep^s \frk c_{\BB r} e^{\xi h_{\BB r}(0)} \right] \geq 1 -    \ep^{\tfrac{(  \xi Q - s)^2}{2\xi^2} + o_\ep(1)}   , \quad \text{as $\ep\rta 0$},
\eqe  
uniformly over the choices of $\BB r$ and $z\in \BB r K$. 
Furthermore, if we let $S^{\ep \BB r}(z)$ be the square of side length $\ep \BB r$ centered at $z$, then for $\BB r > 0$ and $z\in \BB r K$, the $D_h$-internal diameter of $S^{\ep \BB r}(z)$ satisfies
\eqb \label{eqn-ep-diam-square}
\BB P\left[ \sup_{u,v \in S^{\ep \BB r}(z)} D_h\left( u,v ; S^{ \ep \BB r}(z) \right) \leq \ep^s \frk c_{\BB r} e^{\xi h_{\BB r}(0)} \right] \geq 1 -    \ep^{\tfrac{(  \xi Q - s)^2}{2\xi^2} + o_\ep(1)}  , \quad \text{as $\ep\rta 0$},
\eqe  
uniformly over the choices of $\BB r$ and $z\in \BB r K$. 
\end{lem}
\begin{proof}
We know that $h_{2\ep\BB r}(z) - h_{\BB r}(z)$ is centered Gaussian of variance $\log\ep^{-1}   - \log 2$ and is independent from $(h-h_{2\ep\BB r}(z))|_{B_{2\ep \BB r}(z)}$. 
By Axioms~\ref{item-metric-local} and~\ref{item-metric-f}, $h_{2\ep\BB r}(z) - h_{\BB r}(z)$ is also independent from the internal metric
\eqbn
D_{h-h_{2\ep\BB r}(z)}\left( u,v ; B_{2\ep \BB r}(z) \right)
= e^{-\xi h_{2\ep\BB r}(z)} D_h\left( u,v ; B_{2\ep \BB r}(z) \right) .
\eqen
Consequently, we can apply Theorem~\ref{thm-metric-scaling} and Proposition~\ref{prop-diam-moment} (with $\ep\BB r$ in place of $\BB r$) together with the formula $\BB E[e^X] = e^{\op{Var}(X)/2}$ for a Gaussian random variable $X$ to get that for $p\in (0,4/(\gamma\xi))$, 
\allb \label{eqn-ep-diam-moment}
&\BB E\left[ \left( \frk c_{\BB r}^{-1} e^{- \xi h_{\BB r}(0)} \sup_{u,v \in B_{\ep \BB r }(z)} D_h\left( u,v ; B_{2\ep \BB r}(z) \right) \right)^p \right] \notag \\
&\qquad = \left( \frac{\frk c_{\ep \BB r}}{\frk c_{\BB r}} \right)^p \BB E\left[  e^{\xi p (h_{\ep\BB r}(z) - h_{\BB r}(z)} \right]  \BB E\left[ \left( \frk c_{\ep \BB r}^{-1} e^{- \xi h_{\ep \BB r}(z)} \sup_{u,v \in B_{\ep \BB r }(z)} D_h\left( u,v ; B_{2\ep \BB r}(z) \right) \right)^p \right] \notag \\
&\qquad \leq \ep^{\xi Q p - \xi^2 p^2/2 + o_\ep(1)} ,
\alle 
with the $o_\ep(1)$ uniform over all $\BB r >0$ and $z\in\BB C$. 

By~\eqref{eqn-ep-diam-moment} and the Chebyshev inequality,  
\eqb
\BB P\left[ \sup_{u,v\in B_{\ep\BB r}(z) } D_h\left( u , v ; B_{2\ep \BB r}(z) \right)   > \ep^s \frk c_{\BB r} e^{\xi h_{\BB r}(z)} \right] \leq   \ep^{ p \xi Q  - \frac{p^2\xi^2}{2} - p s +o_\ep(1)}  .
\eqe 
The exponent on the right side is maximized for $p = (\xi Q - s)/\xi^2$, which is always at most $4/(\xi\gamma)$ for $s  > 0$ (since $\gamma < 2$) and is positive provided $s < \xi Q$. 
Making this choice of $p$ gives~\eqref{eqn-ep-diam} but with $h_{\BB r}(z)$ in place of $h_{\BB r}(0)$. The random variables $h_{\BB r}(z) - h_{\BB r}(0)$ for $z\in\BB r K $ are Gaussian with variance bounded above by a constant depending only on $K$. Consequently, we can apply the Gaussian tail bound to get~\eqref{eqn-ep-diam} in general. 

The bound~\eqref{eqn-ep-diam-square} is proven similarly but with Proposition~\ref{prop-internal-moment} used in place of Proposition~\ref{prop-diam-moment}. 
\end{proof}

We can now prove a slightly sharper version of the upper bound of Proposition~\ref{prop-holder-uniform}. 

\begin{lem} \label{lem-holder-upper}
For each $\chi \in (0,\xi(Q-2))$ and each $\BB r > 0$, it holds with polynomially high probability as $\ep\rta 0$, at a rate which is uniform in $\BB r$, that
\eqb \label{eqn-holder-upper}
 \frk c_{\BB r}^{-1} e^{-\xi h_{\BB r}(0)}  D_h\left( u,v ; B_{ 2|u-v| }(u) \right) \leq \left|\frac{u-v}{\BB r} \right|^\chi , 
 \quad\forall u,v\in \BB r K \: \text{with} \: |u-v| \leq \ep \BB r .
\eqe
Furthermore, it also holds with polynomially high probability as $\ep\rta 0$, at a rate which is uniform in $\BB r$, that for each $k\in\BB N_0$ and each $2^{-k}\ep \BB r \times 2^{-k}\ep \BB r$ square $S$ with corners in $2^{-k}\ep \BB r \BB Z^2$ which intersects $\BB r K$, we have
\eqb \label{eqn-holder-upper-square}
 \frk c_{\BB r}^{-1} e^{-\xi h_{\BB r}(0)} \sup_{u,v \in S} D_h\left( u,v ; S \right) \leq (2^{-k} \ep)^\chi .
\eqe
\end{lem}
\begin{proof}
The bound~\eqref{eqn-holder-upper} follows from~\eqref{eqn-ep-diam}, applied with $s = \chi$ and with $2^{-k} \ep$ for $k\in\BB N_0$ in place of $\ep$, together with a union bound over all $z\in B_{\ep\BB r}(K)\cap (2^{-k-2}\ep\BB r \BB Z^2) $ and then over all $k\in\BB N_0$. 
The bound~\eqref{eqn-holder-upper-square} similarly follows from~\eqref{eqn-ep-diam-square}.  
\end{proof}

To prove the H\"older continuity of the Euclidean metric w.r.t.\ $D_h$, we first need the following estimate which plays a role analogous to Lemma~\ref{lem-ep-diam}. 

\begin{lem} \label{lem-ep-cross}
For each $s  > \xi Q$, each $\BB r>0$, and each $ z\in \BB r K $, 
\eqb \label{eqn-ep-cross}
\BB P\left[  D_h\left( B_{\ep\BB r}(z) , \bdy B_{2\ep\BB r}(z) \right) \geq \ep^s \frk c_{\BB r} e^{\xi h_{\BB r}(0)} \right] \geq 1 -  \ep^{\tfrac{( s - \xi Q)^2}{2\xi^2} + o_\ep(1)}  ,\quad
\text{as $\ep\rta 0$},
\eqe   
uniformly over the choices of $\BB r$ and $z\in \BB r K$. 
\end{lem}
\begin{proof}
The proof is similar to that of Lemma~\ref{lem-ep-diam} but we use Proposition~\ref{prop-two-set-dist} instead of Proposition~\ref{prop-diam-moment}.
Proposition~\ref{prop-two-set-dist} implies that $\frk c_{\ep\BB r}^{-1} e^{-\xi h_{\ep\BB r}(z)} D_h\left( B_{\ep\BB r}(z) , \bdy B_{2\ep\BB r}(z) \right)$ has finite moments of all negative orders which are bounded above uniformly over all $z\in\BB C$ and $\BB r > 0$. 
By the same calculation as in~\eqref{eqn-ep-diam-moment}, for each $p > 0$ we have 
\allb \label{eqn-ep-cross-moment}
 \BB E\left[ \left( \frk c_{\BB r}^{-1} e^{- \xi h_{\BB r}(z)} D_h\left( B_{\ep\BB r}(z) , \bdy B_{2\ep\BB r}(z) \right) \right)^{-p} \right]  
 = \ep^{- \xi Q p - \xi^2 p^2/2 + o_\ep(1)} ,
\alle 
uniformly over all $z\in\BB C$ and $\BB r > 0$. Applying the Chebyshev inequality and setting $p = (s-\xi Q)/\xi^2$ gives~\eqref{eqn-ep-cross} with $h_{\BB r}(z)$ in place of $h_{\BB r}(0)$. For $z\in\BB r K$, we can replace $h_{\BB r}(z)$ with $h_{\BB r}(0)$ via exactly the same argument as in the proof of Lemma~\ref{lem-ep-diam}. 
\end{proof}

\begin{lem} \label{lem-holder-inverse}
For each $\chi'  > \xi(Q+2)$ and each $\BB r >0$, it holds with polynomially high probability as $\ep\rta 0$, at a rate which is uniform in $\BB r$, that
\eqb \label{eqn-holder-inverse}
 \frk c_{\BB r}^{-1} e^{-\xi h_{\BB r}(0)}  D_h\left( u,v   \right) \geq \left|\frac{u-v}{\BB r} \right|^{\chi'} , 
 \quad\forall u,v\in K \: \text{with} \: |u-v| \leq \ep .
\eqe 
\end{lem}
\begin{proof}
This follows from~\eqref{eqn-ep-diam}, applied with $s = \chi'$ and with $2^{-k} \ep$ for $k\in\BB N_0$ in place of $\ep$, together with a union bound over all $z\in B_{\ep\BB r}(K)\cap (2^{-k-2}\ep\BB r \BB Z^2) $ and then over all $k\in\BB N_0$. 
\end{proof}

\begin{proof}[Proof of Proposition~\ref{prop-holder-uniform}]
Combine Lemmas~\ref{lem-holder-upper} and \ref{lem-holder-inverse}.
\end{proof}

To conclude the proof of Theorem~\ref{thm-optimal-holder}, we need to check that the H\"older exponents $\xi(Q-2)$ and $(\xi(Q+2))^{-1}$ are optimal.

\begin{lem} \label{lem-no-holder}
Let $V\subset \BB C$ be an open set. Almost surely, the identity map from $V$, equipped with the Euclidean metric, to $(V,D_h|_V)$ is \emph{not} H\"older continuous with any exponent greater than $ \xi(Q-2)$. Furthermore, the inverse of this map is \emph{not} H\"older continuous with any exponent greater than $\xi^{-1} (Q+2)^{-1}$.  
\end{lem}
\begin{proof}
The idea of the proof is to use Proposition~\ref{prop-dist-to-int} to study $D_h$-distances as we approach an $\alpha$-thick point of $h$ for $\alpha$ close to $2$ or to $-2$. 
To produce such a thick point, we will sample a point from the $\alpha$-LQG measure induced by the zero-boundary part of $h|_V$. 
By Axiom~\ref{item-metric-f}, we can assume without loss of generality that $h$ is normalized so that $h_1(0) = 0$. 
We can also assume without loss of generality that $V$ is bounded with smooth boundary.   
Let $h^V$ be the zero-boundary part of $h|_V$, so that $h - h^V$ is harmonic on $V$.

Let $\alpha \in (-2,2)$ which we will eventually send to either $-2$ or $2$, and let $\mu_{h^V}^\alpha$ be the $\alpha$-LQG measure induced by $h^V$. 
Also let $\BB z$ be sampled uniformly from $\mu_h^\alpha$, normalized to be a probability measure.
Let $\wt{\BB P}$ be the law of $(h  ,\BB z)$ weighted by the total mass $\mu_{h^V}^\alpha(V)$, so that under $\wt{\BB P}$, $h $ is sampled from its marginal law weighted by $\mu_{h^{V}}^\alpha(V)$ and conditional on $h$, $\BB z$ is sampled from $\mu_{h^{V}}^\alpha$, normalized to be a probability measure.
By a well-known property of the $\alpha$-LQG measure (see, e.g.,~\cite[Lemma A.10]{wedges}), a sample $(h ,\BB z)$ from the law $\wt{\BB P}$ can be equivalently be produced by first sampling $\wt h $ from the unweighted marginal law of $h$, then independently sampling $\BB z$ uniformly from Lebesgue measure on $\BB S'$ and setting $h  = \wt h  - \alpha \log|\cdot-\BB z| + \frk g_{\BB z}$, where $\frk g_{\BB z} : V \rta\BB R$ is a deterministic continuous function.  
 
By Proposition~\ref{prop-dist-to-int} (applied with the field $\wt h - \alpha \log|\cdot - \BB z|$ in place of $h^\alpha$), the fact that $\frk g_{\BB z}$ is a.s.\ bounded in a neighborhood of $\BB z$ (by continuity), and the Borel-Cantelli lemma, we find that a.s.\ 
\allb \label{eqn-no-holder-int}
D_{h }\left( \BB z , \bdy B_r(\BB z)  \right)  
= r^{  o_r(1)} \frac{c_r}{r^{\alpha\xi}} \int_0^\infty e^{\xi \wt h_{r e^{-t} }(\BB z)  - \xi (Q-\alpha) t o_t(t) } \, dt  ,
\alle 
where here the $o_t(t)$ is deterministic and tends to 0 as $t\rta\infty$ (it comes from the error $\psi(t)$ in Proposition~\ref{prop-dist-to-int}) and the $o_r(1)$ denotes a random variable which tends to 0 a.s.\ as $r\rta 0$.
The description in the preceding paragraph shows that conditional on $\BB z$, the process $t\mapsto \wt h_{r e^{-t}}(\BB z) - \wt h_{r }(\BB z)$ evolves as a standard linear Brownian motion. Consequently, the Gaussian tail bound shows that with probability tending to 1 as $r\rta 0$,   
\eqb \label{eqn-no-holder-int-asymp}
\int_0^\infty e^{\xi \wt h_{r e^{-t}}(\BB z)  - \xi (Q-\alpha) t +  o_t(t) } \, dt  =  r^{o_r(1)}  e^{\xi \wt h_r(\BB z)}  = r^{ o_r(1)} . 
\eqe 
By plugging~\eqref{eqn-no-holder-int-asymp} into~\eqref{eqn-no-holder-int} and using the fact that $\frk c_r = r^{\xi Q  + o_r(1)}$ (Theorem~\ref{thm-metric-scaling}), it therefore follows that with probability tending to 1 as $r\rta 0$, 
\eqbn
D_{h }\left( \BB z , \bdy B_r(\BB z)  \right)  =   r^{   \xi ( Q-\alpha) + o_r(1) } . 
\eqen
Since $\alpha$ can be made arbitrarily close to 2, this shows the desired lack of H\"older continuity for identity map $(V,|\cdot|) \rta (V, D_h)$.   
Since $\alpha$ can be made arbitrarily close to $-2$, we also get the desired lack of H\"older continuity for the inverse map $(V,D_h) \rta (V , |\cdot|)$. 
\end{proof}

\section{Constraints on the behavior of $D_h$-geodesics}
\label{sec-geodesic-bounds}

Let $D$ be a weak $\gamma$-LQG metric. 
By Lemma~\ref{lem-infinite-dist}, for a whole-plane GFF $h$, the metric space $(\BB C , D_h)$ is a boundedly compact length space (i.e., closed bounded subsets are compact) so there is a $D_h$-geodesic --- i.e., a path of minimal $D_h$-length --- between any two points of $\BB C$~\cite[Corollary 2.5.20]{bbi-metric-geometry}. 
In this section we will apply the main results of this paper to prove two estimates which constrain the behavior of $D_h$-geodesics. 
The first of these estimates, Proposition~\ref{prop-line-path}, tells us that paths which stay in a small Euclidean neighborhood of a straight line or an arc of the boundary of a circle have large $D_h$-lengths. In particular, $D_h$-geodesics are unlikely to stay in such a neighborhood.
The second estimate, Proposition~\ref{prop-geo-bdy}, says that a $D_h$-geodesic cannot spend a long time near the boundary of a $D_h$-metric ball.

\subsection{Lower bound for $D_h$-distances in a narrow tube}
\label{sec-line-path}

\begin{prop} \label{prop-line-path}
Let $L\subset\BB C$ be a compact set which is either a line segment, an arc of a circle, or a whole circle and fix $ b > 0$. 
For each $\BB r > 0$ and each $p > 0$, it holds with probability at least $1 -  \ep^{p^2/(2\xi^2) + o_\ep(1)}$ that
\eqb \label{eqn-line-path}
 \inf\left\{ D_h\left(u,v ; B_{\ep\BB r}(\BB r L ) \right) : u,v \in B_{\ep\BB r}(\BB r L ) , |u-v| \geq b\BB r \right\} \geq \ep^{ p + \xi Q - 1-\xi^2/2 } \frk c_{\BB r} e^{\xi h_{\BB r}(0)}  ,
\eqe
where the rate of the $o_\ep(1)$ depends on $L, b,  p$ but not on $\BB r$.
\end{prop} 

By~\cite[Theorem 1.9]{ang-discrete-lfpp}, for each $\gamma \in (0,2)$ we have $\xi Q \leq 1$ and hence $\xi Q - 1 - \xi^2/2 < 0$. Therefore, the power of $\ep$ on the right side of~\eqref{eqn-line-path} is negative for small enough $p$. Hence, Proposition~\ref{prop-line-path} implies that when $\ep$ is small and $u,v\in B_{\ep \BB r}(\BB r L)$ with $|u-v| \geq b\BB r$, it holds with high probability that $D_h\left(u,v; B_{\ep\BB r}(\BB rL)\right)$ is much larger than $D_h(u,v)$. In particular, a $D_h$-geodesic from $u$ to $v$ cannot stay in $B_{\ep\BB r}(L)$.

\begin{proof}[Proof of Proposition~\ref{prop-line-path}]
\noindent\textit{Step 1: bounding distances in terms of circle averages.}
View $ L$ as a path $[0,| L|] \rta \BB C$ parametrized by Euclidean unit speed. 
For $k\in [0,| L|/(6\ep)]_{\BB Z}$, let $z_k^\ep := \BB r L(6 k \ep )$. 
Then the balls $B_{3\ep\BB r}(z_k^\ep)$ are disjoint and the balls $B_{7\ep\BB r}(z_k^\ep)$ cover $B_{\ep\BB r}(\BB r L)$. 
 
Fix $\zeta\in (0,1)$, which we will eventually send to zero. 
By Proposition~\ref{prop-two-set-dist} and a union bound, it holds with superpolynomially high probability as $\ep\rta 0$ that 
\eqb \label{eqn-line-path-annulus}
D_h\left( B_{2\ep \BB r}(z_k^\ep) ,  B_{3\ep \BB r}(z_k^\ep) \right) 
\geq \ep^\zeta \frk c_{\ep \BB r} e^{\xi h_{\ep\BB r}(z_k^\ep)} ,\quad
\forall k \in  [0,| L|/(6\ep)]_{\BB Z} .
\eqe
Henceforth assume that~\eqref{eqn-line-path-annulus} holds. The idea of the proof is that a path in $B_{\ep\BB r}(\BB r L)$ has to cross between the inner and outer boundaries of a large number of the annuli $B_{3\ep\BB r}(z_k^\ep) \setminus B_{2\ep\BB r}(z_k^\ep)$. Thus~\eqref{eqn-line-path-annulus} reduces our problem to proving a lower bound for the sum of the quantities $\ep^\zeta \frk c_{\ep \BB r} e^{\xi h_{\ep\BB r}(z_k^\ep)}$ for these annuli, which in turn can be proven using Theorem~\ref{thm-metric-scaling} and basic estimates for the circle average process.
\medskip 

\noindent\textit{Step 2: lower-bounding lengths of paths in $B_{\ep\BB r}(\BB r L)$ in terms of circle averages.}
There is a constant $c > 0$ depending only on $b$ and $L$ such that for small enough $\ep > 0$ (depending only on $b$ and $L$), the following is true.
If $u,v\in B_{\ep\BB r}(\BB r L)$ satisfy $|u-v| \geq b\BB r$, there are integers $0 \leq k_1' < k_2' \leq | L|/(6\ep)$ such that $k_2'-k_1' \geq c \ep^{-1}$, $u \in B_{7\ep\BB r}(z_{k_1'}^\ep)$, and $v\in B_{7\ep\BB r}(z_{k_2'}^\ep)$. 
Each path from $u$ to $v$ in $B_{\ep\BB r}(\BB r L)$ must enter $B_{2\ep\BB r}(z_k^\ep)$ for each $k\in [k_1'+2,k_2'-2]_{\BB Z}$, and hence must cross the annulus $\BB A_{2\ep \BB r , 3\ep \BB r}(z_k^\ep)  $ for each such $k$. 
Combining this with~\eqref{eqn-line-path-annulus} shows that  
\eqb \label{eqn-line-path-sum}
D_h\left( u , v ; B_{\ep\BB r}(\BB r L)   \right) \geq   \ep^{ \zeta}   \frk c_{\ep \BB r} \sum_{k=k_1'+2}^{k_2'-2} e^{\xi h_{\ep\BB r}(z_k^\ep)} . 
\eqe
\medskip 

\noindent\textit{Step 3: proof conditional on a circle average estimate.}
We claim that for any fixed $k_1,k_2 \in [0,|L|/(6\ep)]_{\BB Z}$ with $k_2-k_1 \geq (c/2) \ep^{-1}$ and any $p > 0$, 
\eqb \label{eqn-line-path-sum-lower}
\BB P\left[ \sum_{k=k_1}^{k_2} e^{\xi h_{\ep\BB r}(z_k^\ep)} \geq \ep^{p-1-\xi^2/2} e^{\xi h_{\ep\BB r}(0)} \right] \geq 1 - \ep^{\frac{p^2}{2\xi^2} + o_\ep(1)}  
\eqe
where the rate of the $o_\ep(1)$ depends on $L, b,  p$ but not on $\BB r$ or the particular choice of $k_1,k_2$.
We will prove~\eqref{eqn-line-path-sum-lower} just below using standard Gaussian estimates. 

Let us first conclude the proof assuming~\eqref{eqn-line-path-sum-lower}. 
We can find a constant-order number of pairs $k_1,k_2\in [0,|L|/(6\ep)]_{\BB Z}$ with $k_2-k_1 \geq (c/2) \ep^{-1}$ such that for small enough $\ep$ (depending only on $L$ and $b$), each interval $[k_1' + 2 ,k_2' - 2] \subset [0,|L|/(6\ep)]_{\BB Z}$ with $|k_2'-k_1'| \geq c\ep^{-1}$ contains one of the intervals $[k_1,k_2]$. 

By applying~\eqref{eqn-line-path-sum-lower} (with $p-2\zeta$ in place of $p$) to each such pair $k_1,k_2$, then taking a union bound, we get that with probability at least $1 -  \ep^{\frac{(p-2\zeta)^2}{2\xi^2} + o_\ep(1)}   $, the sum on the right side of~\eqref{eqn-line-path-sum} is bounded below by $ \ep^{p-1-\xi^2/2 - 2\zeta} e^{\xi h_{\ep\BB r}(0)} $ simultaneously for every possible choice of $k_1',k_2'$.
By~\eqref{eqn-line-path-sum}, with probability at least $1 -  \ep^{\frac{(p-2\zeta)^2}{2\xi^2} + o_\ep(1)}   $ it holds simultaneously for each $u,v\in B_{\ep\BB r}(\BB r L)$ satisfying $|u-v| \geq b\BB r$ that
\eqb
D_h\left( u , v ; B_{\ep\BB r}(\BB r L)   \right) \geq \ep^{p-1-\xi^2/2 -\zeta} \frk c_{\ep \BB r} e^{\xi h_{\BB r}(0)} \geq \ep^{ p + \xi Q -1-\xi^2/2 -\zeta + o_\ep(1)} \frk c_{  \BB r} e^{\xi h_{\BB r}(0)} 
\eqe
where in the second inequality we use Theorem~\ref{thm-metric-scaling}.
Sending $\zeta\rta 0$ now gives~\eqref{eqn-line-path}.
\medskip 

\noindent\textit{Step 4: proof of the circle average estimate.}
The rest of the proof is devoted to proving the inequality~\eqref{eqn-line-path-sum-lower}. 
To lighten notation, write $X_k := h_{\ep\BB r}(z_k^\ep) - h_{\BB r}(0)$. 
By the calculations in~\cite[Section 3.1]{shef-kpz} (and the scale invariance of the law of $h$, modulo additive constant), the $X_k$'s are jointly centered Gaussian with variances satisfying
\eqb \label{eqn-exp-sum-var}
\op{Var}(X_k) = \log\ep^{-1} + O (1),
\eqe 
where here $O(1)$ denotes a quantity which is bounded above and below by constants depending only on $L,b$ (not on $\ep,\BB r,j,k$).
Since $z_k^\ep = \BB r L(6 k \ep )$ and $L$ is parametrized by Euclidean unit speed, we also have the following covariance formula for $j\not=k$:
\eqb \label{eqn-exp-sum-covar}
\op{Cov}\left( X_j ,X_k \right)  = \log \left( \frac{\BB r}{|z_j^\ep -z_k^\ep|} \right) + O (1)  = \log\left(\frac{1}{\ep |k-j|} \right) + O (1) .
\eqe 

Recall the formula $\BB E[e^X] = e^{\op{Var}(X)/2}$ for a centered Gaussian random variable $X$. 
Applying this to the $X_k$'s and recalling~\eqref{eqn-exp-sum-var} and the fact that $k_2-k_1 \asymp \ep^{-1}$ gives 
\eqb \label{eqn-exp-sum-first-moment}
\BB E\left[ \sum_{k=k_1}^{k_2} e^{\xi X_k} \right] \asymp \ep^{-1 -\xi^2/2} ,
\eqe
with the implicit constant depending only on $L,b$. 
From~\eqref{eqn-exp-sum-var} and~\eqref{eqn-exp-sum-covar} we obtain $\op{Var}(X_j+X_k) =  \log\left( \ep^{-4} |k-j|^{-2} \right) + O(1)$ for $j\not=k$.
Hence
\allb \label{eqn-exp-sum-second-moment}
\BB E\left[ \left( \sum_{k=k_1}^{k_2} e^{\xi X_k} \right)^2 \right]
&= \sum_{k=k_1}^{k_2} \BB E\left[ e^{2\xi X_k} \right] +  2 \sum_{k=k_1}^{k_2} \sum_{j=k+1}^{k_2} \BB E\left[ e^{\xi(X_j + X_k)} \right] \notag\\
&\preceq \ep^{-1-2\xi^2} +  2 \ep^{-2\xi^2} \sum_{k=k_1}^{k_2} \sum_{j=k+1}^{k_2}  |j-k|^{-\xi^2} \notag \\
&\preceq \ep^{-1-2\xi^2}  + \ep^{-2-\xi^2} \preceq \ep^{-2-\xi^2}
\alle
with the implicit constants depending only on $L,b$, where in the last inequality we use that $\xi < 2/d_2 <  1$, so $1+2\xi^2 < 2+\xi^2$. 

By~\eqref{eqn-exp-sum-first-moment}, \eqref{eqn-exp-sum-second-moment}, and the Payley-Zygmund inequality, we find that there is a constant $a = a(L) > 0$ such that
\eqb \label{eqn-use-pz}
\BB P\left[ \sum_{k=k_1}^{k_2} e^{\xi X_k}  \geq a \ep^{-1-\xi^2/2} \right] \geq a .
\eqe
To improve the lower bound for this probability, we will apply the following elementary Gaussian concentration bound (see, e.g.,~\cite[Lemma 2.1]{dzz-heat-kernel}):

\begin{lem} \label{lem-gaussian-concentration}
For any $a >0$, there exists $C = C(a) > 0$ such that the following is true.
Let $\mathbf X = (X_1,\dots,X_n)$ be a centered Gaussian vector taking values in $\BB R^n$ and let $\sigma^2 := \max_{1\leq j \leq n} \op{Var}(X_j)$. 
If $B\subset \BB R^n$ such that $\BB P[X\in B] \geq a$, then for any $\lambda \geq C\sigma$,
\eqb
\BB P\left[ \inf_{\mathbf x \in B} |\mathbf X - \mathbf x|_\infty > \lambda \right] \leq e^{-\tfrac{(\lambda-C\sigma)^2}{2\sigma^2} } ,
\eqe
where $|\cdot|_\infty$ is the $L^\infty$ norm on $\BB R^n$. 
\end{lem} 

We now apply Lemma~\ref{lem-gaussian-concentration} with $a$ as in~\eqref{eqn-use-pz}, with $\sigma^2 = \log\ep^{-1} + O(1)$ (recall~\eqref{eqn-exp-sum-var}), with
\eqb
B = \left\{ (x_{k_1} , \dots , x_{k_2}) \in \BB R^{k_1+k_2+1} : \sum_{k=k_1}^{k_2} e^{\xi x_k} \geq a \ep^{-1-\xi^2/2} \right\} ,
\eqe
and with $\lambda = \frac{p}{\xi} \log\ep^{-1}$. This shows that with probability $1 - \ep^{p^2/(2\xi^2) + o_\ep(1)}$, there exists $(x_{k_1} , \dots , x_{k_2}) \in B$ such that $\max_{k \in [k_1,k_2]_{\BB Z}} |X_k - x_k| \leq \frac{p}{\xi} \log\ep^{-1}$. 
If this is the case, then
\eqb
\sum_{k=k_1}^{k_2} e^{\xi X_k} \geq \ep^p \sum_{k=k_1}^{k_2} e^{\xi x_k} \geq a \ep^{p -1-\xi^2/2  } .
\eqe
Since $X_k = h_{\ep\BB r}(z_k^\ep) - h_{\BB r}(0)$, this implies~\eqref{eqn-line-path-sum-lower}.  
\end{proof}

\subsection{$D_h$-geodesics cannot trace the boundaries of $D_h$-metric balls}
\label{sec-geo-bdy}

For $s >0$ and $z\in\BB C$, we write $\mcl B_s(z;D_h)$ for the $D_h$-metric ball of radius $s$ centered at $z$. 
The following proposition prevents a $D_h$-geodesic from spending a long time near the boundary of a $D_h$-metric ball.

\begin{prop} \label{prop-geo-bdy}
For each $ M  > 0$ and each $\BB r >0$, it holds with superpolynomially high probability as $\ep\rta 0$, at a rate which is uniform in the choice of $\BB r$, that the following is true.
For each $s > 0$ for which $\mcl B_s(0;D_h)\subset B_{\ep^{-M} \BB r}(0)$ and each $D_h$-geodesic $P$ from 0 to a point outside of $\mcl B_s(0;D_h)$,  
\eqb \label{eqn-geo-bdy}
\op{area}\left( B_{\ep \BB r}(P) \cap B_{\ep \BB r}\left(\bdy\mcl B_s(0;D_h) \right) \right) \leq \ep^{2 - 1/ M} \BB r^2,
\eqe
where $\op{area}$ denotes 2-dimensional Lebesgue measure.  
\end{prop}

\begin{figure}[ht!]
\begin{center}
\includegraphics[scale=1]{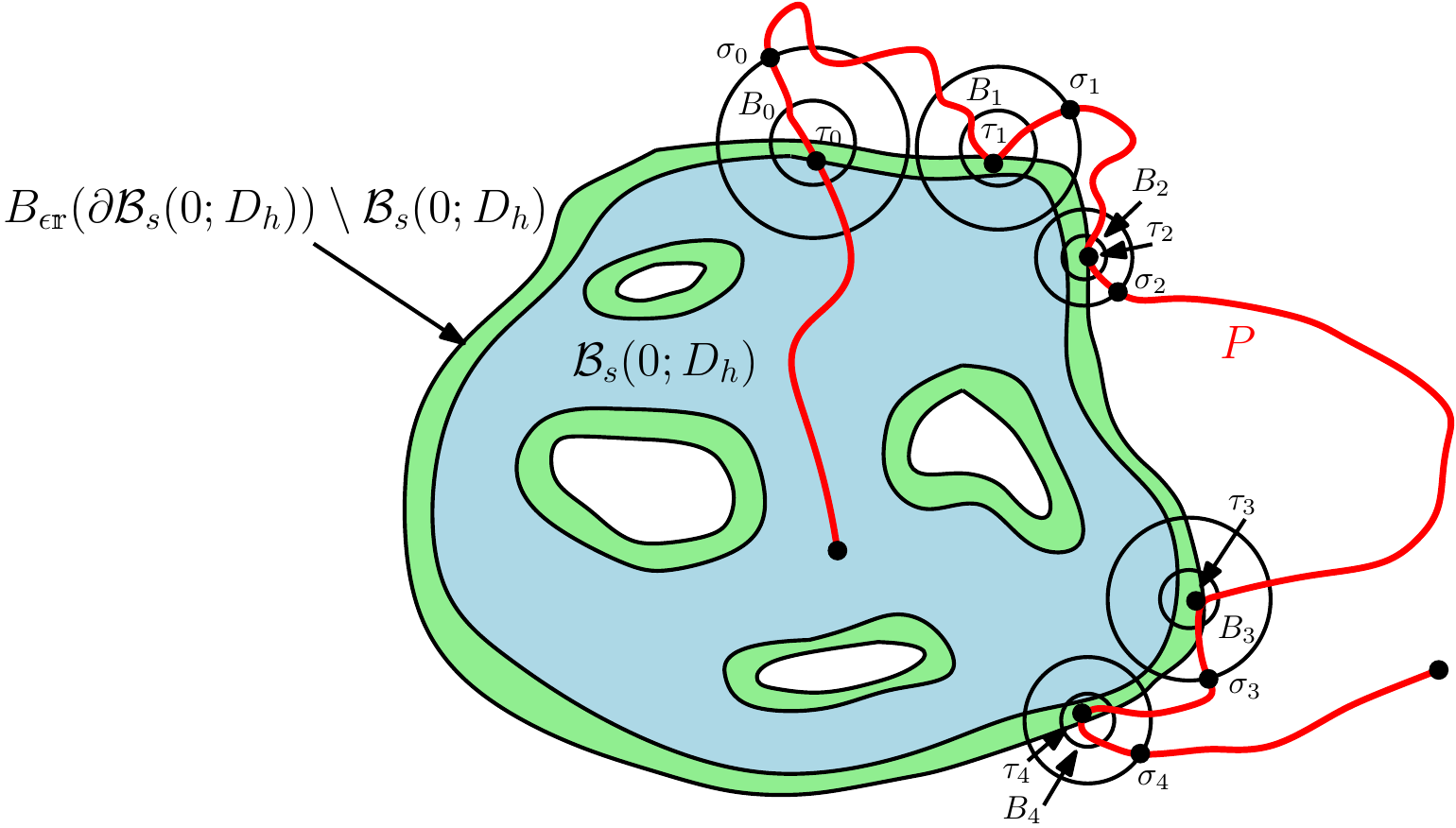} 
\caption{\label{fig-geo-bdy} Illustration of the proof of Proposition~\ref{prop-geo-bdy}. 
By considering successive times at which $P$ enters $B_{\ep \BB r}(\mcl B_s(0;D_h))$, we can find $K\in\BB N$ and a collection of $K$ $C$-good Euclidean balls $B_0,\dots,B_K$ with radii in $[2\ep\BB r , \ep^{1-\zeta}\BB r]$ with the following properties:
(a) each $B_k$ intersects $\bdy \mcl B_s(0;D_h)$; (b) the $D_h$-geodesic $P$ crosses the annuli $(2B_k)\setminus B_k$ for $k\in [0,K-1]_{\BB Z}$ in numerical order; and (c) the balls of radii $4\ep^{1-\zeta}\BB r$ with the same centers as the $B_k$'s cover $P\cap B_{\ep\BB r}(\mcl B_s(0;D_h))$.
This last property implies that $\op{area}\left( B_{\ep \BB r}(P) \cap B_{\ep \BB r}(\bdy\mcl B_s(0;D_h) \right) \leq \op{const} \times \ep^{2-2\zeta  } \BB r^2 K$, so we are left to bound $K$. 
To this end, we show using the definition~\eqref{eqn-C-good-def} of a $C$-good ball and the fact that $P$ is a $D_h$-geodesic that $D_h(\bdy B_k , \bdy (2B_k))$ increases exponentially in $k$. Due to Lemma~\ref{lem-geo-bdy-ratio}, this implies that $K \leq \ep^{-1/(2M)}$. 
}
\end{center}
\end{figure}

For $C>1$, $z\in\BB C$, and $r>0$, we say that the Euclidean ball $B_r(z)$ is \emph{$C$-good} if 
\eqb \label{eqn-C-good-def}
 \sup_{u,v\in \bdy B_r(z)} D_h\left(u,v; \BB A_{r/2,2r}(z) \right) \leq  C D_h\left( \bdy B_r(z) , \bdy B_{2r}(z)\right)  .
\eqe
To prove Proposition~\ref{prop-geo-bdy}, we will consider $C$-good balls which intersect $\bdy \mcl B_s(0;D_h)$ and which are hit by a given $D_h$-geodesic started from 0. See Figure~\ref{fig-geo-bdy} for an illustration and outline of the proof. 

\begin{lem} \label{lem-geo-bdy-annuli}
For each $\zeta \in (0,1)$ and each $M >0$, there exists $C = C(\zeta, M) > 1$ such that for each $\BB r > 0$, it holds with probability at least $1-O_\ep(\ep^M)$, at a rate which is uniform in $\BB r$, that the Euclidean ball $B_{\ep^{-M} \BB r}(0)$ can be covered by $C$-good balls with radii in $ [2\ep \BB r , \ep^{1-\zeta} \BB r]$. 
\end{lem}
\begin{proof}
This is an immediate consequence of Lemma~\ref{lem-good-annulus-all} applied with $\ep^{1-\zeta}$ in place of $\ep$ and any choice of $\nu  \in (0, \frac{1}{1-\zeta}  -1)$. 
\end{proof}


We will also need the following easy consequence of the distance bounds from Section~\ref{sec-a-priori-estimates}. 

\begin{lem} \label{lem-geo-bdy-ratio}
For each $M>0$, there exists $A = A(M) > 0$ such that for each $\BB r > 0$, the following holds with probability $1-O_\ep(\ep^M)$ as $\ep\rta 0$, at a rate which is uniform in $\BB r$. For each $z,w\in B_{\ep^{- M} \BB r}(0)$ with $|z-w| \geq \ep \BB r$, 
\eqb \label{eqn-geo-bdy-ratio}
D_h(z,w) \geq \ep^A \sup_{u,v\in B_{\ep^{-  M} \BB r}(0) } D_h(u,v) .
\eqe 
\end{lem}
\begin{proof}
We will prove a lower bound for the left side of~\eqref{eqn-geo-bdy-ratio} (see~\eqref{eqn-geo-bdy-lower}) and an upper bound for the right side of~\eqref{eqn-geo-bdy-ratio} (see~\eqref{eqn-geo-bdy-upper}), then compare them.

By Proposition~\ref{prop-two-set-dist} and a union bound, it holds with superpolynomially high probability as $\ep\rta 0$ that
\eqb \label{eqn-geo-bdy-across}
D_h(\bdy B_{\ep\BB r/4}(x) ,\bdy B_{\ep \BB r/2}(x)) \geq \ep \frk c_{\ep \BB r} e^{\xi h_{\ep \BB r}(x)} ,\quad\forall x \in B_{\ep^{-M}\BB r}(0)\cap \left( \frac{\ep \BB r}{8} \BB Z^2 \right) .
\eqe 
The circle averages $h_{\ep\BB r}(x) - h_{\BB r}(0)$ for $x\in B_{\ep^{-M}\BB r}(0)$ are Gaussian with variance at most $(M+1)\log\ep^{-1}$. By the Gaussian tail bound and a union bound, if we choose $A_0 =A_0(M)$ to be sufficiently large, then it holds with probability $1-O_\ep(\ep^M)$ that
\eqb \label{eqn-geo-bdy-circle-avg}
| h_{\ep \BB r}(x) - h_{\BB r}(0)| \leq A_0\log\ep^{-1} \quad\forall x \in B_{\ep^{-M}\BB r}(0)\cap \left( \frac{\ep \BB r}{8} \BB Z^2 \right)  .
\eqe
By Theorem~\ref{thm-metric-scaling},  
\eqb \label{eqn-geo-bdy-const}
\frk c_{\ep\BB r} = \ep^{\xi Q + o_\ep(1)} \frk c_{\BB r} .
\eqe
If $z,w\in B_{\ep^{-M}\BB r}(0)$ with $|z-w| \geq \ep \BB r$, then any path from $z$ to $w$ must cross between the inner and outer boundaries of an annulus of the form $B_{\ep\BB r/2}(x)  \setminus B_{\ep \BB r/4}(x) $ for some $x \in B_{\ep^{-M}\BB r}(0)\cap \left( \frac{\ep \BB r}{8} \BB Z^2 \right)$.
Combining this last observation with~\eqref{eqn-geo-bdy-across} shows that with superpolynomially high probability as $\ep\rta 0$, $D_h(z,w)$ is at least the right side of~\eqref{eqn-geo-bdy-across} for each such $z,w$. 
We then apply~\eqref{eqn-geo-bdy-circle-avg} and~\eqref{eqn-geo-bdy-const} to lower-bound the right side of~\eqref{eqn-geo-bdy-across}. This shows that with probability $1-O_\ep(\ep^M)$,  
\eqb \label{eqn-geo-bdy-lower}
D_h(z,w) \geq \ep^{\xi A_0  + \xi Q + 1 + o_\ep(1)} \frk c_{\BB r} e^{\xi h_{\BB r}(0)}, \quad\forall z,w\in B_{\ep^{-M}\BB r}(0) \:\text{with}\: |z-w|\geq \ep\BB r .
\eqe

By Proposition~\ref{prop-diam-moment}, 
\eqb \label{eqn-geo-bdy-moment}
\BB E\left[\frk c_{\ep^{-M}\BB r}^{-1} e^{-\xi h_{\ep^{-M}\BB r}(0)} \sup_{u,v\in B_{\ep^{-  M} \BB r}(0) } D_h(u,v) \right] \preceq 1 ,
\eqe
with the implicit constant uniform over all $\BB r > 0$ and $\ep \in (0,1)$. 
By Theorem~\ref{thm-metric-scaling}, $\frk c_{\ep^{-M}\BB r} = \ep^{-\xi Q M + o_\ep(1)} \frk c_{\BB r}$.
By the Gaussian tail bound, we can find $A_1 = A_1(M) > 0$ such that with probability $1-O_\ep(\ep^M)$, we have $|h_{\ep^{-M}\BB r}(0) - h_{\BB r}(0)| \leq A_0 \log\ep^{-1}$.
Combining these estimates with~\eqref{eqn-geo-bdy-moment} and Markov's inequality shows that with probability $1-O_\ep(\ep^M)$,
\eqb \label{eqn-geo-bdy-upper}
 \sup_{u,v\in B_{\ep^{-  M} \BB r}(0) } D_h(u,v) \leq \ep^{-\xi A_1 -\xi Q M - M  + o_\ep(1)} \frk c_{\BB r} e^{\xi h_{\BB r}(0)} .
\eqe
Combining~\eqref{eqn-geo-bdy-lower} and~\eqref{eqn-geo-bdy-upper} gives~\eqref{eqn-geo-bdy-ratio} for any choice of $A> \xi A_1 +\xi Q M + M +\xi A_0  + \xi Q + 1$. 
\end{proof}

\begin{proof}[Proof of Proposition~\ref{prop-geo-bdy}] 
\noindent\textit{Step 1: defining a regularity event.}
For $\wt M >0$, $\zeta\in (0,1)$, $C > 1$, and $A > 1$, let $G_{\BB r}^\ep = G_{\BB r}^\ep(\wt M,\zeta,C,A)$ be the event that the following is true. 
\begin{enumerate}
\item The ball $B_{\ep^{-\wt M} \BB r}(0)$ can be covered by $C$-good Euclidean balls with radii in $[2\ep \BB r,\ep^{1-\zeta} \BB r]$. 
\item For each $z,w\in B_{\ep^{-\wt M} \BB r}(0)$ with $|z-w| \geq \ep \BB r$, 
\eqb \label{eqn-holder-cont-bdy}
D_h(z,w) \geq \ep^A \sup_{u,v\in B_{\ep^{-\wt M} \BB r}(0) } D_h(u,v) .
\eqe 
\end{enumerate}
By Lemmas~\ref{lem-geo-bdy-annuli} and~\ref{lem-geo-bdy-ratio}, for any $\wt M >0$ and $\zeta\in(0,1)$ we can find $C,A > 1$ for which 
\eqb
\BB P[G_{\BB r}^\ep] \geq 1 - O_\ep(\ep^{\wt M} ),\quad \text{uniformly over all $\BB r > 0$} .
\eqe
Henceforth assume that $G_{\BB r}^\ep$ occurs for such a choice of $C,A$ and that $\wt M > M$.  
\medskip

\noindent\textit{Step 2: reducing to a bound for the number of excursions of a geodesic.}
Let $s > 0$ such that $\mcl B_s(0 ; D_h) \subset B_{\ep^{-M} \BB r}(0)$ and let $P$ be a $D_h$-geodesic from 0 to a point outside of $\mcl B_s(0;D_h)$. 
Let $\tau_0 = s$ and inductively for $k\in \BB N$ let $\tau_k$ be the first time $t$ after the exit time of $P$ from $B_{4\ep^{1-\zeta} \BB r}(P(\tau_{k-1}))$ for which $P(t) \in B_{\ep \BB r}(\bdy\mcl B_s )$, or $\tau_k = \infty$ if no such time exists. 
Let $K$ be the smallest $k\in \BB N$ for which $\tau_k = \infty$. 

We claim that there exists a constant $c > 0 $ depending on $C,A$ such that $K\le c\log\ep^{-1}$ on $G_{\BB r}^\ep$.
If this is the case, then $P \cap B_{\ep \BB r}(\bdy\mcl B_s )$ can be covered by at most $c \log\ep^{-1} $ Euclidean balls of radius $4\ep^{1-\zeta} \BB r$. 
This means that $\op{area}\left( B_{\ep \BB r}(P) \cap B_{\ep \BB r}(\bdy\mcl B_s(0;D_h) \right) \leq 4\pi \ep^{2-2\zeta + o_\ep(1)} \BB r^2$. Choosing $\zeta < 1/(2M)$ and sending $\wt M \rta \infty$ then concludes the proof.
Hence we only need to prove a logarithmic upper bound for $K$ assuming that $G_{\BB r}^\ep$ occurs.  
\medskip

\noindent\textit{Step 3: bounding excursions using $C$-good balls.}
For $k\in [0,K-1]_{\BB Z}$, we can find a $C$-good Euclidean ball $B_k$ with radius in $[\ep \BB r, \ep^{1-\zeta}\BB r]$ which contains $P(\tau_k)$. 
Write $2B_k$ for the Euclidean ball with the same center as $B_k$ and twice the radius of $B_k$. 
Let $\sigma_k$ be the first time after $\tau_k$ at which $P$ exits $2B_k$. 
The time $\sigma_k$ is smaller than the exit time of $P$ from $B_{4\ep^{1-\zeta} \BB r}(P(\tau_{k }))$. 
Consequently, the definition of the $\tau_k$'s shows that $\sigma_k \in [\tau_k , \tau_{k+1}]$ for each $k\in [0,K]_{\BB Z}$. 
  
Since $P$ is a $D_h$-geodesic and $P$ crosses the annulus $(2B_k)\setminus B_k$ between times $\tau_k$ and $\sigma_k$, 
\eqb \label{eqn-C-good-cross}
\sigma_k - \tau_k \geq D_h(\bdy B_k , \bdy (2B_k) ) . 
\eqe
We now argue that 
\eqb \label{eqn-tau-upper}
\tau_k \leq s  + C D_h(\bdy B_k ,\bdy(2B_k))  .
\eqe
Indeed, since $B_k$ intersects $B_{\ep \BB r}(\bdy\mcl B_s(0;D_h))$ and has radius at least $2\ep \BB r$, it follows that $B_k$ intersects $\bdy\mcl B_s(0;D_h)$.  
Let $z\in \bdy\mcl B_s(0;D_h)$ and let $t\in [\tau_k , \sigma_k]$ such that $P(t) \in \bdy B_k$ (such a $t$ exists by the definition of $\sigma_k$). 
By the definition of a $C$-good ball, the $D_h$-diameter of $\bdy B_k$ is at most $C D_h(\bdy B_k , \bdy (2B_k) )$.
Hence
\alb
\tau_k\leq t\leq D_h(0,z)+D_h(z,P(t)) \leq s+ C D_h(\bdy B_k, \bdy (2B_k)) ,
\ale
which is~\eqref{eqn-tau-upper}. 

By~\eqref{eqn-C-good-cross} and~\eqref{eqn-tau-upper} and the fact that the intervals $[\tau_k , \sigma_k] \subset [s,\infty)$ are disjoint, we get
\eqbn
\sum_{j=0}^{k-1} (\sigma_j - \tau_j ) \leq \tau_k - s \leq C (\sigma_k - \tau_k ) .
\eqen
This holds for each $k\in [0,K-1]_{\BB Z}$, from which we infer that
\eqb \label{eqn-C-good-dist}
\sigma_{K-1}  - \tau_{K-1} \geq C^{-1} (1+C^{-1})^K (\sigma_0 -\tau_0 ) .
\eqe 

By the definition of $\sigma_0$, we have $|P(\sigma_0) - P(\tau_0)| = \ep \BB r$. 
Moreover, since $P(\tau_{K-1}) \in B_{\ep\BB r}(B_s(0;D_h))$, $B_s(0;D_h) \subset B_{\ep^{-M} \BB r}(0)$, and $\wt M > M$, we have $P(\sigma_{K-1}) , P(\tau_{K-1}) \in B_{\ep^{-\wt M} \BB r}(0)$. 
By~\eqref{eqn-holder-cont-bdy} in the definition of $G_{\BB r}^\ep$, it follows that
\eqb
\sigma_0 -\tau_0  \geq \ep^A  (\sigma_{K-1}  - \tau_{K-1}) .
\eqe
Combining this with~\eqref{eqn-C-good-dist} shows that $C^{-1} (1+C^{-1})^K  \leq \ep^{-A}$, which means that $K \leq \frac{A}{\log(1+C^{-1})} \log\ep^{-1} + O_\ep(1)$, as required.  
\end{proof}

\bibliography{cibiblong,cibib}
\bibliographystyle{hmralphaabbrv}

\subsubsection*{Author's addresses}

\medskip

\noindent J.\ Dub\'edat, Columbia University, New York, USA, 
\textit{Email:} dubedat@math.columbia.edu

\medskip

\noindent H.\ Falconet, Columbia University, New York, USA, 
\textit{Email:} hugo.falconet@columbia.edu

\medskip

\noindent E.\ Gwynne, University of Cambridge, Cambridge, UK,
\textit{Email:} eg558@cam.ac.uk

\medskip

\noindent J.\ Pfeffer, Massachusetts Institute of Technology, Cambridge, USA,
\textit{Email:} pfeffer@mit.edu

\medskip
 
\noindent X.\ Sun, Columbia University, New York, USA, 
\textit{Email:} xinsun@math.columbia.edu

\end{document}